\documentclass[preprint,12pt]{elsarticle}

\usepackage{amssymb,amsmath,graphicx,subcaption,mathrsfs,multirow,makecell,float}
\usepackage[colorlinks=true,citecolor=blue,urlcolor=blue]{hyperref}
\usepackage[table]{xcolor}

\usepackage{amsthm}

\newtheorem{ass}{Assumption}
\newtheorem{remark}{Remark}
\newtheorem{corollary}{Corollary}

\newtheorem{theorem}{Theorem}[section]
\newtheorem{lemma}[theorem]{Lemma}

\newtheorem{example}{Example}

\makeatletter
\newcommand*{\textlabel}[2]{%
  \edef\@currentlabel{#1}
  \phantomsection
  #1\label{#2}
}
\makeatother

\makeatletter
\def\ps@pprintTitle{%
  \let\@oddhead\@empty
  \let\@evenhead\@empty
  \let\@oddfoot\@empty
  \let\@evenfoot\@empty}
\makeatother
\begin{document}

\begin{frontmatter}

\title{Prokhorov Metric Convergence of the Partial Sum Process for Reconstructed Functional Data}

\author[CSU]{Tim Kutta\corref{cor1}}
\ead{tim.kutta@math.au.dk}
\author[CSU]{Piotr Kokoszka}
\ead{piotr.kokoszka@colostate.edu}

\cortext[cor1]{Corresponding author.}
\address[CSU]{Colorado State University, Department of Statistics}

\begin{abstract}
Motivated by applications in functional data analysis, we study the partial sum process of sparsely observed, random functions. A key novelty of our analysis are bounds for the distributional distance between the limit Brownian motion and the entire partial sum process in the function space. To measure the distance between distributions, we employ the Prokhorov and Wasserstein metrics. We show that these bounds have important probabilistic implications, including strong invariance principles and new couplings between the partial sums and their Gaussian limits. Our results are formulated for weakly dependent, nonstationary time series in the Banach space of $d$-dimensional, continuous functions. Mathematically, our approach rests on a new, two-step proof strategy: First, using entropy bounds from empirical process theory, we replace the function-valued partial sum process by a high-dimensional discretization. Second, using Gaussian approximations for weakly dependent, high-dimensional vectors, we obtain bounds on the distance. As a statistical application of our coupling results, we validate an open-ended monitoring scheme for sparse functional data. Existing probabilistic tools were not appropriate for this task.
\end{abstract}

\begin{keyword}
coupling \sep functional time series \sep invariance principle \sep monitoring \sep sparse functional data \sep weak dependence
\MSC[2020] 62M10 \sep 62R10
\end{keyword}

\end{frontmatter}

\section{Introduction} \label{s:lit_inv}
Invariance principles have many statistical applications.
In the  framework of functional data analysis (FDA), they have
been particularly useful   in change point analysis
(e.g. \cite{aston:kirch:2012,gromenko:kokoszka:reimherr:2017,aue:rice:sonmez:2018}),
self-normalized inference  (e.g. \cite{kim:zhao:shao:2015},
\cite{dette:kokot:volgushev:2020}, \cite{dette:kutta}
and  goodness-of-fit and stationarity tests
(e.g.  \cite{bugni:2009,horvath:kokoszka:rice:2014,cuesta:2019}).

We  introduce a new approach to invariance principles for functional data.
First, we  provide finite sample bounds for the distance between the functional
partial sum process $P_N$ and its Gaussian limit $W$ in weak convergence
metrics. Our key theorems state that $d(P_N, W) = \mathcal{O}(N^{-\tau})$ for
a positive $\tau$, not merely that $P_N$ converges to $W$ in a suitable space.
We consider the Prokhorov and the Wasserstein metrics, both of
which imply strong coupling results for $(P_N,W)$. To the best of our
knowledge, such bounds do not exist for functional data and have only
recently been derived in some finite dimensional cases
(see, e.g. \cite{hafouta:2023}). Second, our theory
applies to  functional triangular arrays whose rows are only approximately
stationary. This more general setting is motivated by
applications to sparse functional data.
Third, in contrast to most results for  functional data,
our theory  holds for random functions defined on a
 growing and even unbounded domain.
Such results are useful, e.g., when considering time series of random densities.
Motivated by such applications,
our results are formulated for random functions in  the Banach space
of continuous functions, rather than a separable Hilbert space, as typical in
FDA. Fourth, our coupling results provide a new strategy to validate
open-ended monitoring schemes for functional data and even for nonstationary
random functions derived from sparse observations. This strategy is not confined to functional data and
represents a novel way to validate monitoring schemes even in the scalar case.
Yet, on function spaces it is particularly meaningful because
tools derived from  KMT approximations do not exist in infinite dimension.
We also investigate some probabilistic
implications of the coupling results, including a bounded law of the iterated
logarithm and a strong invariance principle for dependent functional time series.
Fifth, proof strategies for our main results, the bounds on weak
convergence metrics,
are new. We study the (function--valued) partial sum process $P_N$
on the space of bounded functions $L^\infty$. Here, we approximate
it by a discretization, which is essentially a step function in a high-dimensional
subspace. On this subspace, we exploit Gaussian approximations to obtain
the desired finite sample bounds.
Finally, even when considered as merely weak invariance principles
(without rates), our results make a useful contribution to the literature,
because they hold under very general assumptions regarding dependence,
stationarity and function spaces. We explain these aspects and their
relation to existing work in Section \ref{sec_lit}, and throughout the paper
once suitable definitions and results have been formulated.

The remainder of this paper is organized as follows.
In Section \ref{sec_lit},  we review related literature
on (particularly weak) invariance principles for functional data
and draw comparisons to our work. In Section \ref{sec_prelim},
we provide preliminaries for
our subsequent theory. Our main probabilistic results are gathered
in Section \ref{sec1}. We begin by developing a framework for sparse
functional data in Section \ref{sec_sparse}. In Section \ref{sec_inv},
we provide a finite sample bound on the
Prokhorov distance between the functional partial sum process and its
limit, which is our chief contribution.
In Section \ref{s:cons}, we  investigate
some mathematical consequences of this result. In Section \ref{sec_mon},
we present a new monitoring scheme for arrays of sparsely observed
functional data. Proofs of our main results and additional background
are provided in the online Supplementary Information.

\section{Discussion of related research} \label{sec_lit}
The relative usefulness of an invariance principle in Functional Data Analysis
depends on  assumptions regarding the function space, dependence
structure and robustness to departures from stationarity.
In the following, we discuss these criteria and review
how the related literature and this work address them. Since research
on invariance principles in function spaces
is very extensive, we have confined our discussion
to a selection of most closely related contributions.

\textbf{Function space} A fundamental decision in
FDA is in which space  the random functions are modeled.
The usual  choice is the Hilbert space $L^2([0,1])$,
chiefly due to  the availability of the  inner product
(e.g. \cite{bosq:2000}, \cite{HKbook}, \cite{hsing:eubank:2015}).
Up to an isomorphism,  it  is the unique separable
Hilbert space of infinite dimension.
Examples of weak invariance principles on separable Hilbert
spaces can be found in \cite{merlevede:2003},
\cite{berkes:horvath:rice:2013}, \cite{cuny:merlevede:2014}
or more recently in \cite{lu:wu:xiao:xu:2022}.
Hilbert spaces have the mathematical benefit of allowing
a close analogy to the multivariate setting - for instance
one can define and interpret principal components as
the directions of maximum variation. However, in  statistical applications
that involve  continuous functions, for example estimated densities,
the space   $L^2([0,1])$ is on the one hand too large because it
contains functions whose regularity is not restricted, on the other
hand it is too small because it contains only  functions with compact support.
Moreover, there are  a number of relevant functionals
that are continuous w.r.t. to the sup-norm but not w.r.t. the $L^2$-norm.
Examples include   point evaluations, the range and the integral of a function.
Invariance principles for continuous functions have been considered in
\cite{dette:kokot:aue} and in \cite{dette:kokot:2022}.
Besides these works that are tailored to functional data,
there also exists a literature on invariance principles for general separable
Banach spaces,
such as \cite{kuelbs:1973}, \cite{dehling:1983},
\cite{samur:1987}. While these results cover a variety of different
spaces, their application is usually challenging
(see \cite{dette:kokot:aue}). The reason  is that invariance
principles for Banach spaces involve high-level conditions on the tightness
of the partial sum process. These conditions have to be validated using
the specific geometry of the Banach space of interest. Therefore,
when focusing on FDA, results that are tailored to a specific space,
such as $\mathcal{C}([0,1], \mathbb{R})$,  are  preferable,
because simpler, low-level conditions can be formulated.
For this reason, we  consider in this work functional data
on spaces of continuous functions. We expand the existing literature,
by considering not only functional data in $\mathcal{C}(I, \mathbb{R})$
for a fixed, compact $I$, but also for an expanding domain and even
for $I=\mathbb{R}$, a case that is of interest e.g. in the study of
random probability densities which are often not compactly supported.
Both of these latter cases, to the best of our knowledge,
have not been considered so far.

\textbf{Dependence} When studying functional time series,
there are a number of ways to model dependence between  functions.
Examples include fast decaying mixing coefficients
(\cite{kuelbs:philipp:1980,dehling:1983,samur:1987}),
weak association of projections (\cite{burton:dabrowski:dehling:1986})
or approximability by a weakly dependent time series
(\cite{berkes:horvath:rice:2013}). In the case of a
separable Hilbert space, dependence can usually be stronger than for general
Banach spaces and a greater variety of dependence concepts have
been explored (see, e.g.
\cite{merlevede:2003}, \cite{cuny:merlevede:2014},
\cite{berkes:horvath:rice:2013}). For Banach spaces, most
results rely on either $\phi$-mixing
or $\beta$-mixing (see e.g. \cite{dehling:1983,samur:1987}).
However, in this work, we allow for stronger dependence than
usually considered in Banach spaces and quantify dependence
by polynomially decaying $\alpha$-mixing coefficients.
Recall that both $\phi-$ and $\beta-$mixing imply
$\alpha$-mixing (see \cite{bradley:2007}). As we discuss in Remark \ref{rem_mixing}, we show coupling results under $\alpha$-mixing, which require very different proof techniques than in the case of $\phi$- or $\beta$-mixing.  For $\phi$- and $\beta$-mixing, general coupling results exist even for abstract Polish spaces, which do not hold even for separable Hilbert spaces in the case of $\alpha$-mixing.

\textbf{Triangular arrays, sparse FDA and stationarity}
Our results are motivated by and applied to
functional data that may be sparsely observed. Statistically,
this means that a user wants to make inference for a sample of
functions, $X_1,...,X_N$, but has only access to imperfect
reconstructions $\widehat{X}_1,\ldots,\widehat{X}_N$.
A way to include sparsity is to formulate convergence results for a triangular
array of random functions, $\widehat{X}_1,\ldots,\widehat{X}_N$.
These reconstructions may change with $N$, reflecting state of the
art approaches that pool information from the whole sample.
While most related contributions consider time series of random functions,
 \cite{kuelbs:1973} and   \cite{samur:1987}
prove invariance principles for triangular arrays. Stationarity within each row is
however assumed. To model sparse functional data,
even a notion of weak stationarity can be unrealistic. Suppose that
the functions $X_1,...,X_N$ are strictly stationary,
but their estimates $\widehat{X}_1,\ldots,\widehat{X}_N$ are
based on imbalanced samples. In this case, the sparse
estimators might not even be weakly stationary and none
of the named results would be applicable. For this reason,
we develop theory for data that is only ``approximately,
weakly stationary", which, in particular, covers  estimators
computed from imbalanced
samples. We present a new and general framework that covers
commonly used  estimators, cf. Section \ref{sec1}, and
 most existing reconstruction methods
(e.g.  \cite{yao:muller:wang:2005JASA,zhang:wang:2016}).

\begin{table}[t]
    \centering
    \small
    \begin{tabular}{|p{1cm}|p{3.05cm}|p{2cm}|c|c|}
        \hline
        Paper & Function Space & Dependence & \makecell{$ \,\,\quad$ Data $ \,\,\quad$\\ structure } & Stationarity \\
        \hline\hline
        \cite{kuelbs:1973} & \cellcolor{yellow!25}\makecell{$ \,\,\quad$ separable $ \,\,\,\,\,\quad$\\ Banach space } & \cellcolor{red!25} independent & \cellcolor{green!25}\makecell{$ \,\,$ Triangular $ \,\,$\\ arrays }& \cellcolor{red!25}\makecell{Strict\\ Stationarity  $\,\,\,\,\,$} \\
        \hline
        \cite{kuelbs:philipp:1980} & \cellcolor{yellow!25}\makecell{$ \,\,\quad$ separable $ \,\,\,\,\,\quad$\\ Banach space } & \cellcolor{yellow!25} $\phi$-mixing & \cellcolor{red!25}Time Series & \cellcolor{yellow!25}\makecell{Weak \\ Stationarity  $\,\,\,\,\,$} \\
        \hline
        \cite{dehling:1983}& \cellcolor{yellow!25}\makecell{$ \,\,\quad$ separable $ \,\,\,\,\,\quad$\\ Banach space } & \cellcolor{yellow!25} \makecell{$\phi-$ and $\quad\,\,\,\,\,\,\,\,\,\,$\\ $\beta-$mixing $\quad\,\quad$ } & \cellcolor{red!25}Time Series & \cellcolor{yellow!25}\makecell{Weak\\ Stationarity $\,\,\,\,\,$} \\
        \hline
      \cite{burton:dabrowski:dehling:1986} & \cellcolor{red!25}\makecell{$ \,\,\quad$ separable $ \,\,\,\,\,\quad$\\  Hilbert space } & \cellcolor{yellow!25} weak assoc. & \cellcolor{red!25}Time series & \cellcolor{red!25} \makecell{Strict\\ Stationarity $\,\,\,$} \\
        \hline
         \cite{samur:1987}  & \cellcolor{yellow!25}\makecell{$ \,\,\quad$ separable $ \,\,\,\,\,\quad$\\ Banach space } & \cellcolor{yellow!25} $\phi-$mixing & \cellcolor{green!25}\makecell{$ \,\,$ Triangular $ \,\,$\\ arrays }& \cellcolor{yellow!25} \makecell{Weak\\ Stationarity  $\,\,\,$} \\
        \hline
         \cite{merlevede:2003} & \cellcolor{red!25}\makecell{$ \,\,\quad$ separable $ \,\,\,\,\,\quad$\\ Hilbert space } & \cellcolor{green!25} $\alpha-$mixing & \cellcolor{red!25}Time Series& \cellcolor{red!25} \makecell{Strict\\ Stationarity  $\,\,\,$} \\
        \hline
       \cite{berkes:horvath:rice:2013} & \cellcolor{red!25}\makecell{$ \,\,\quad$ separable $ \,\,\,\,\,\quad$\\ Hilbert space } & \cellcolor{green!25} \makecell{$m-$approx.\\ sequences} & \cellcolor{red!25} Time Series & \cellcolor{red!25} \makecell{Strict\\ Stationarity  $\,\,\,$} \\
        \hline
        \cite{dette:kokot:aue}  & \cellcolor{green!25}\makecell{$ \,\,\quad$ $\mathcal{C}(I, \mathbb{R})$ with $\,\,\,$\\ $I$ compact $\,\,$} & \cellcolor{yellow!25} $\phi-$mixing & \cellcolor{red!25} Time Series & \cellcolor{red!25} \makecell{Strict\\ Stationarity  $\,\,\,$} \\
        \hline
        \cite{lu:wu:xiao:xu:2022} & \cellcolor{red!25}\makecell{$ \,\,\quad$ separable $ \,\,\,\,\,\quad$\\ Hilbert space }  & \cellcolor{yellow!25} $\beta-$mixing & \cellcolor{red!25} Time Series & \cellcolor{red!25} \makecell{Strict\\ Stationarity  $\,\,\,$} \\
        \hline
        This work & \cellcolor{green!25}\makecell{$ $ $\mathcal{C}_0(I, \mathbb{R}^d)$ with\\ $I$ compact,\\
          $I$ expanding,\\ $ I=\mathbb{R}$  } & \cellcolor{green!25} $\alpha-$mixing &\cellcolor{green!25}\makecell{$ \,\,$ Triangular $ \,\,$\\ arrays }& \cellcolor{green!25}  \,\makecell{Approximate\\ weak \\
           Stationarity }\, \\
        \hline
    \end{tabular}
    \caption{Overview of some invariance principles applicable to functional data.}
    \label{tab:1}
\end{table}

\textbf{Comparison to existing work} In Table \ref{tab:1},
we provide an overview of some important papers  on invariance
principles for functional data. Evidently, such an overview is
bound to omit certain details and we only want to point out some of them.
First, a number of the cited works, such as
\cite{kuelbs:philipp:1980,dehling:1983,lu:wu:xiao:xu:2022}
-as well as this paper- contain strong invariance principles that
guarantee almost sure convergence of the partial sum process
on a suitably altered probability space. Approximations
in $p$-th moments are considered in \cite{cuny:merlevede:2014}.
Second, the dependence conditions are not always comparable.
For instance, \cite{dette:kokot:aue} and \cite{lu:wu:xiao:xu:2022}
employ exponentially decaying mixing coefficients, whereas
\cite{samur:1987}
and \cite{dehling:1983} use the same notions of mixing,
but with polynomially decaying mixing coefficients.

\cite{merlevede:2003} employs a non-standard definition of
$\alpha$-mixing that is more general than the usual definition.
Finally, to keep our discussion succinct we have left out a number
of related works. For instance, there exists an entire literature
on almost sure invariance principles with slower approximation
rates of order $o(\sqrt{n\log(\log(n))})$ that have uses in the
study of the limit superior of sequences (see \cite{dehling:philipp:1982},
\cite{dedecker:merlevede:2010}). The recent work of
\cite{mies:steland:2023} is also related as it studies Gaussian
approximations of the partial sum process of high dimensional scalar
observations which need not form a stationary
sequence (the dimension is finite,  but tends to infinity with the
sample size).  Since these results are not directly comparable to
the results in this work, we do not discuss them further.

\section{Preliminaries} \label{sec_prelim}

\subsection{Notation}\label{sec_not} $ $ 

\textbf{Interval notation} Throughout this work, we denote by $I, I_1, I_2$
non-empty subintervals of the real line that are either equal to $\mathbb{R}$ or compact. We will always assume that $I_2 \subset I_1$.

\textbf{Vector space notation}
We consider the vector space $(\mathbb{R}^d, |\cdot|)$ for some $d \ge 1$. For any dimension, $|\cdot|$ denotes the maximum norm. For a non-empty index set $\mathcal{I} \subset \{1,...,d\}$ and a vector $v \in \mathbb{R}^d$, we define the subvector with components in $\mathcal{I} $ as
\begin{equation} \label{e:subv}
v^{(\mathcal{I} )}:=(v[\ell])_{\ell \in \mathcal{I} }.
\end{equation}

\textbf{Space of continuous functions} By $\mathcal{C}(I, \mathbb{R}^d)$ we denote the space of bounded continuous functions $f:I \to \mathbb{R}^d$. Notice that any such function can be interpreted as a vector of continuous functions $f=(f[1],...,f[d])$ with $f[\ell] \in \mathcal{C}(I, \mathbb{R})$. A natural distance on $\mathcal{C}(I, \mathbb{R}^d)$ is given by the sup-norm
\[
\|f\|:= \sup_{u \in I} |f(u)|=\sup_{u \in I} \max_{1 \le \ell \le d}|f[\ell](u)|.
\]
For the special case where $I = \mathbb{R}$, we define the closed subspace $\mathcal{C}_0(I, \mathbb{R}^d)$ of $\mathcal{C}(I, \mathbb{R}^d)$, which consists of functions that vanish at $\pm \infty$, i.e.
\[
\lim_{y \to \infty}\sup_{|u|>y}|f(u)|=0.
\]
We often present unified results for the space of continuous functions
on a compact interval $I$  and for the space of functions
that vanish at infinity with $I=\mathbb{R}$,
and hence generically write $\mathcal{C}_0(I, \mathbb{R}^d)$.
We recall that, equipped with the sup-norm, $\mathcal{C}_0(I, \mathbb{R}^d)$
is a separable  Banach space.

\textbf{Notation for restrictions} If $I_2 \subset I_1$, we can 
naturally understand any function 
$f \in \mathcal{C}_0(I_1, \mathbb{R}^d)$ as a function in 
$\mathcal{C}_0(I_2, \mathbb{R}^d)$, by considering the restriction $\{(u,f(u)): u \in I_2\}$. We will often consider the restriction $f$ to $I_2$ and for parsimony of notation also refer to it by $f$. If $g \in \mathcal{C}_0(I_2, \mathbb{R}^d)$ we will thus understand the difference $g-f$ as the difference of $g$ and the restriction of $f$.

\subsection{Probability on function spaces} \label{sec23} $ $

\textbf{Random functions in $\mathcal{C}_0$} 
Let $(\Omega, \mathcal{A}, \mathbb{P})$ be a probability space.
We call a measurable map $X: \Omega \to \mathcal{C}_0(I, \mathbb{R}^d)$
a random function in $\mathcal{C}_0(I, \mathbb{R}^d)$.
Recall that the measurability of  $X$ is equivalent to the measurability
of all evaluations $X(u)$ for $u \in I$,
(see \cite{billingsley:1968}, p. 84).
Assuming $\mathbb{E}\|X\|<\infty$, we define its mean function $\mu \in \mathcal{C}_0(I, \mathbb{R}^d)$ by the pointwise identity
\[
\mathbb{E}X[\ell](u)=\mu[\ell](u), \ \ \ u \in I, \ \ 1 \le \ell \le d.
\]
Moreover, if $\mathbb{E}\|X\|^2<\infty$,
we  define its continuous covariance kernel $c$ for all $u,v \in I, \, 1 \le \ell_1, \ell_2 \le d$ by
\[
c[\ell_1, \ell_2](u,v):= \mathbb{E}[(X[\ell_1](u)-\mu[\ell_1](u))(X[\ell_2](v)-\mu[\ell_2](v))].
\]
For details on continuous random functions,
 we refer the reader to \cite{janson:kaijser:2015}.

\textbf{Bounded functions} A second important function space in this work is the space of bounded functions. Let $\tilde I$ be an index set that is not necessarily an interval.  We define the space of bounded functions
\[
L^\infty(\tilde I, \mathbb{R}^d):=\{f: \tilde I \to \mathbb{R}^d: \|f\|<\infty\},
\]
where $\|f\|:= \sup_{x \in \tilde I}|f(x)|$.
If $\tilde I \subset I$, $L^\infty(\tilde I, \mathbb{R}^d)$
can be used to describe discretizations of functions in
$\mathcal{C}_0(I, \mathbb{R}^d)$.
A further important use will be
$L^\infty([0,1] \times I, \mathbb{R}^d)$, where a partial sum process (which is not continuous) of functions in $\mathcal{C}_0(I, \mathbb{R}^d)$ can be defined. Another process that takes values in $L^\infty([0,1] \times I, \mathbb{R}^d)$ is the Brownian motion with sample paths in $\mathcal{C}_0(I, \mathbb{R}^d)$.

\textbf{The Banach-valued Brownian motion} In order to describe  our limiting distributions,  we employ the Brownian motion in  $\mathcal{C}_0(I, \mathbb{R}^d)$. This Brownian motion $\{W(x): x \in [0,1] \}$ can be understood as a stochastic process with continuous sample paths in $\mathcal{C}_0(I, \mathbb{R}^d)$ that has stationary, independent and normally distributed increments. It is characterized by its mean $\mathbb{E}W(1) \in \mathcal{C}_0(I, \mathbb{R}^d)$ (in our case, always the constant zero function) and the covariance of $W(1)$. We can also understand $W$ as a random variable in the space $L^\infty([0,1] \times I, \mathbb{R}^d)$ on which we will formulate our weak invariance principles.
It is possible to prove existence and uniqueness of the Brownian motion in Banach spaces similarly as for the finite dimensional Brownian motion. For details we refer to \cite{kuelbs:1973} and references therein.

\textbf{The Prokhorov metric}
Let  $(\mathcal{M}, d_\mathcal{M})$ be a generic metric space.
We define for a set $A \subset \mathcal{M}$ and $\chi>0$ the $\chi$-environment
\[
A^\chi:=\{x \in \mathcal{M}|\,\, \exists a \in A: d_\mathcal{M}(x,a) <\chi\}.
\]
We call a set $A$ measurable if it is a Borel set in $\mathcal{M}$ w.r.t. to $d_\mathcal{M}$.
Now we can define the distance of two probability measures $\mathbb{P}_1, \mathbb{P}_2$ as
\begin{align} \label{e:def_Prokhorov}
\pi(\mathbb{P}_1,\mathbb{P}_2) =&\inf \{\chi > 0 : \mathbb{P}_1(A) \leq \mathbb{P}_2(A^\chi) + \chi, \,\, \mathbb{P}_2(A) \leq \mathbb{P}_1(A^\chi) + \chi \\
& \qquad\qquad\qquad\qquad\qquad\qquad\qquad\quad\text{ for all measurable } A\}.\nonumber 
\end{align}
Notice that in this definition, the value of $\pi$ depends
on the metric $d_\mathcal{M}$ and even two metrics
that create the same topology might lead to different values of $\pi$.
In order to clarify which metric is used, we will subsequently denote by
$\pi_\infty$ the Prokhorov metric w.r.t. the supremum norm. Finally,
we point out that while the Prokhorov metric is defined on any metric
space, it metricizes weak convergence on separable, complete metric spaces.
We will come back to this detail, when discussing coupling results
for stochastic processes in Section \ref{s:cons}.

\textbf{The Wasserstein metric} A second metric on probability
measures studied in this work is the Wasserstein metric.
For a general introduction to this metric and its properties we refer
 to \cite{panaretos:zemel:2020}.
Consider a complete, separable metric space
$(\mathcal{M}, d_\mathcal{M})$ with two probability measures
$\mathbb{P}_1, \mathbb{P}_2$ defined on it and
let $\Theta(\mathbb{P}_1,\mathbb{P}_2)$ be the set of
all couplings of $\mathbb{P}_1$ and $ \mathbb{P}_2$.
More precisely, $\Theta(\mathbb{P}_1,\mathbb{P}_2)$
is the set of all probability measures $\theta$ on the product
space $\mathcal{M} \times \mathcal{M}$ with marginals
$\mathbb{P}_1, \mathbb{P}_2$. Then, the Wasserstein
distance of $\mathbb{P}_1, \mathbb{P}_2$ is for any $q \ge 1$ defined as
\begin{equation} \label{e:wm}
\mathcal{W}_q(\mathbb{P}_1,\mathbb{P}_2):= \inf_{\theta \in \Theta(\mathbb{P}_1,\mathbb{P}_2)}
 \big\{\mathbb{E}_{(x,y) \sim \theta}\big[ d_{\mathcal{M}}^q(x,y)\big]\big\}^{1/q}.
\end{equation}
As in the case of the Prokhorov metric, we define for two random variables $X \sim \mathbb{P}_1, Y \sim \mathbb{P}_2$
\[
\mathcal{W}_q(X,Y):= \mathcal{W}_q(\mathbb{P}_1,\mathbb{P}_2).
\]
We point out that (following   \cite{gibbs:su:2002})
the Wasserstein metric upper bounds the squared Prokhorov metric:
\[
\pi(\mathbb{P}_1,\mathbb{P}_2)^2 \le \mathcal{W}_q(\mathbb{P}_1,\mathbb{P}_2).
\]

\section{An invariance principle for sparse functional data} \label{sec1}

\subsection{A framework for sparse functional data } \label{sec_sparse}
The framework used in this paper is motivated by
applications to sparse functional data and the absence
of suitable probabilistic tools that could be used in several statistical
problems for such data. Recall the convention that $I_1, I_2$ denote
non-empty intervals that are each either compact
or equal to $\mathbb{R}$ with $I_2 \subset I_1$.
Let $(X_n)_{n \in \mathbb{N}}$ be a weakly stationary
time series of random functions in $\mathcal{C}_0(I_1, \mathbb{R}^d)$.
To model estimators of these functions, we consider an array
of  random variables $\{\delta_{n,M}: n, M \in \mathbb{N}\}$,
where for any $n$ the variable $ \delta_{n,M}$ takes values in a
metric space $\mathcal{S}_M$. Moreover, for any $M\in \mathbb{N}$
we define the deterministic, measurable map
\begin{align} \label{e:rest}
\mathscr{A}_M: \mathcal{C}_0(I_1, \mathbb{R}^d) \times \mathcal{S}_M \to \mathcal{C}_0(I_2, \mathbb{R}^d),
\end{align}
and therewith the \textit{sparse estimator} $X_{n,M}$ of $X_n$ as
\begin{align} \label{e:spest}
X_{n,M} := \mathscr{A}_M(X_n, \delta_{n,M}).
\end{align}
It is useful think of $X_{n,M}$ as an observable approximation to
the sparsely (incompletely) observed $X_n$ which may be defined
on a smaller subinterval $I_2$. To fix attention, one can think of
$X_n$ as a density on $I_1 = \mathbb{R}$ which is  estimated on
a subinterval $I_2$. Even if the sequence of random densities
is stationary, the estimators  $X_{n,M}$,  based on a finite number of
points for each $n$, may no longer form a stationary sequence.
More illustrations  are given in Example \ref{ex1}.
In general, we are often interested in the behavior of the
\textit{latent functions} $X_1,X_2,\ldots$, but we are only able
to construct the sparse estimators $X_{1,M}, X_{2,M},\ldots$,
where the random variables $\delta_{1,M}, \delta_{2,M},\ldots$
model \textit{distortions} of the ground truth.
The case of of completely observed  functions $X_1, X_2,...$ is
covered by our theory, by setting  $\mathscr{A}(X_n, \delta_{n,M})=X_n$.

\begin{example} \label{ex1} We discuss three examples of the
 above framework:

\textbf{i) Non-parametric smoothing} Suppose  that $I_1= \mathbb{R}$
and that $(X_n)_{n \in \mathbb{N}}$ is a time series in
$\mathcal{C}_0(\mathbb{R},\mathbb{R})$.
Assume that for each $n$, we observe
\[
y_{n,m}^{(M)} = X_n(u_{n,m}^{(M)})+\epsilon_{n,m}^{(M)},
\quad m=1,\ldots,M,
\]
where $(u_{n,m}^{(M)})_{n,m,M \in \mathbb{N}}$ are i.i.d. design points, distributed according to a
        continuous probability density $f:\mathbb{R} \to (0, \infty)$, and
        $(\epsilon_{n,m}^{(M)})_{n,m,M \in \mathbb{N}}$ are i.i.d. errors following a normal distribution $\mathcal{N}(0, \sigma^2)$. 
        We assume that the arrays  $(X_n)_{n \in \mathbb{N}}$, $(u_{n,m}^{(M)})_{n,m,M \in \mathbb{N}}$, $(\epsilon_{n,m}^{(M)})_{n,m,M \in \mathbb{N}}$ are independent of each other. Next, we define the distortion
        \[
        \delta_{n,M}:=(u_{n,m}^{(M)},\epsilon_{n,m}^{(M)})_{m=1,\ldots,M},
        \]
        which is a random variable on the space $\mathcal{S}_M:= \mathbb{R}^{2M}$.
        In principle, a variety of non-parametric estimators can be used to define $ \mathscr{A}_M$. For example, we can consider the Nadaraya–Watson estimator.
        Therefore, let $K$ be a univariate probability density function and denote for a bandwidth $h>0$,
        \[
        K_h(u) := \frac{1}{h} K\bigg(\frac{u}{h}\bigg).
        \]
        Then we define for any $n \in \mathbb{N}$ the sparse estimator
        \[
       X_{n,M}(u):= \mathscr{A}_M(X_n, \delta_{n,M})(u) := \frac{\sum_{m=1}^M K_{h}(u-u_{n,m}^{(M)})y_{n,m}^{(M)}}{\sum_{m=1}^M K_{h}(u-u_{n,m}^{(M)})}.
        \]
       Here, the bandwidth $h=h_{M,n}$ may depend on $M$ and $n$. Notice that then
        \[
        \mathscr{A}_M: \mathcal{C}_0(\mathbb{R}, \mathbb{R}) \times \mathbb{R}^{2M}\to \mathcal{C}(I_2, \mathbb{R}),
        \]
        where $I_2$ can be any compact interval over which the estimators $X_{n,M}$ are computed.

\textbf{ii) Observations on a grid} We consider a second setup
        with both intervals being compact $I_1=I_2=[0,1]$ and $(X_n)_{n \in \mathbb{N}}$ a time series in $\mathcal{C}([0,1],\mathbb{R})$. We assume that we observe these objects on a grid of size $M+1$, i.e. we have for each $n$ the data sample
        \[
        \{(m/M, X_n(m/M)): 0 \le m \le M\}.
        \]
        In this case, we can formally  define the deterministic distortion $\delta_{n,M} = (0, 1/M, 2/M,\ldots,1) \in \mathbb{R}^{M+1}$ and the approximation map
        \[
        \mathscr{A}_M: \mathcal{C}_0([0,1],\mathbb{R}) \times \mathbb{R}^{M+1} \to \mathcal{C}_0([0,1],\mathbb{R}),
        \]
        of a function $g$ as the linear interpolation of the points
        \[
        \{(m/M, g(m/M)): 0 \le m \le M\}.
        \]
        It is easy to show that $\mathscr{A}_M$ is measurable and hence the sparse estimators $X_{n,M}$ are well-defined.

\textbf{iii) Kernel density estimation} We consider the case where
$I_1=\mathbb{R}$  and  $(X_n)_{n \in \mathbb{N}}$
is a stationary time series of random probability densities in $\mathcal{C}_0(\mathbb{R},\mathbb{R})$. For a random sample size $D_{M,n}$ we now consider the i.i.d. observations
\begin{align} \label{e:sampling}
        u_{n,1},\ldots.,u_{n,D_{M,n}} \overset{i.i.d.}{\sim} X_n
\end{align}
for every $n \in \mathbb{N}$. To make this model mathematically precise,
we will assume that the points $(u_{n,d})_{n,d \in \mathbb{N}}$
are independent conditionally on the $(X_n)_{n \in \mathbb{N}}$.
The time series of random sample sizes $(D_{M,n})_{n \in \mathbb{N}}$
is independent of everything else and not necessarily stationary.
 This models  imbalanced observational studies, e.g. voluntary
viral tests,  where
the numbers of subjects may vary significantly from week to week.
The distortions are defined as
        \[
        \delta_{n,M} :=(D_{M,n}, u_{n,1},u_{n,2},\ldots,u_{n,D_{M,n}},0,0,\ldots) \in \ell^\infty(\mathbb{N}),
        \]
        in $\ell^\infty(\mathbb{N})$ the space of bounded sequences.
        We then define $X_{n,M}$ as the kernel density estimator
\[
\mathscr{A}(X_n, \delta_{n,M}):=X_{n,M}(u)
:=\frac{1}{D_{M,n} h}\sum_{d=1}^{D_{M,n}}
K\bigg( \frac{u-u_{n,d}}{h}\bigg), \ \ \  u\in I_2,
\]
where, as before, $K$ is a continuous probability density
and $h$ a bandwidth that may depend on $D_{M,n}$.
\end{example}

\begin{remark} \label{rem:dep} We make a few observations:
\begin{itemize}
        \item[i)] While the underlying time series $(X_n)_{n \in \mathbb{N}}$ is assumed to be  weakly stationary, this is not necessarily true of the derived, sparse estimators $(X_{n,M})_{n \in \mathbb{N}}$, since the distortions do not need to be stationary (see, e.g., Example \ref{ex1}, part iii)). However, if $X_{n,M} \approx X_{n}$ for sufficiently large $M$ the sparse estimators will  be “approximately stationary” (see also eq.  \eqref{e:apprst} below).
        \item[ii)] Given their representation in \eqref{e:spest},
        it is intuitively clear, that the sparse estimators are only
        as dependent as the original time series $(X_n)_{n \in \mathbb{N}}$
        and the time series of distortions $(\delta_{n,M})_{n \in \mathbb{N}}$.
        More rigorously, if we impose jointly on these time series
        $\phi-, \beta-$ or $\alpha-$mixing conditions,
        this mixing property will be inherited by the sparse estimators.
        This is one reason why it is more convenient to work under mixing
        conditions in our setting rather than with moment conditions imposed
        on Bernoulli shifts (see, e.g.,  \cite{wu:2007,hormann:kokoszka:2010}).
\end{itemize}
\end{remark}

\subsection{A bound on the Prokhorov distance} \label{sec_inv}

We are interested in deriving a weak invariance principle for an array
of $N$ sparse estimators $X_{1,M},\ldots,X_{N,M}$. Here,
the term ``sparse estimators" is to be understood  in the sense of the previous section and we notice that in particular the fully observed case where $X_{n,M}=X_n$ is covered by it.  To formulate our asymptotics, we will consider $N \to \infty$ and assume that simultaneously
$
M = M_N \to \infty $.  Hence, we study the partial sum process
\begin{align} \label{e:partials}
P_{N}(\lambda,u):= \frac{1}{\sqrt{N}} \sum_{i=1}^{\lfloor \lambda N \rfloor} \big\{X_{i,M_N}(u)- \mathbb{E}[X_{i,M_N}(u)]\big\}, \quad \lambda \in [0,1], \,\, u \in I_2
\end{align}
that takes values  in the space $L^\infty([0,1] \times I_2, \mathbb{R}^d)$.
The double asymptotic $N \to \infty, M_N \to \infty$ means
that as the number of functions $N$ gets larger, the quality
of the individual estimator $X_{i,M_N}$ increases.  This reflects
broadly used functional data reconstruction
approaches that utilize the whole functional sample
rather than each curve  separately, see \cite{chen:muller:2012}
who provide numerous references. 

We now formulate the mathematical assumptions for our results.
The notion of $\alpha$-mixing is well-known, but it is
reviewed in Section  \ref{sec_details} of Supplementary Information.

\begin{ass} $ $ \label{ass_1}
    \begin{itemize}
        \item[i)] (Dependence) The time series $(X_n)_{n \in \mathbb{N}}$ 
        is weakly stationary  in $\mathcal{C}_0(I_1, \mathbb{R}^d)$. 
        Moreover, for any $M \in \mathbb{N}$ the time series $(X_n, \delta_{n,M})_{n \in \mathbb{N}}$ is $\alpha$-mixing, with mixing coefficients denoted by $\alpha_{X,\delta,M}$. There exist two numbers $C_1, \nu>0$, such that for all $n \in \mathbb{N}$
        \[
        \sup_M \alpha_{X,\delta,M}(n) \le C_1 n^{-\nu}.
        \]
        \item[ii)] (Smoothness \& Moments)
        Let  $J >4$ and $\xi>0$ be constants that satisfy
        \[
        3+\frac{9}{J-4}<\nu, \qquad \frac{1}{2}<\xi.
        \]
        It holds for some $C_2>0$ that
        \[
        \sup_n \mathbb{E}\|X_n\|^{J}\le C_2, \quad \sup_n \sup_{u \neq v} \mathbb{E}\bigg(\frac{|X_n(u)-X_n(v)|}{|u-v|^\xi}\bigg)^{J} \le C_2.
        \]
        \item[iii)] (Approximation) There exist $C_3>0$,  $\gamma >0$ such that
        \begin{align*}
            & \sup_{n}\mathbb{E} \|X_n-X_{n,M_N}\|^{J} \le C_3 N^{-\gamma},\\
            &   \sup_{n,N}\sup_{u \neq v} \mathbb{E}\bigg(\frac{|X_{n,M_N}(u)-X_{n,M_N}(v)|}{|u-v|^\xi}\bigg)^{J} \le C_3.
        \end{align*}
    \end{itemize}
\end{ass}
We briefly discuss Conditions i)-iii).
Condition i) postulates weak stationarity of the original time series
and ensures that neither the random functions, nor the distortions,
are too strongly dependent. Mixing  conditions are well suited for our framework,
as explained in Remark \ref{rem:dep} (ii).
The assumption of the polynomial $\alpha$-mixing rate
is quite weak, as invariance principles in Banach spaces usually
require the more restrictive concepts of $\phi$- or $\beta$-mixing;
see Table \ref{tab:1}. Condition ii) ensures smoothness and moments
of the latent functions and is typical in the study of random continuous
functions (see, e.g., \cite{dette:kokot:aue}, \cite{berger:hermann:holzmann:2023} who use related conditions).
These conditions are later required to prove asymptotic equicontinuity
of the partial sum process. We notice that stronger moment assumptions
(larger $J$) permit stronger dependence (smaller $\nu$).
Condition iii) implies that the sparse estimators are uniformly
approximating the latent functions, ensuring that the partial sum
process of sparse estimators behaves asymptotically like the one of
latent functions. An important implication of this is that the sparse
estimators are \textit{approximately weakly stationary} in the sense
that for any $\ell \ge 0$
\begin{align} \label{e:apprst}
\max_{n} \sup_{u,v \in I_2}\big|\mathbb{C}ov(X_{n,M_N}(u),X_{n+\ell,M_N}(v))-\mathbb{C}ov(X_{1}(u),X_{\ell+1}(v))\big|=o(1).
\end{align}
The second part of Condition iii) implies that the sparse estimators
are also smooth in the H{\"o}lder-sense. It is sufficient to take the supremum over all $N \ge N_0$ for some fixed $N_0 \in \mathbb{N}$ in this condition.
Typically, smoothness of the sparse estimators follows from
the specific estimation technique combined with smoothness
of the latent functions. Illustrations based  on
Example \ref{ex1} are provided in Section \ref{sec_ex} of the
Supplementary Information.

We now specify assumptions on the interval $I_2$ on which the
functions $X_{n, M}$ are defined.
We consider three cases: A fixed indexing interval, an expanding one
and an infinite one. In the expanding scenario,
we allow the growth of the interval $I_2$ to be at a sufficiently slow,
polynomial speed in $N$, where the precise rate depends on the constants
in Assumption \ref{ass_1}. In the infinite domain case,
we assume that the tails of the random functions are exponentially decaying.

\begin{ass} \label{ass_1b}
One of the following three scenarios holds:
    \begin{itemize}
    \item[i)] (Fixed domain condition) $I_2$ is a fixed, compact interval. 
\item[ii)] (Increasing domain condition) It holds that $I_1=\mathbb{R}$ and the interval
        $I_2=I_2(N)$ depends on $N$ and is monotonically increasing, i.e., $I_2(N) \subset I_2(N+1)$. There exists a sufficiently small $\rho>0$ such that $|I_2(N)| \le C_4 N^{\rho}$. The constant $\rho$ depends only on the constants $C_i, J, \nu, \xi, \gamma$ from Assumption \ref{ass_1}. 
\item[iii)] (Infinite domain condition)  It holds that $I_1=I_2=\mathbb{R}$.
        There exist constants $\kappa, C_5>0$  such that for any $y>0$
        \begin{align*}
         \sup_n \mathbb{E} \sup_{|u|>y} |X_n(u)|^2 \le &C_5 \exp(-y^{\kappa}), \\
         \sup_{n,M}\mathbb{E} \sup_{|u|>y} |X_{n,M}(u)|^2 \le &C_5 \exp(-y^{\kappa}).
        \end{align*}
\end{itemize}
\end{ass}

To state our main result,
we define the long-run variance kernel of the latent functions restricted
to $I_2$. For $u,v \in I_2$ and $1 \le k, \ell \le d$, set
\begin{align} \label{e:lr_var}
c[k,\ell](u,v) := & \mathbb{E}\bigg[\Big(X_1[k](u) -\mathbb{E}[X_1[k](u)]\Big)\Big(X_1[\ell](v) -\mathbb{E}[X_1[\ell](v)]\Big)\bigg]\\
&  +2 \sum_{n \ge 2} \mathbb{E}\bigg[\Big(X_1[k](u) -\mathbb{E}[X_1[k](u)]\Big)\Big(X_n[\ell](v) -\mathbb{E}[X_n[\ell](v)]\Big)\bigg]. \nonumber
\end{align}
 Recall that $\pi_\infty$ is the Prokhorov metric defined by
\eqref{e:def_Prokhorov} with
$\mathcal{M} = \mathcal{C}_0(I_2,\mathbb{R}^d)$ and
$d_{\mathcal{M}}$ defined by the supremum norm.

\begin{theorem} \label{thm1}
Suppose that Assumptions \ref{ass_1} and \ref{ass_1b} hold.
Let $W$ be the centered Brownian motion with sample paths in
$\mathcal{C}_0(I_2, \mathbb{R}^d)$, defined by the covariance \eqref{e:lr_var}. Then, there exist some $C, \tau>0$ such that
 \begin{align}\label{e:weak:thm1}
 \pi_\infty\big(P_{N}, W \big) \le C N^{-\tau}.
\end{align}
The constants $C, \tau$ in this theorem depend only on the constants $C_i, J, \nu$, $\xi, \gamma, \rho, \kappa$ and $d$.
 \end{theorem}

Theorem \ref{thm1} advances research on invariance principles in function
spaces in  the following ways: It gives a  finite sample bound on
the distributional distance between $P_N$ and $W$. 
As a consequence,
we obtain in the next sections new coupling results with important
applications to statistical theory.  As far as we can tell,
such applications cannot be derived from traditional
weak convergence results.
But even as a mere weak convergence result,
Theorem \ref{thm1} makes a valuable contribution, c.f. Table~\ref{tab:1}.
There, we point out that our invariance principle operates
under stronger dependence, less restrictions on the domain of the functions and weaker stationarity restrictions than related works. Finally, we notice that the constants $C, \tau$ in the Theorem are not bounded explicitly, and such bounds are not required for our subsequent probabilistic and statistical applications, where polynomial decay is sufficient.

\medskip

\noindent{\sc outline of the proof  of Theorem
\ref{thm1}}:  The proof  is quite involved,
and hence we only provide an outline of the most important steps here,  and
present a complete, rigorous proof in Section \ref{AppA}
of the Supplementary Information. We confine ourselves to the
case of $I_2 =\mathbb{R}$,  which exhibits most challenges,  and set $d=1$.

\textbf{Preliminaries} The proof rests on some definitions
that we will here only heuristically introduce.
First, there exists a constant $\rho>0$ that helps
us  decompose the index set of $P_N$ into two parts
\[
[0,1]\times \mathbb{R} = \big([0,1]\times [-N^\rho, N^\rho] \big) \cup \big([0,1]\times [-N^\rho, N^\rho]^c\big).
\]
Second, there is a constant $\varsigma>0$ that we use to define a
grid $\mathcal{G}_N$ on $[0,1]\times [-N^\rho, N^\rho]$ with any
point in $[0,1]\times [-N^\rho, N^\rho]$  having a distance
$\le N^{-\varsigma}$ to the closest gridpoint.
This is possible with a total of
$|\mathcal{G}_N|\le C N^{\rho+2\varsigma}$ points.
For a pair of coordinates $(\lambda, u)$,  we refer to the closest
gridpoint by $(\lambda_{\mathcal{G}_N}, u_{\mathcal{G}_N})$
(ambiguities in the case of multiple closest points are ignored in this outline).
Third, we define the discretization of a function
$f \in L^\infty([0,1]\times \mathbb{R}, \mathbb{R})$,
called $f^{dis}$,  essentially by setting
\[
f^{dis}(\lambda, u):=
\begin{cases}
 f(\lambda_{\mathcal{G}_N}, u_{\mathcal{G}_N}),\,\,\quad |u|\le N^\rho,\\
    0, \qquad\qquad\quad \quad |u|> N^\rho.
\end{cases}
\]
Notice that the subspace of discretizations
 \[
    Dis(N):= \{f^{dis}: f \in L^\infty([0,1]\times \mathbb{R}, \mathbb{R})\}
\]
is finite dimensional with dimension $|\mathcal{G}_N|$.
The dimension grows with $N$, but only slowly as in the
proof $\rho$ and $\varsigma$
are small numbers with $\rho<<\varsigma$.

\textbf{The proof} The main idea of the proof is the following:
Get from the process $P_N$ to the finite dimensional version
$P_N^{dis}$. Use a finite sample Gaussian approximation
on the space $Dis(N)$ showing that  $P_N^{dis}$ is close to $W^{dis}$
in distribution,  and finally ensure that $W^{dis}$ is close to $W$.\\
\textit{Step 1:} First, we show that $\|P_N-P_N^{dis}\|$
    tends to $0$ sufficiently fast. This problem can be decomposed
    into bounding the tail $T_N$ and the difference $D_N$, where
\begin{align*}
    T_N := \sup_{\lambda \in [0,1]}\sup_{|u|>N^\rho} |P_N(\lambda, u)|, \quad D_N := \sup_{\lambda \in [0,1],u \in [-N^{\rho}, N^{\rho}]}|P_N(\lambda,u)-P_N^{dis}(\lambda,u)|.
\end{align*}
The tail is controlled by the exponential tail conditions in Assumption \ref{ass_1b}, part iii). The difference $D_N$ is closely related to a modulus of continuity, because  we are looking at a maximum of the type
\[
D_N= \sup_{\lambda \in [0,1],u \in [-N^{\rho}, N^{\rho}]}|P_N(\lambda, u)- P_N(\lambda_{\mathcal{G}_N}, u_{\mathcal{G}_N})|.
\]
To prove that $D_N$ is small, we control moments of the modulus
of continuity. The proof relies on empirical process bounds, such as
Theorem 2.2.4 in \cite{vaart:wellner:1996}. We also see that we carefully have to weigh the size of $\rho$ compared to $\varsigma$, because the grid has to be dense enough (big $\varsigma$) such that $D_N \to 0$, all the while the domain that we are considering is growing (with speed moderated by $\rho$).\\
\textit{Step 2:} Next, we prove proximity in the Prokhorov distance
of the discretized verssions $P_N^{dis}, W^{dis}$. This proof consists
essentially of two separate approximations. First, we take a normal
distribution $\mathcal{N}(0, \Sigma_N)$ on the vector space $Dis(N)$,
where $\Sigma_N$ reflects the covariance of $P_N^{dis}$.
We use a finite sample Gaussian approximation for high-dimensional
strongly mixing vectors to show  that
$\pi_\infty(P_N^{dis},\mathcal{N}(0, \Sigma_N))$ is small.
The approximation is a consequence of Theorem 3.27 in
\cite{dehling:mikosch:sorensen:2002} and
a precise statement is given in Lemma \ref{lemdehl} in  the
Supplementary Information.
The approximation guarantees that
$\pi_\infty(P_N^{dis},\mathcal{N}(0, \Sigma_N))$ decays at polynomial rate
controlled  by the dependence (i.e. $\nu$),
as well as the dimension of $Dis(N)$ (i.e. $\rho$ and $\varsigma$).
In the  second approximation, we need to show that
$\pi_\infty(W^{dis},\mathcal{N}(0, \Sigma_N))$
is also small - where we can understand $W^{dis}$
as a Gaussian distribution with different covariance,
say $\tilde  \Sigma_N$ that is determined by the long-run variance
kernel $c$, defined in \eqref{e:lr_var}.
We demonstrate first that  $\Sigma_N$
and $\tilde  \Sigma_N$ are close to each other.
Then,  we use the fact that the squared Prokhorov metric can
be upper bounded by the Wasserstein metric
$\mathcal{W}_2(W^{dis},\mathcal{N}(0, \Sigma_N))$.
The Wasserstein metric of two centered, high-dimensional
Gaussians has a closed
expression only depending on their covariances. 
Thus, with some additional arguments (including inequalities from  the famous contributions of \cite{powers:stormer:1970}) we are able to establish for a
large enough $C>0$,
\[
\pi_\infty(W^{dis},\mathcal{N}(0, \Sigma_N)) \le \mathcal{W}_2(\mathcal{N}(0, \tilde \Sigma_N),\mathcal{N}(0, \Sigma_N))^{1/2}\le C N^{-\tau}.
\]
 \textit{Step 3:} In the last step of the proof,
we show that $\|W-W^{dis}\|$ decays sufficiently fast. On a high level, the proof is similar to the one for the difference $\|P_N-P_N^{dis}\|$ -  we have a tail part and a modulus of continuity part, both of which have to be controlled. Yet the techniques to establish these results are very different. Showing that both parts are small involves some nuanced arguments for the Brownian motion in Banach spaces. For instance, to control the difference of $W$ and $W^{dis}$ on $[0,1] \times [-N^\rho, N^\rho]$, we 
exploit concentration inequalities for the modulus of continuity of the Banach-valued Brownian motion, found in the proofs of \cite{acosta:1985}. To control the tails of $W(\lambda, u)$ in $u$, we carefully investigate the interplay of decaying tails of the latent functions $X_i$ and the decay of the mixing coefficients, controlled by $\nu$.
\hfill $\square$

\section{Mathematical consequences} \label{s:cons}
In this section, we explore some probabilistic implications of
Theorem \ref{thm1}. Throughout the proofs of this section,
we denote by $C$ a generic positive constant, independent
of $N$, that may increase from one line to the next.
We begin by deriving couplings for sparse
functional data and then give two convergence results
for the case where the functions $X_1,X_2,\ldots$ are fully observed.

\textbf{Couplings and convergence in the Wasserstein metric}
In the following, we consider the linearly interpolated version of
the process $P_N$,
\begin{align} \label{e:partialsl}
P_N^{lin}(\lambda) = \begin{cases}
    P_N(\lambda), \,\, if \,\, \lambda \in \{0,1/N,\ldots,1\},\\
    linear\,\, between.
\end{cases}
\end{align}
It is easy to show that the distance between $P_N$ and $P_N^{lin}$
is asymptotically negligible. However, $ P_N^{lin}$ has the advantage
of living in  the simpler space
$\mathcal{C}_0([0,1]\times I_2, \mathbb{R}^d)$.
If $I_2=\mathbb{R}$,
we also require, that the functions are asymptotically vanishing, i.e. that for any $f \in \mathcal{C}_0([0,1]\times \mathbb{R}, \mathbb{R}^d)$
\[
\lim_{y \to \infty} \sup_{\lambda \in [0,1], |u|>y}|f(\lambda,u)|=0.
\]
Equipped with the sup-norm, the space
$\mathcal{C}_0([0,1]\times I_2, \mathbb{R}^d)$ is a complete,
separable subspace of $L^\infty([0,1]\times I_2, \mathbb{R}^d)$.
As we will now see, on such Polish spaces, proximity in weak
convergence metrics translates into useful coupling results. We begin with a consequence of Theorem \ref{thm1} that
is a starting point for subsequent results. It provides a novel coupling result for the functional partial sum process with a Banach-valued Brownian motion.

\begin{corollary} \label{cor_coup}
    Suppose that Assumptions \ref{ass_1} and  \ref{ass_1b} hold.
    Then, it is possible to define a coupling of the vector $(X_{1,M_N},\ldots,X_{N,M_N})$ and the Brownian motion $W$ on a probability space $(\Omega, \mathcal{A}, \mathbb{P})$ such that
    \begin{align} \label{e:coupl1}
\mathbb{P}\Big(\sup_{x \in [0,1]} \big\| P_N^{lin}(x) -W(x)\big\|
\ge CN^{-\tau}\Big)\le CN^{-\tau}.
\end{align}
\end{corollary}

\begin{proof}
    First, we notice that by a standard argument (see the maximum inequality in Lemma 2.2.2. of \cite{vaart:wellner:1996})
    \begin{align} \label{e:proxwe}
        & \mathbb{P}\big(\sup_{x \in [0,1]} \|P_N^{lin}(x)-P_N(x)\|>N^{-1/5}\big) \\
        \le & \mathbb{P}\big(\max_{n=1,\ldots,N}\|X_{n,M_N}/\sqrt{N}\|>N^{-1/5}\big)\nonumber\\
        \le & N^{-3/10}\{\mathbb{E} \max_{n=1,\ldots,N} \|X_{n,M_N}\|^J\}^{1/J} \le C N^{-3/10+1/J}, \nonumber
    \end{align}
    where $C>0$ is some constant depending only on $J$.
    Since $J>4$, the right side is of size $\mathcal{O}(N^{-1/20})$. Here, we have used that the $J$th moment of $\|X_{n,M_N}\|$ is uniformly bounded. Next, we exploit the fact that \eqref{e:proxwe} implies according to Theorem 2.13 from \cite{huber:ronchetti:2009} that $\pi_\infty(P_N^{lin}, P_N) \le  C N^{-1/20}$. We point out that this direction of the Theorem does not require that the underlying space is separable or complete. Using Theorem \ref{thm1}, we obtain with the triangle inequality and for $\tau<1/20$ that
    \[
     \pi_\infty(P_N^{lin},W) \le \pi_\infty(P_N^{lin},P_N) + \pi_\infty(P_N,W) \le C [N^{-\tau}+ N^{-1/20}] \le CN^{-\tau}.
    \]
    Since $P_N^{lin}, W$ live in the separable,  complete space
    $\mathcal{C}_0([0,1]\times I_2, \mathbb{R}^d)$,
    we can use the reverse direction of Theorem 2.13 from
    \cite{huber:ronchetti:2009}, which implies that there exists
    a coupling of  $(P_N^{lin}, W)$ on some appropriate probability
    space $(\Omega, \mathcal{A}, \mathbb{P})$, such that
    \eqref{e:coupl1} holds.
On this space, we can define the variables $X_{n,M_N}(u) = P_N^{lin}(n/N,u)-P_N^{lin}((n-1)/N,u)$, which finishes the proof.
\end{proof}

We now derive upper bounds
on the Wasserstein metric, c.f. \eqref{e:wm},
and the moments of the
coupling in Corollary \ref{cor_coup}. 
Part i) of Theorem \ref{t:wass}  is related to recent results
of \cite{Liu:Wang:2023}
who investigate invariance principles in the Wasserstein metric
for univariate dynamical systems. We formulate the result for the Wasserstein metric $\mathcal{W}_q$ and $q>2$ (depending again on the constants in Assumptions \ref{ass_1} and \ref{ass_1b}), which directly implies the result for $q =2$. However, it is sometimes useful in the theory of stochastic processes to have results for slightly more than second moments (see also our proof of Corollary \ref{cor:strong2}).

\begin{theorem} \label{t:wass}
    Suppose  Assumptions \ref{ass_1} and  \ref{ass_1b} hold.
    \begin{itemize}
        \item[i)]
    Then, for some $q>2$
    \[
    \mathcal{W}_q(P_N^{lin},W) \le  \bar C N^{-\bar \tau},
    \]
    where the constants $q, \bar C, \bar \tau$ only depend on the constants
    $C_i, J, \nu, \xi, \gamma, \rho$ and $\kappa$ from Assumptions \ref{ass_1} and \ref{ass_1b}.
    \item[ii)] It is possible to define a coupling of the vector $(X_{1,M_N},\ldots,X_{N,M_N})$ and the Brownian motion $W$ on a probability space $(\Omega, \mathcal{A}, \mathbb{P})$ such that
    \[
    \Big\{\mathbb{E} \sup_{x \in [0,1]} \big\| P_N^{lin}(x) -W(x)\big\|^q\Big\}^{1/q} \le \bar C N^{-\bar \tau}.
    \]
    \end{itemize}
\end{theorem}

\begin{proof}
Consider the coupling $P_N^{lin}, W$ from Corollary \ref{cor_coup}.
Then, with the same constant $C$ from that corollary,
we can decompose the expected distance of   $P_N^{lin}$ and $ W$ as follows:
\begin{align*}
     & \mathbb{E}\big[\|P_N^{lin}-W\|^q\big] \\
    = & \mathbb{E}\big[\|P_N^{lin}-W\|^q\mathbb{I}\{\|P_N^{lin}-W\|\le CN^{-\tau}\}\big]\\
    &+\mathbb{E}\big[\|P_N^{lin}-W\|^q\mathbb{I}\{\|P_N^{lin}-W\|> CN^{-\tau}\}\big]\\
    \le & C^qN^{-q\tau} + \mathbb{E}\big[\|P_N^{lin}-W\|^q\mathbb{I}\{\|P_N^{lin}-W\|> CN^{-\tau}\}\big].
\end{align*}
Using Hölder's inequality  and \eqref{e:coupl1}, we get
\begin{align*}
& \mathbb{E}\big[\|P_N^{lin}-W\|^q\mathbb{I}\{\|P_N^{lin}-W\|> C N^{-\tau }\}\big]\\
\le & C^{(4-q)/4} N^{-(4-q)\tau /4} \big\{\mathbb{E}\big[\|P_N^{lin}-W\|^4\big]\big\}^{q/4}.\\
\end{align*}
Now, from Lemma \ref{lemboundPN} it follows that
\[
\mathbb{E}\big[\|P_N^{lin}\|^4\big]^{q/4}\le C' N^{ \tau'}, \quad\mathbb{E}\big[\|W\|^4\big]^{q/4} \le C',
\]
for some large enough constant $C'$ and some $ \tau' <(4-q)\tau /4$. This proves the first claim of the Theorem.
The second part follows directly from our above construction.
We take the probability space on which our coupling $W, P_N^{lin}$ is defined.
Noticing that $X_{n,M_N}(u) = P_N^{lin}(n/N,u)-P_N^{lin}((n-1)/N,u)$,
 we see that we can define on the same probability space the vector
 of functions $(X_{1,M_N},\ldots,X_{N,M_N})$. \\
\end{proof}

\begin{remark} \label{rem_mixing} (Couplings for mixing sequences)
    One interesting innovation introduced by Theorem \ref{t:wass} is that it allows approximating a sequence of $\alpha$-mixing random functions by a sequence of independent random functions. In particular, Theorem \ref{t:wass}, part ii) entails directly that we can define on a suitable probability space $(\Omega, \mathcal{A}, \mathbb{P})$ the vector $(X_{1,M_N},\ldots,X_{N,M_N})$ and i.i.d. Gaussians $Z_1,...,Z_N$ such that for some $q>2$
    \[
    \Big\{\mathbb{E} \max_{1 \le k \le N} \Big\| \frac{1}{\sqrt{N}}\sum_{i=1}^k X_{i,M_N}-\frac{1}{\sqrt{N}}\sum_{i=1}^k  Z_i\Big\|^q\Big\}^{1/q} \le \bar C N^{-\bar \tau}.
    \] 
    Approximations of such types 
    are generally \textit{not available under $\alpha$-mixing on Banach or even Hilbert spaces, not even for time series of fully observed functions and not even in the strictly stationary case.}. 
    In the case of $\beta$- or $\phi$-mixing, numerous coupling results exist, such as the ones by \cite{berbee:1979}  or \cite{dehling:philipp:1982}, that approximate weakly dependent by independent variables, and that extend even to abstract Polish spaces. Such approximations remain true for $\alpha$-mixing in finite dimensional Euclidean spaces, as shown in \cite{berkes:philipp:1979}. But, these approximations break down in infinite dimensions, even in the simplest case of separable Hilbert spaces (\cite{dehling:1983b}). 
\end{remark}

We now consider some applications of  Theorems \ref{thm1} and  \ref{t:wass}
for fully observed functions. More precisely, we make the following assumption.

\begin{ass} \label{a:S1} (Fully observed setting)
For all $i, N$,   $X_{i,M_N}=X_i$  with $I_1=I_2$.
\end{ass}

\textbf{A bounded law of the iterated logarithm}
Under Assumption \ref{a:S1}, Theorem \ref{thm1} implies
a bounded law of the iterated logarithm.

\begin{corollary} \label{cor:lil}
    Suppose that the conditions of Theorem \ref{thm1} and Assumption
    \ref{a:S1} hold. Put $\sigma^2 := \mathbb{E}\|W\|^2>0$.
    Then, for a fixed, sufficiently large constant $C>0$,
    \[
    \limsup_N \bigg\|\frac{  \sum_{i=1}^{N} \big\{X_{i}- \mathbb{E}[X_{i}]\big\}}{\sqrt{N\log(\log(N))}}\bigg\|\le C \sigma \qquad a.s.
    \]
\end{corollary}
Laws of the iterated logarithm have been well-investigated in the literature
on probability theory for Banach spaces. We refer to \cite{einmahl:1993}
for some results and further references. It is possible to further strengthen
Corollary \ref{cor:lil} to a compact law of the iterated logarithm,
i.e. the object inside the norm lives almost surely in  a compact set. 
Such extensions are possible by applying the bounded law of the iterated logarithm to a projected version of the data and for details on this (standard) argument, we refer the interested reader to  Section 4.1 in \cite{kuelbs:philipp:1980}.
Compact laws of the iterated logarithm are in turn closely connected
to invariance principles with a slow approximation rate.
For instance, \cite{dehling:philipp:1982} show that they are equivalent
under certain additional conditions on the geometry of the Banach space.
Corollary \ref{cor:lil} is not a consequence of existing results and imposes
weaker dependence requirements than  existing bounded laws for
strongly mixing time series. For instance, if all variables have finite
exponential moments, i.e. $\sup_n \mathbb{E}\exp(\|X_n\|)< \infty$,
we require for the mixing coefficients
$\alpha_{X,\delta,M}(n) \le C n^{-\nu}$ for some $\nu>3$,
while \cite{dehling:1983} requires $\nu>8$.
Indeed, our requirement of $\nu>3$ is exactly identical to the mixing
rate in \cite{dehling:philipp:1982} for the simpler case of
Hilbert space valued time series. More important than the precise mixing
rate is that our theorem 
only imposes low-level conditions on the individual functions $X_i$, whereas the cited works rely on abstract, high-level
conditions for the entire sum $P_N(1)$. Such conditions are often required in Banach spaces, but can be quite hard to validate 
(see also Chapter 10.3 in \cite{ledoux:talagrand:1991}).

\begin{proof}
We assume that the functions $X_i$ are centered and define the sum $S_{1,n}:= \sum_{i=1}^n X_i$.
    Next, for any $k$ we define $n_k:= \lceil e^k \rceil$ and consider the inequalities
\begin{align*}
   &  \max_{n_k \le n \le n_{k+1}}\frac{\|S_{1,n}\|}{\sqrt{n\log(\log(n))}}\le \frac{\max_{n_k \le n \le n_{k+1}} \|S_{1,n}\|}{\sqrt{n_k\log(\log(n_k))}} \le \frac{\max_{n_k \le n \le n_{k+1}} \|S_{1,n}\|}{e^{k/2}\sqrt{\log(k)}} \\
    \le & \frac{\lceil e^{(k+1)/2}\rceil \max_{n_k \le n \le n_{k+1}} \|S_{1,n}/\sqrt{n_{k+1}}\|}{e^{k/2}\sqrt{\log(k)}} \le 2e^{1/2} \frac{\|P_{n_{k+1}}\|}{\sqrt{\log(k)}}.
\end{align*}
To employ the Borel-Cantelli Lemma, we now consider for some $C>0$ the bounds
\begin{align*}
    \mathbb{P}\bigg(\frac{\|P_{n_{k+1}} \|}{\sqrt{\log(k)}}>C\sigma \bigg) \le \mathbb{P}\bigg(\frac{\|W \|}{\sqrt{\log(k)}}>C\sigma \bigg) + C n_{k+1}^{-\tau} =:T_1(k)+T_2(k).
\end{align*}
Here, we have used Theorem \ref{thm1} in the first step.
By definition of $n_k$ we have $T_2(k) \le C e^{-k\tau}$, which is summable in $k$. For $T_1(k)$ we use Markov's inequality, to give
\begin{align*}
&\mathbb{P}\bigg(\frac{\|W \|}{\sqrt{\log(k)}}>C\sigma \bigg) = \mathbb{P}\bigg(\frac{\sqrt{2}\|W \|}{C\sigma }> \sqrt{\log(k^2)}\bigg)\\
\le & \mathbb{E}\exp(2\|W \|^2/[C^2\sigma^2]) k^{-2} \le C k^{-2}.   
\end{align*}
In the second step, we have used a well-known result by  \cite{fernique:1975} (p.11), which implies that the square-exponential moment exists for any Gaussian random variable on a separable Banach space.   In our case it implies that for some sufficiently large $C>0$
\[
\mathbb{E}[\exp(\|W \|^2/(C^2 \sigma^2)] <\infty.
\]
Now, $T_1(k)$ is also summable, and the Borel-Cantelli Lemma yields the desired result. 

\end{proof}

\textbf{An almost sure invariance principle} We now prove a strong
invariance principle with a polynomial approximation rate.
For this purpose, we impose the  stronger
$\beta$-mixing (reviewed in Section
\ref{sec_details} of the Supplementary Information).

\begin{ass}\label{a:S2}
The functions $X_i$ are $\beta$-mixing with coefficients
$\beta(n) \le C_1 n^{-\nu}$ for the same $\nu$ as in Assumption \ref{ass_1}.
\end{ass}

\begin{corollary} \label{cor:strong2}
Suppose that the conditions of Theorem \ref{thm1} and Assumptions
\ref{a:S1} and \ref{a:S2} hold.
Then, without changing its distribution, the sequence
$(X_i)_{i \in \mathbb{N}}$ can be redefined on a suitable probability space,
where a centered Brownian motion $W$  with covariances \eqref{e:lr_var}
exists, and  for some $\bar \tau>0$,
    \[
    \Big\|\sum_{i=1}^N X_i -W(N)\Big\|=\mathcal{O}(N^{1/2-\bar \tau}), \quad a.s.
    \]
\end{corollary}
The  almost sure invariance principle in Corollary \ref{cor:strong2}
is related to recent results by \cite{lu:wu:xiao:xu:2022}
who prove a strong invariance principle for $\beta$-mixing  functional data.
Yet, there are important differences between the two results.
Our corollary is formulated on the Banach space of continuous functions
(as opposed to a separable Hilbert space), requires polynomially decaying
$\beta$-mixing coefficients (as opposed to exponentially decaying ones),
requires only weak stationarity (as opposed to strict stationarity)
and approximates the functional Brownian motion on an unbounded domain
$[0, \infty)$ (as opposed to a bounded domain).
The last point is relevant, because approximations on an unbounded domain
can be used  to derive and justify  open-ended monitoring schemes,
as elaborated on in Section \ref{sec_mon}.

\begin{proof}   The proof has two components: It employs a blocking argument for dependent data, with a coupling for $\beta$-mixing random variables and for the details we refer to a classical argumentation in \cite{dehling:1983}
    (Section 8). 
    Our focus is on the preconditions of this argument, which pose the major theoretical challenge, but are in our case a consequence of Theorems \ref{thm1} and \ref{t:wass}. \\
    First, we need to establish that our time series $(X_n)_{n \in \mathbb{N}}$ is $\beta$-mixing with coefficients $\beta(n) \le C n^{-(2+\zeta)}$ and has bounded moments $\sup_n \mathbb{E}\|X_n\|^{2+\zeta} \le C$ for some $\zeta>0$. Both conditions are obviously fulfilled in our case. Next, the following two conditions on the sums of Banach-valued variables have to be satisfied: For some $q>2$ the inequality
    \begin{align} \label{e:mom}
        \mathbb{E}\Big\| \sum_{i=a+1}^{a+n} X_i\Big\|^q \le Cn^{q/2}, \quad \forall a \ge 0, \,\, n \ge 1
    \end{align}
    holds and for some $\tau>0$ the weak approximation
    \begin{align}\label{e:weak}
 \pi_\infty\big(P_{n}(1), W(1) \big) \le C n^{-\tau}
\end{align}
holds.
These two results are generally hard to show, but they follow immediately from our previous results. Indeed, \eqref{e:weak} is a direct consequence of Theorem \ref{thm1}, while \eqref{e:mom} follows from Theorem \ref{t:wass}. For the second claim we only consider $a=0$ (for illustration) and notice that Theorem \ref{e:mom} implies that
\begin{align*}
    \mathbb{E}\Big\| \frac{1}{\sqrt{n}}\sum_{i=1}^{n} X_i\Big\|^q = \mathbb{E}\| P_n(1)\|^q \le C\big\{\mathbb{E}\| W(1)\|^q + \mathbb{E}\| P_n(1)-W(1)\|^q\big\} \le C.
\end{align*}
    In the last step, we have used Lemma \ref{lemboundPN} to show that $\mathbb{E}\| W(1)\|^q\le C$ and Theorem  \ref{t:wass} to show that $\mathbb{E}\| P_n(1)-W(1)\|^q\le C$.
\end{proof}

$ $\\
We point out that while our proof of Corollary \ref{cor:strong2}
exploits the same blocking strategy as Theorem 3 in \cite{dehling:1983},
key technical components differ and it cannot be derived
from existing results.
One way to see this, is to consider random functions $X_1, X_2,\ldots$
with exponential moments $\sup_n \mathbb{E}\exp(\|X_n\|)< \infty$.
In this case, we need for some arbitrarily small $\epsilon>0$
the mixing assumption  $\beta(n) \le C n^{-(3+\epsilon)}$,
which is weaker than  the condition $\beta(n) \le C n^{-(8+\epsilon)}$
imposed by \cite{dehling:1983} for the same scenario.

\section{Monitoring sparse functional time series} \label{sec_mon}
The theory related to Theorem \ref{thm1} is most directly motivated by
statistical problems arising in monitoring for a change point in
a time series of random densities that  must be estimated  from
sparse scalar observations. In this section, after presenting some
background, we develop  the monitoring theory
in a more general setting.

\textbf{The monitoring paradigm} The monitoring problem was originally
formulated by
\cite{chu:stinchcombe:white:1996} in the context of linear regression.
In a monitoring regime, the user observes an incoming data stream
and tests sequentially for structural breaks
over an open-ended and theoretically infinite  time period.
Monitoring procedures have been developed for univariate and multivariate
data, and we refer to \cite{horvath:kokoszka:huskova:steinebach:2003},
\cite{aue:horvath:huskova:kokoszka:2006}
and \cite{gosmann:kley:dette:2021} for a few examples.
An informative account is given by \cite{aue:kirch:2023}.
The validation of monitoring schemes typically rests on strong
approximations of the partial sum process by a Brownian motion,
such as variants of the famous KMT
approximations (Koml{\'o}s {\it et al.}
(\cite{kmt:1975},  \cite{kmt:1976})) and related results (\cite{eberlein:1986}).
For a discussion of this subject, we refer to Section 6
of \cite{horvath:rice:2014}.

\textbf{Monitoring of  functional time series} KMT type 
approximations  exist only in  finite dimensional 
Euclidean spaces. They are not available for functional time series,
so the monitoring problem
has been much less studied in the context of functional data. 
In
\cite{aue:hormann:horvath:huskova:2014} monitoring is conducted after projecting  
Hilbert-space valued data on a finite dimensional coordinate system.  
Such an approach is not  suitable for monitoring an entire
infinite dimensional parameter, like a mean function. Another approach is adapting strong approximations, such as Corollary \ref{cor:strong2} to monitoring without dimension reduction (see \cite{kutta:kokoszka:2025}, who consider $\phi$-dependent time series), but this approach is inherently not suitable for non-stationary sparse estimators.
Therefore, we present in this section, a different approach.
We will demonstrate that the coupling given in
Corollary \ref{cor_coup} is already enough to validate monitoring.
This represents a new  strategy to the validation of monitoring schemes that has two additional benefits:
First, it allows us to work under the weaker dependence assumption
of $\alpha$-mixing rather than that of $\beta$-mixing.
To the best of our knowledge, there do not exist comparable results to
Corollary \ref{cor:strong2} on Banach spaces or even Hilbert spaces
for $\alpha$-mixing sequences. As explained in Remark \ref{rem_mixing}, this is because central coupling results available for $\beta$-mixing cannot be generalized to strong mixing in infinite dimensional spaces (see also \cite{dehling:1983b}). Second, and more importantly,
our results  hold for sparse functional estimators that form 
a triangular array with only approximately stationary rows.  
Such a  scenario  is clearly not covered by strong
convergence results such as Corollary \ref{cor:strong2}.
We also note that this strategy  can be used for different
data structures as long as a result analogous to Theorem \ref{thm1} holds. 
We proceed with  an essential mathematical outline of the monitoring
problem. For additional motivation and a more general introduction we refer
to \cite{gosmann:2020}.

\textbf{Problem statement} Consider a time series of latent functions
$(X_n)_{n \in \mathbb{N}}$ in $\mathcal{C}_0(I_1, \mathbb{R})$
with mean functions $\mu_n:= \mathbb{E}X_n$. 
We assume that there exists an initial \textit{stable period} of length $N$,
where $\mu_1=\mu_2=\ldots.=\mu_N=:\mu^{(1)}$, and we are
interested in detecting a change in the mean function afterwards.
The stable period is treated as a training data set. 
We compare the observations over  the \textit{monitoring period} to those in the 
training set. The monitoring period 
is potentially of infinite length. Asymptotics are formulated for
$N\to \infty$, i.e. for an increasing training set. We phrase
the monitoring problem as a test decision between the hypotheses
\begin{align} \label{e:hyp}
   & H_0: \mu^{(1)}=\mu_1=\ldots=\mu_N=\mu_{N+1}=\mu_{N+2}=\ldots, \\
   \textnormal{versus} \quad& H_1: \mu^{(1)}=\mu_1
   =\ldots=\mu_{N+k^*}\neq \mu^{(2)}=\mu_{N+k^*+1}
   =\mu_{N+k^*+2}=\ldots\quad \nonumber
\end{align}
for some $k^*<\infty$.
As before,  we assume that the user has no direct
access to the latent functions $X_i$ but only to sparse
estimators $X_{i,M_N} = \mathscr{A}_{M_N}(X_n, \delta_{n,M_N})$,
which improve as $N$ grows (for the  definition of the sparse estimators,
see Section \ref{sec_sparse}).
The sampling scheme can then be described as follows:
During the stable period of length $N$, the user observes
(constructs from available data) the sparse functions
$X_{1,M_N},\ldots,X_{N,M_N}$. Then, the monitoring procedure is launched,
where successively new observations $X_{N+1,M_N}, X_{N+2,M_N}, \ldots$
arrive. At each time $N+k$, the user employs the updated  sample
to decide whether to reject $H_0$
or not. 

\textbf{A change point detector} The decision is based on 
a detector $\widehat{\Gamma}(k)$ passing a threshold value $q$. 
We employ a version of the CUSUM statistic defined as
\begin{align} \label{e:level}
 \widehat{\Gamma}(k):=\bigg\| \frac{\frac{k}{N}
 \sum_{i=1}^NX_{i,M_N}-\sum_{i=N+1}^{N+k}X_{i,M_N}}{\sqrt{N}\big(1+\frac{k}{N}\big)}\bigg\|.
\end{align}
For a threshold value $q$, we reject the hypothesis $H_0$
at time $N+k$ ($k \ge 1$) if 
\begin{align} \label{e:testd}
      \widehat{\Gamma}(k)>q.
\end{align}
We will  specify $q$  to control the
type I error, also known as the  false alarm rate.  Specifically, for a 
 user-determined $\alpha \in (0,1)$, we require 
\begin{align} \label{e:level2}
 \lim_{N \to \infty}\mathbb{P}_{H_0}\Big(\sup_{k \ge 1}\widehat{\Gamma}(k)>q\Big)=\alpha.
\end{align}
 Our second aim will be to prove asymptotic power under $H_1$
 for $k^*=k^*(N)<\infty$, i.e. to show that
 \begin{align}
     \label{e:cons}
     \lim_{N \to \infty}\mathbb{P}_{H_1}\Big(\sup_{k \ge 1}\widehat{\Gamma}(k)>q\Big)=1.
 \end{align}
 The model for the latent functions is  
\begin{align} \label{e:Xn}
    X_n:= \begin{cases}
        \mu^{(1)}+\varepsilon_n, \quad n \le N+k^*,\\
        \mu^{(2)}+\varepsilon_n, \quad n > N+k^*.
    \end{cases}
\end{align}
We impose the following conditions.

\begin{ass} $ $ \label{ass_2}

\begin{itemize} 
\item[i)] (Dependence \& Data structure)   Assumptions  \ref{ass_1}  
 and  \ref{ass_1b} hold. 
    \item[iii)] (Mean functions) The functions $\mu^{(1)}, \mu^{(2)}$ may depend on $N$. Under $H_1$, they satisfy 
    \[
    \sqrt{N}\|\mu^{(1)}-\mu^{(2)}\| \to \infty.
    \]
    \item[iv)] (Mean approximation) Defining $\mu_{i,M_N}:=\mathbb{E}[X_{i,M_N}]$, it holds that
    \[
    \sup_{i\in \mathbb{N}} \sqrt{N} \|\mu_{i,M_N}-\mu_i\|=o(1).
    \]
\end{itemize}

\end{ass}
 Condition iii) ensures that  the distance between the mean before
and after the change is detectable. 
Condition iv)  guarantees that the sparse estimators $X_{i,M_N}$
have approximately the same mean functions as the unobservable 
latent functions. 
We can now formulate the main result of this section. For this purpose,
denote by $q_{1-\alpha}$ the upper $\alpha$-quantile of the distribution of
$\sup_{x \in [0,1]}\|W_\varepsilon(x)\|$, where $W_\varepsilon$ is a
centered Brownian motion on $\mathcal{C}(I_2, \mathbb{R})$
with covariance kernel
\begin{align} \label{e:lr_var}
c_\varepsilon(u,v) := & \mathbb{E}[\varepsilon_1(u)\varepsilon_1(v)]
  +2 \sum_{n \ge 2} \mathbb{E}[\varepsilon_1(u) \varepsilon_n(v)]. \nonumber
\end{align}

\begin{theorem} \label{thm_monitoring}
Suppose that Assumption \ref{ass_2} holds and that $\alpha \in (0,1)$.
Then for a threshold value $q=q_{1-\alpha}$,  the test decision \eqref{e:testd}
asymptotically holds nominal level $\alpha$ (see \eqref{e:level2}) and is
consistent (see \eqref{e:cons}).
\end{theorem}

\textbf{Main \textlabel{coupling argument}{m:c}} Before we give a
precise proof of Theorem \ref{thm_monitoring},
we want to highlight the main novelty of our proof strategy.
Sequential change point detection can be roughly divided into
two different paradigms. First, there exist approaches where
the monitoring ends after a fixed time period, proportional to the size of
the stable sample $N$, say after $TN$ observations for some $T>1$.
In such \textit{closed-end} scenarios the main probabilistic tool to
validate monitoring schemes is a weak invariance principle of the form
\begin{equation} \label{e:weakcon}
\{P_N(x)\}_{x \in [0,T]}\overset{d}{\to} \{W(x)\}_{x \in [0,T]}.
\end{equation}
Second, in this paper, we have defined monitoring as an \textit{open-ended}
process, where the procedure runs (theoretically) for an infinite time
if no change is detected. Such procedures have the practical advantage
that it is not necessary to pre-specify an end of the procedure before
it is launched. As we have mentioned above, open-ended procedures
are theoretically more challenging and require stronger probabilistic tools.
Often some sort of invariance principle on the non-compact domain $[0,\infty)$
is invoked, such as Corollary \ref{cor:strong2}, because a result of type
\eqref{e:weakcon} is insufficient for any fixed $T>1$. \\
Our main Theorem \ref{thm1} provides a  middle-ground between
these two kinds of results. It  implies existence of a coupling $P_N,W$
such that (essentially)
\[
\mathbb{P}\Big( \sup_{x \in [0,N^{\tau_1}]}\|P_N(x)-W(x)\|>CN^{-\tau_2}\Big)\le C N^{-\tau_2}
\]
for some sufficiently small numbers $\tau_1, \tau_2>0$.
This is still an approximation on a compact interval $[0,N^{\tau_1}]$,
but the size of the interval increases with $N$, making it stronger
than \eqref{e:weakcon} and sufficient to prove an open-ended monitoring
scheme. To the best of our knowledge such
\textit{approximations on a growing domain}
have not been used before to validate monitoring schemes.
Yet, they constitute an interesting alternative compared to standard methods,
specifically because results like Theorem \ref{thm1} seem to be available
under weaker assumptions than approximations on an infinite domain such
as Corollary \ref{cor:strong2}.

\begin{proof} We begin with the analysis under the null hypothesis $H_0$,  starting with some mathematical preliminaries. First, we define the partial sum process
\[
P_{N_1,N_2}(\lambda,u):= \frac{1}{\sqrt{N_2}} \sum_{i=1}^{\lfloor \lambda N_1 \rfloor} \big\{X_{i,M_{N_2}}(u)- \mu_{i,M_{N_2}}(u)\big\}, \quad \lambda \in [0,1], \,\, u \in I_2.
\]
Let now $\zeta \in (0,1/2)$ be a sufficiently small number (to be precisely determined later) and choose $k_N = \lfloor N^{1+\zeta}\rfloor$. Assumption \ref{ass_1}, Conditions i)-iii) still hold true for the array of random variables $\{X_{1,M_N},\ldots,X_{k_N,M_N}\}$ with sample size $k_N$, where Condition iii) remains valid replacing $\gamma$ by $\gamma/2$. Therewith, the constants in Assumption \ref{ass_1} do not depend on the specific choice of $\zeta \in (0,1/2)$ anymore.
We can now use Theorem \ref{thm1} which implies that there exist some $C, \tau>0$ (independent of $\zeta$) satisfying
\begin{align} \label{e:thmrate}
\pi_\infty\big(P_{k_N, N}, W_\varepsilon \big) \le C k_N^{-\tau} \le C N^{-\tau}.
\end{align}
Since $C, \tau$ do not depend on $\zeta$, we can shrink $\zeta$ to satisfy the inequality $\zeta \le \tau$. We will subsequently use the approximation rate \eqref{e:thmrate} for a suitable coupling argument.\\
For now, we define the statistic
\[
\hat \gamma :=  \max_{k \ge 1}\widehat{\Gamma}(k):= \max_{k \ge 1} \bigg\| \frac{\frac{k}{N}\sum_{i=1}^NX_{i,M_N}-\sum_{i=N+1}^{N+k}X_{i,M_N}}{\sqrt{N}\big(1+\frac{k}{N}\big)}\bigg\|.
\]
Centering the random variables $\varepsilon_{i,M_N}:= X_{i,M_N}-\mu_{i,M_N}$ we observe with a simple calculation (under $H_0$) that
\begin{align} \label{e:gamdec}
    \hat \gamma = &\max_{k \ge 1} \tilde \Gamma(k)+\mathcal{O}\Big( \sup_{i}\sqrt{N}\|\mu_{i,M_N}-\mu_i\|\Big), \quad \textnormal{where} \\
    \tilde \Gamma(k):= &\bigg\| \frac{\frac{k}{N}\sum_{i=1}^N\varepsilon_{i,M_N}-\sum_{i=N+1}^{N+k}\varepsilon_{i,M_N}}{\sqrt{N}\big(1+\frac{k}{N}\big)}\bigg\|. \nonumber
\end{align}
According to Assumption \ref{ass_2}, Condition iv) the second term on the right side of \eqref{e:gamdec} is of size $o(1)$.
The first term can be decomposed into two separate maxima, which yields
\begin{align*}
\hat \gamma = &\max(\hat \gamma_1, \hat \gamma_2) +o(1), \quad
\textnormal{where} \quad
    \hat  \gamma_1 := \max_{1 \le k \le k_N}\tilde \Gamma(k), \quad  \hat  \gamma_2 := \max_{ k \ge k_N}\tilde \Gamma(k).
\end{align*}
In the next step, we will show that $\hat  \gamma_1$, $\hat  \gamma_2$ can be expressed (except for negligible remainders) as functionals of the stochastic process $P_{k_N,N}$.
To see this, we first rewrite $\hat  \gamma_1$ as follows
\begin{align} \label{e:hatgam}
    \hat  \gamma_1=& \sqrt{k_N}\max_{1 \le k \le k_N} \bigg\| \frac{\frac{k}{N}\frac{1}{\sqrt{k_N}}\sum_{i=1}^N\varepsilon_{i,M_N}-\frac{1}{\sqrt{k_N}}\sum_{i=N+1}^{N+k}\varepsilon_{i,M_N}}{\sqrt{N}\big(1+\frac{k}{N}\big)}\bigg\|\\
    = & \sqrt{\frac{k_N}{N}}\max_{1 \le k \le k_N} \bigg\| \frac{\frac{k}{N} P_{k_N,N}(N/k_N)-[P_{k_N,N}((N+k)/k_N)-P_{k_N,N}(N/k_N)]}{\big(1+\frac{k}{N}\big)}\bigg\|.\nonumber
\end{align}
The right side is a transform of $P_{k_N,N}$.
Next, we investigate $\hat \gamma_2$. Notice that we can (using the triangle inequality) write it as
\begin{align}\label{e:hatgam2de}
\hat \gamma_2 = \max_{k \ge k_N} \bigg\| \frac{\frac{k}{N}\sum_{i=1}^N\varepsilon_{i,M_N}}{\sqrt{N}\big(1+\frac{k}{N}\big)}\bigg\| + \mathcal{O}\bigg( \max_{ k \ge k_N} \bigg\| \frac{\sum_{i=N+1}^{N+k}\varepsilon_{i,M_N}}{\sqrt{N}\big(1+\frac{k}{N}\big)}\bigg\| \bigg).
\end{align}
We further study the second term on the right of \eqref{e:hatgam2de}, which gives
\begin{align} \label{e:lilapp}
     \max_{ k \ge k_N} \bigg\| \frac{\sum_{i=N+1}^{N+k}\varepsilon_{i,M_N}}{\sqrt{N}\big(1+\frac{k}{N}\big)}\bigg\|  \le &  \max_{ k \ge k_N} \bigg\| \frac{\sqrt{N}}{k}\sum_{i=N+1}^{N+k}\varepsilon_{i,M_N}\bigg\| \\
     \le & \max_{ k \ge k_N} \bigg\| \frac{1}{k^\beta}\sum_{i=N+1}^{N+k}\varepsilon_{i,M_N}\bigg\|. \nonumber
\end{align}
On the right side, we can choose $\beta >1/2$ since $k \ge k_N = \lfloor N^{1+\zeta}\rfloor$. Using Lemma 9.2 in \cite{dehling:1983}, we can show that the right side of \eqref{e:lilapp} is $o_P(1)$ (we omit the precise calculation for sake of brevity).  
Adding this fact to \eqref{e:hatgam2de} yields
 \begin{align*}
     \hat \gamma_2 =  \max_{k \ge k_N} \bigg\| \frac{\frac{k}{N}\sum_{i=1}^N\varepsilon_{i,M_N}}{\sqrt{N}\big(1+\frac{k}{N}\big)}\bigg\| +o_P(1) = \bigg\| \frac{1}{\sqrt{N}}\sum_{i=1}^N\varepsilon_{i,M_N}\bigg\| +o_P(1) =: \bar \gamma_2 +o_P(1),
 \end{align*}
where $ \bar \gamma_2 $ is defined in the obvious way. Notice that we can rewrite
\[
\bar \gamma_2 =  \sqrt{\frac{k_N}{N}}\|P_{k_N,N}(N/k_N)\|.
\]
Now, our considerations up to this point have shown that
\[
\hat \gamma = \max(\hat \gamma_1, \bar \gamma_2) + o_P(1),
\]
where $\hat \gamma_1, \bar \gamma_2$
are transforms of the stochastic process $\{P_{k_N,N}(x): 0 \le x \le 1\}$.
The approximation result \eqref{e:thmrate} implies, that we can redefine
on a suitable probability space the processes $P_{k_N,N}$ and $W_\varepsilon $
such that
\begin{align} \label{e:coupl11}
\mathbb{P}\Big(\sup_{x \in [0,1]} \big\| P_{k_N,N}(x) -W_\varepsilon(x)\big\| \ge CN^{-\tau}\Big)\le CN^{-\tau}.
\end{align}
For this coupling result we refer again to Theorem 2.13
from \cite{huber:ronchetti:2009} (see also our Corollary \ref{cor_coup}, where the partial sum process and its linearized version can be interchanged for one another).
Recalling that by construction $k_N/ N^{1+\tau}=\mathcal{O}(1)$
implies that
\[
\sup_{x \in [0,1]} \sqrt{\frac{k_N}{N}}\|W_\varepsilon(x)-P_{k_N,N}(x)\|=\mathcal{O}_P\bigg( \sqrt{\frac{N^{1+\tau}}{N^{1+2\tau}}}\bigg)=o_P(1).
\]
We thus obtain for the variables $\hat \gamma_1, \bar \gamma_2$ that are redefined on this new probability space that
\begin{align*}
    \hat  \gamma_1 = & \tilde \gamma_1 + o_P(1), \quad \textnormal{and }\quad \bar  \gamma_2 =  \tilde \gamma_2 + o_P(1) \quad \textnormal{where}\\
    \tilde \gamma_1 := &\max_{1 \le k \le k_N} \sqrt{\frac{k_N}{N}}\frac{ \bigg\| \frac{k}{N} W_\varepsilon(N/k_N)-[W_\varepsilon((k+N)/k_N)-W_\varepsilon(N/k_N)]\bigg\|}{\big(1+\frac{k}{N}\big)},\\
    \tilde \gamma_2 := &\sqrt{\frac{k_N}{N}} \|W_\varepsilon(N/k_N)\|.
\end{align*}
Now, what remains to show is the weak convergence
\[
\max(\tilde \gamma_1, \tilde \gamma_2) \overset{d}{\to} \sup_{x \in [0,1]} \|W_\varepsilon(x)\|.
\]
This last step follows by exploiting some well-known continuity and transformation properties of the Banach-valued Brownian motion that hold in analogy to the one-dimensional Brownian motion (see, e.g. \cite{kuelbs:1973}). Hence, we can copy the steps from the univariate case, that is given in the proof of Theorem 2.1 in \cite{horvath:kokoszka:huskova:steinebach:2003}, which yields the desired result. \\
Finally, let us consider the behavior of the test statistic under the alternative, where a change occurs at $N+k^*$ with $k^*=k^*(N)$. Using the definition of $\hat \gamma$ and $\tilde \Gamma$ (see \eqref{e:gamdec}), we obtain after a short calculation under the alternative the decomposition
\begin{align*}
\hat \gamma \ge &\max_{k \ge 1} \frac{\|\max((k-k^*),0)(\mu^{(1)}-\mu^{(2)})\|}{\sqrt{N}\big(1+\frac{k}{N}\big)}\\
&-\max_{k \ge 1} \tilde \Gamma(k)-\mathcal{O}\Big( \sup_{i}\sqrt{N}\|\mu_{i,M_N}-\mu_i\|\Big).
\end{align*}
From the previous part of the proof, we already know that $\max_{k \ge 1} \tilde \Gamma(k)=\mathcal{O}_P(1)$ and from Assumption \ref{ass_2}, Condition iv) that the last term on the right is of order $o(1)$. We can hence study the first term on the right, which we can make even smaller by choosing $k=2k^*+N$. In this case, a small calculation shows that it equals
\begin{align*}
&\frac{(k-k^*)\sqrt{N}\|\mu^{(1)}-\mu^{(2)}\|}{k+N}= \sqrt{N}\|\mu^{(1)}-\mu^{(2)}\| \frac{k^*+N}{2(k^*+N)}\\
\ge &\frac{\sqrt{N}}{2}\|\mu^{(1)}-\mu^{(2)}\|\to \infty. 
\end{align*}
In the last step we have used that $\sqrt{N}\|\mu^{(1)}-\mu^{(2)}\| \to \infty$,
which holds according to Assumption \ref{ass_2}, Condition iii).
It follows that $\hat \gamma \overset{P}{\to} \infty$,
which concludes the proof.

\end{proof}

We conclude this section with some final remarks.
    In this section, we have explored the problem of sequential change point detection for sparse functional time series, using the coupling result from Corollary \ref{cor_coup}. 
    We have worked out a second statistical application of the same corollary in \cite{kutta:jach:kokoszka:2024}, where we focus more on numerical results and data applications. Roughly speaking, the aim of that work is identifying components of a functional panel, where changes have occurred.
    Many further potential applications exist. For example, the invariance principle from \cite{berkes:horvath:rice:2013} has frequently been used in proofs to show rapid convergence of change point estimators for Hilbert spaces (e.g. \cite{dette:kokot:volgushev:2020}). It is clear that our results can fill the same role for sparse functional data, even when the domain is unbounded.

\subsection*{Acknowledgements}

This work was conducted while Tim Kutta was affiliated with Colorado State University (he is currently at Aarhus University). Piotr Kokoszka was partially supported by the United States National Science Foundation
grant DMS--2123761. 

\subsection*{Declaration of interests}

The authors have no relevant financial or non-financial interests to disclose.

 \bibliographystyle{elsarticle-num} 
 \bibliography{main.bib}

\begin{thebibliography}{10}
\expandafter\ifx\csname url\endcsname\relax
  \def\url#1{\texttt{#1}}\fi
\expandafter\ifx\csname urlprefix\endcsname\relax\def\urlprefix{URL }\fi
\expandafter\ifx\csname href\endcsname\relax
  \def\href#1#2{#2} \def\path#1{#1}\fi

\bibitem{aston:kirch:2012}
M.~Aston, C.~Kirch, Detecting and estimating changes in dependent functional data, Journal of Multivariate Analysis 109 (2012) 204--220.

\bibitem{gromenko:kokoszka:reimherr:2017}
O.~Gromenko, P.~Kokoszka, M.~Reimherr, Detection of change in the spatiotemporal mean function, Journal of the Royal Statistical Society (B) 79 (2017) 29--50.

\bibitem{aue:rice:sonmez:2018}
A.~Aue, G.~Rice, O.~Sonmez, Detecting and dating structural breaks in functional data without dimension reduction, Journal of the Royal Statistical Society. Series B (Statistical Methodology) 80 (2018) pp. 509--529.

\bibitem{kim:zhao:shao:2015}
S.~Kim, Z.~Zhao, X.~Shao, Nonparametric functional central limit theorem for time series with application to self-normalized confidence interval, Journal of Multivariate Analysis 36 (2015) 277--290.

\bibitem{dette:kokot:volgushev:2020}
H.~Dette, K.~Kokot, S.~Volgushev, Testing relevant hypotheses in functional time series via self-normalization, Journal of the Royal Statistical Society Series B: Statistical Methodology 82 (2020) 629--660.

\bibitem{dette:kutta}
H.~Dette, T.~Kutta, Detecting structural breaks in eigensystems of functional time series, Electronic Journal of Statistics 15 (2021) 944 -- 983.

\bibitem{bugni:2009}
F.~Bugni, P.~Hall, J.~Horowitz, G.~Neumann, Goodness-of-fit tests for functional, Econometrics Journal 12 (2009) S1--S18.

\bibitem{horvath:kokoszka:rice:2014}
L.~Horv{\'a}th, P.~Kokoszka, G.~Rice, Testing stationarity of functional time series, Journal of Econometrics 179 (2014) 66--82.

\bibitem{cuesta:2019}
J.~Cuesta-Albertos, E.~Garcia-Portugu{\'e}s, M.~Febrero-Bande, W.~Gonz{\'a}lez-Manteiga, Goodness of fit tests for the functional linear model based on randomly projected empirical processes, The Annals of Statistics 47 (2019) 439--467.

\bibitem{hafouta:2023}
Y.~Hafouta, Convergence rates in the functional {CLT} for $\alpha$-mixing triangular arrays, Stochastic Processes and their Applications 161 (2023) 247--264.

\bibitem{bosq:2000}
D.~Bosq, Linear {P}rocesses in {F}unction {S}paces, Springer, 2000.

\bibitem{HKbook}
L.~Horv{\'a}th, P.~Kokoszka, Inference for {F}unctional {D}ata with {A}pplications, Springer, New York, 2012.

\bibitem{hsing:eubank:2015}
T.~Hsing, R.~Eubank, Theoretical {F}oundations of {F}unctional {D}ata {A}nalysis, with an {I}ntroduction to {L}inear {O}perators, Wiley, 2015.

\bibitem{merlevede:2003}
F.~Merlev\`ede, On the central limit theorem and its weak invariance principle for strongly mixing sequences with values in a {H}ilbert space via martingale approximation, Journal of Theoretical Probability 16 (2003) 625--653.

\bibitem{berkes:horvath:rice:2013}
I.~Berkes, L.~Horv{\'a}th, G.~Rice, Weak invariance principles for sums of dependent random functions, Stochastic Processes and their Applications 123 (2013) 385--403.

\bibitem{cuny:merlevede:2014}
C.~Cuny, F.~Merlev{\`e}de, {On martingale approximations and the quenched weak invariance principle}, The Annals of Probability 42~(2) (2014) 760--793.

\bibitem{lu:wu:xiao:xu:2022}
J.~Lu, W.~B. Wu, Z.~Xiao, L.~Xu, Almost sure invariance principle of $\beta$-mixing time series in {H}ilbert space, arXiv:2209.12535 (2022).

\bibitem{dette:kokot:aue}
H.~Dette, K.~Kokot, A.~Aue, {Functional data analysis in the Banach space of continuous functions}, The Annals of Statistics 48 (2020) 1168 -- 1192.

\bibitem{dette:kokot:2022}
H.~Dette, K.~Kokot, Detecting relevant differences in the covariance operators of functional time series: a sup-norm approach, Annals of the Institute of Statistical Mathematics 74(2) (2022) 195--231.

\bibitem{kuelbs:1973}
J.~Kuelbs, The invariance principle for {B}anach space valued random variables, Journal of Multivariate Analysis 3 (1973) 161--172.

\bibitem{dehling:1983}
H.~Dehling, Limit theorems for sums of weakly dependent {B}anach space valued random variables, Z. Wahrsch. Verw. Gebiete 63 (1983) 393--432.

\bibitem{samur:1987}
J.~D. Samur, On the invariance principle for stationary $\varphi$-mixing triangular arrays with infinitely divisible limits, Probability Theory and Related Fields 75 (1987) 245–259.

\bibitem{kuelbs:philipp:1980}
J.~Kuelbs, W.~Philipp, Almost sure invariance principles for partial sums of mixing {B} -valued random variable, Annals of Probability 8 (1980) 1003--1036.

\bibitem{burton:dabrowski:dehling:1986}
R.~M. Burton, A.~R. Dabrowski, H.~Dehling, An invariance principle for weakly associated random vectors, Stochastic Processes and their Applications 23 (1986) 301--306.

\bibitem{bradley:2007}
R.~C. Bradley, Introduction to {S}trong {M}ixing {C}onditions, Vol. 1,2,3, Kendrick Press, 2007.

\bibitem{yao:muller:wang:2005JASA}
F.~Yao, H.-G. M{\" u}ller, J.-L. Wang, Functional data analysis for sparse longitudinal data, Journal of the American Statistical Association 100 (2005) 577--590.

\bibitem{zhang:wang:2016}
X.~Zhang, J.~Wang, From sparse to dense functional data and beyond, The Annals of Statistics 44~(5) (2016) 2281--2321.

\bibitem{dehling:philipp:1982}
H.~Dehling, W.~Philipp, {Almost Sure Invariance Principles for Weakly Dependent Vector-Valued Random Variables}, The Annals of Probability 10~(3) (1982) 689 -- 701.

\bibitem{dedecker:merlevede:2010}
J.~Dedecker, F.~Merlev{\'e}de, On the almost sure invariance principle for stationary sequences of {H}ilbert-valued random variables, Dependence in Probability, Analysis and Number Theory (2010) 157--175.

\bibitem{mies:steland:2023}
F.~Mies, A.~Steland, Sequential gaussian approximation for nonstationary time series in high dimensions, Bernoulli 29 (2023) 3114--3140.

\bibitem{billingsley:1968}
P.~Billingsley, Convergence of {P}robability {M}easures, Wiley, New York, 1968.

\bibitem{janson:kaijser:2015}
S.~Janson, S.~Kaijser, Higher moments of {B}anach space valued random variables, Memoirs of the American Mathematical Society 238 (2015).

\bibitem{panaretos:zemel:2020}
V.~Panaretos, Y.~Zemel, {An Invitation to Statistics in Wasserstein Space}, Springer, 2020.

\bibitem{gibbs:su:2002}
A.~L. Gibbs, F.~E. Su, On choosing and bounding probability metrics, International Statistical Review 70 (2002) 419--435.

\bibitem{wu:2007}
W.~Wu, Strong invariance principles for dependent random variables, The Annals of Probability 35 (2007) 2294--2320.

\bibitem{hormann:kokoszka:2010}
S.~H{\"o}rmann, P.~Kokoszka, Weakly dependent functional data, The Annals of Statistics 38 (2010) 1845--1884.

\bibitem{chen:muller:2012}
K.~Chen, H.-G. M{\"u}ller, Modeling repeated functional observations, Journal of the American Statistical Association 107 (2012) 1599--1609.

\bibitem{berger:hermann:holzmann:2023}
M.~Berger, P.~Hermann, H.~Holzmann, From dense to sparse design: Optimal rates under the supremum norm for estimating the mean function in functional data analysis, arXiv:2306.04550 (2023).

\bibitem{vaart:wellner:1996}
A.~W. van~der Vaart, J.~A. Wellner, {{Weak Convergence and Empirical Processes. With Applications to Statistics}}, Springer Series in Statistics., New York, 1996.

\bibitem{dehling:mikosch:sorensen:2002}
H.~Dehling, T.~Mikosch, M.~S{{\o}}rensen, Empirical {P}rocess {T}echniques for {D}ependent {D}ata, Birkh\"auser, New York, 2002.

\bibitem{powers:stormer:1970}
R.~T. Powers, E.~Stormer, Free states of the canonical anticommutation relations, Communications in Mathematical Physics 16 (1970) 1--33.

\bibitem{acosta:1985}
A.~de~Acosta, On the functional form of {L}evy's modulus of continuity for {B}rownian motion, Zeitschrift f{\"u}r Wahrscheinlichkeitstheorie und verwandte Gebiete 69 (1985) 567--579.

\bibitem{huber:ronchetti:2009}
P.~J. Huber, E.~M. Ronchetti, Robust {S}tatistics, Wiley, Hoboken, 2009.

\bibitem{Liu:Wang:2023}
Z.~Liu, Z.~Wang, Wasserstein convergence rates in the invariance principle for sequential dynamical systems, arXiv:2307.13913 (2023).

\bibitem{berbee:1979}
H.~C.~P. Berbee, Random walks with stationary increments and renewal theory, Mathematisches Centrum, Amsterdam, 1979.

\bibitem{berkes:philipp:1979}
I.~Berkes, W.~Philipp, Almost sure invariance principles for independent and weakly dependent random vectors, Annals of Probability 7 (1979) 29--54.

\bibitem{dehling:1983b}
H.~Dehling, A note on a theorem of berkes and philipp, Zeitschrift für Wahrscheinlichkeitstheorie und Verwandte Gebiete 62 (1983) 39--42.

\bibitem{einmahl:1993}
U.~Einmahl, Toward a general law of the {I}terated {L}ogarithm in {B}anach space, The Annals of Probability 21 (1993) 2012--2045.

\bibitem{ledoux:talagrand:1991}
M.~Ledoux, M.~Talagrand, Probability in {Banach} {Spaces}: {Isoperimetry} and {Processes}, Springer, 1991.

\bibitem{fernique:1975}
X.~Fernique, Regularité des trajectoires des fonctions aléatoires gaussiennes, in: P.~Hennequin (Ed.), Ecole d’Eté de Probabilités de Saint-Flour IV—1974, Vol. 480 of Lecture Notes in Mathematics, Springer, Berlin, Heidelberg, 1975.

\bibitem{chu:stinchcombe:white:1996}
C.-S.~J. Chu, M.~Stinchcombe, H.~White, Monitoring structural change, Econometrica 64 (1996) 1045--1065.

\bibitem{horvath:kokoszka:huskova:steinebach:2003}
L.~Horv{\'a}th, M.~Hu{\v{s}}kov{\'{a}}, P.~Kokoszka, J.~Steinebach, Monitoring changes in linear models, Journal of Statistical Planning and Inference 126 (2003) 225--251.

\bibitem{aue:horvath:huskova:kokoszka:2006}
A.~Aue, L.~Horv{\'a}th, M.~Hu{\v{s}}kov{\'{a}}, P.~Kokoszka, Change--point monitoring in linear models with conditionally heteroskedastic errors, Econometrics Journal 9 (2006) 373--403.

\bibitem{gosmann:kley:dette:2021}
J.~G{\"o}smann, T.~Kley, H.~Dette, A new approach for open-end sequential change point monitoring, Journal of Time Series Analysis 42 (2021) 63--84.

\bibitem{aue:kirch:2023}
A.~Aue, C.~Kirch, The state of cumulative sum sequential change point testing seventy years after {P}age, Biometrika (2023+).

\bibitem{kmt:1975}
J.~Koml{\'o}s, P.~Major, G.~Tusn{\'a}dy, An approximation of partial sums of independent {R}.{V}.'s and the sample {DF.I}, Zeitschrift f{\"u}r Wahrscheinlichkeitstheorie und verwandte Gebiete 32 (1975) 111--131.

\bibitem{kmt:1976}
J.~Koml{\'o}s, P.~Major, G.~Tusn{\'a}dy, An approximation of partial sums of independent {R}.{V}.'s and the sample {DF.II}, Zeitschrift f{\"u}r Wahrscheinlichkeitstheorie und verwandte Gebiete 34 (1976) 33--58.

\bibitem{eberlein:1986}
E.~Eberlein, On strong invariance principles under dependence assumptions, The Annals of Probability 14 (1986) 260--270.

\bibitem{horvath:rice:2014}
L.~Horv{\'a}th, G.~Rice, Extensions of some classical methods in change point analysis, Test 23 (2014) 219--255.

\bibitem{aue:hormann:horvath:huskova:2014}
A.~Aue, S.~H{\"o}rmann, L.~Horv{\'a}th, M.~Hu{\v{s}}kov{\'{a}}, Dependent functional linear models with applications to monitoring structural change, Statistica Sinica 24 (2014) 1043--1073.

\bibitem{kutta:kokoszka:2025}
T.~Kutta, P.~Kokoszka, Monitoring of functional time series, BernoulliAccepted for publication (2025).

\bibitem{gosmann:2020}
J.~G{\"o}smann, New aspects of sequential change point detection, Ph.D. thesis, Ruhr-University Bochum (2020).

\bibitem{kutta:jach:kokoszka:2024}
T.~Kutta, A.~Jach, P.~Kokoszka, Monitoring panels of sparse functional data, Journal of Time Series Analysis (2024).
\newblock \href {https://doi.org/10.1111/jtsa.12796} {\path{doi:10.1111/jtsa.12796}}.

\bibitem{yoshihara:1978}
K.~Yoshihara, Moment inequalities for mixing sequences, Kodai Math. J. 1 (1978) 316--328.

\bibitem{schuster:1969}
E.~F. Schuster, Estimation of a probability density function and its derivatives, The Annals of Mathematical Statistics 40~(4) (1969) 1187--1195.

\bibitem{schuster:yakowitz:1979}
E.~Schuster, S.~Yakowitz, Contributions to the theory of nonparametric regression, with application to system identification, The Annals of Statistics 7 (1979) 139--149.

\bibitem{bhattacharya:1967}
P.~K. Bhattacharya, Estimation of a probability density function and its derivatives, Sankhya: The Indian Journal of Statistics, Series A 29~(4) (1967) 373--382.

\bibitem{aue:hormann:horvath:reimherr:2009}
A.~Aue, S.~H{\"o}rmann, L.~Horv{\'a}th, M.~Reimherr, Break detection in the covariance structure of multivariate time series models, The Annals of Statistics 37 (2009) 4046--4087.

\bibitem{berkes:gabrys:horvath:kokoszka:2009}
I.~Berkes, R.~Gabrys, L.~Horv{\'a}th, P.~Kokoszka, Detecting changes in the mean of functional observations, Journal of the Royal Statistical Society (B) 71 (2009) 927--946.

\bibitem{dette:quanz:2023}
H.~Dette, P.~Quanz, Detecting relevant changes in the spatiotemporal mean function, Journal of Time Series Analysis 44 (2023) 505–532.

\bibitem{stoehr:aston:kirch:2021}
C.~St{\"o}hr, J.~Aston, C.~Kirch, Detecting changes in the covariance structure of functional time series with application to f{MRI} data, Econometrics and Statistics 18 (2021) 44--62.

\bibitem{aue:rice:sonmez:2020}
A.~Aue, G.~Rice, G.~Sonmez, Structural break analysis for spectrum and trace of covariance operators, Environmetrics 31 (2020) e2617.

\bibitem{bradley:1986}
R.~C. Bradley, Basic properties of strong mixing conditions, in: E.~Eberlein, M.~S. Taqqu (Eds.), Dependence in Probability and Statistics, Birkh{\"a}user, Boston, 1986, pp. 165--192.

\bibitem{volkonskii:1959}
V.~Volkonskii, Y.~Rozanov, Some limit theorems for random functions {I}, Theory of Probability \& Its Applications 4 (1959) 178--197.

\end{thebibliography}

\newpage

\normalsize\renewcommand{\baselinestretch}{1.0}
\setcounter{page}{1}

\normalsize\renewcommand{\baselinestretch}{1.0}
\setcounter{page}{1}

{\LARGE $ \quad$ \textbf{Supplementary Information}}

\appendix

\section{Proof of Theorem \ref{thm1} }\label{AppA}

We begin by summarizing the notation introduced in the paper 
and used throughout  the Supplementary Information.  

\begin{center}  \small 
    \begin{tabular}{|c|c|c|}
        \hline
        \textbf{Notation} & \textbf{Definition} & \textbf{Location} \\[1.2ex]
        \hline
        $X_i, X_{i,M}$ &  latent function and  sparse estimator  & Section \ref{sec_sparse}\\[1ex]
        \hline
        $P_N, P_N^{lin}$ &  partial sum process (linearly interpolated)  & eq. \eqref{e:partials},\eqref{e:partialsl} \\[1ex]
        \hline
        $W$ &  Banach-valued Brownian motion  & Theorem \ref{thm1} \\[1ex]
        \hline
        $c$ &  long-run variance kernel  & eq. \eqref{e:lr_var} \\[1ex]
        \hline

        $I_1, I_2$ & indexing intervals for latent functions  & Section \ref{sec_sparse}\\
        &  and sparse estimators & \\[1ex]
        \hline
        $\mathcal{C},\mathcal{C}_0 $ & space of continuous functions & Section \ref{sec23}\\
        & (that vanish at $\pm \infty$) &\\[1ex]
        \hline
        $L^\infty$ & space of bounded functions & Section \ref{sec23}\\[1ex]
        \hline
         $|\cdot|, |\cdot|_2$ & max-norm and Euclidean norm & Section \ref{sec_not}\\[1ex]
        \hline
        $\|\cdot\|$ & supremum norm & Section \ref{sec_not}\\[1ex]
        \hline
        $\pi, \pi_\infty, \pi_2$ & Prokhorov metric  & Section \ref{sec23}\\
        & (w.r.t. sup-/max-norm and  Euclidean norm) & \\[1ex]
        \hline
         $\mathcal{W}_q$ & $q$-Wasserstein distance & Section \ref{sec23}\\[1ex]
        \hline
        $\xi$ & H{\"o}lder constant in $(1/2,1]$  & Assumption \ref{ass_1}\\[1ex]
        \hline
        $\rho$ & constant quantifying size of domain & eq. \eqref{e:rhosig}\\[1ex]
        \hline
        $\varsigma$ & constant quantifying mesh-width & eq. \eqref{e:rhosig}\\[1ex]
        \hline
        $\tau$ & constant quantifying convergence in $\pi_\infty$ & Theorem \ref{thm1}\\[1ex] \hline
        $\kappa$ & constant of exponential tail decay & Assumption \ref{ass_1b}\\[1ex]
        \hline
        $\mathcal{G}_N, (\lambda_{\mathcal{G}_N}, u_{\mathcal{G}_N})$ & grid on $[0,1]\times [-N^\rho, N^\rho]$, & after eq. \eqref{e:rhosig}\\
        &
         closest gridpoint to $(\lambda, u)$ &\\[1ex]
        \hline
        $f^{dis}$ &  function discretized on $\mathcal{G}_N$ & after eq. \eqref{e:rhosig}\\[1ex]
        \hline
        $d_{1/2, \xi}$ & metric on $[0,1] \times [-L,L]$ & eq. \eqref{e:dmet}\\[1ex]
        \hline
    \end{tabular}
\end{center} 

\subsection{Preliminaries} \label{App_Prel}

Throughout the Supplementary Information, we will denote by $C>0$ a generic, positive constant that can change from one equation to the next. In all of of the Supplementary Information $C$, as well as the constants implied by Landau symbols, will be only dependent on the constants in Assumptions \ref{ass_1} and \ref{ass_1b}, but independent of $M$ or $N$.
As before, we denote by $|\cdot|$ the maximum norm on $\mathbb{R}^d$ for any $d \in \mathbb{N}$ and by $\|\cdot\|$ the sup-norm of functions in one or multiple variables. By $|\cdot|_2$ we will refer to the Euclidean norm on $\mathbb{R}^d$. \\
For a non-empty interval $I \subset \mathbb{R}$ we will denote by $L^\infty([0,1] \times I, \mathbb{R}^d)$ the space of bounded functions mapping from  $[0,1] \times I$ into $\mathbb{R}^d$ (see Section \ref{sec23}).
To prove weak convergence, we will frequently use the Prokhorov metric, defined in  Section \ref{sec23} for general metric spaces.
Notice that in the definition of the Prokhorov metric (see \eqref{e:def_Prokhorov}), 
the value of $\pi$ depends on the metric  of the underlying space and even 
two metrics that create the same topology might lead to different notions of 
$\pi$. In order to clarify which metric is used, we will subsequently denote 
by $\pi_\infty$ the Prokhorov metric on 
$(L^\infty([0,1] \times I, \mathbb{R}^d), 
\|\cdot \|)$ and $(\mathbb{R}^d, |\cdot|)$ (w.r.t. the sup- or max-norm) 
and by $\pi_2$ the Prokhorov metric on $(\mathbb{R}^d, |\cdot|_2)$ 
(w.r.t. the Euclidean norm). 

For our subsequent empirical process arguments, we introduce the notion of 
\textit{packing numbers}. For a subset $\mathcal{M} \subset \mathbb{R}^d$, 
a metric $d_{\mathcal{M}}$ and any $\chi >0$, the packing number
$D_{d_\mathcal{M}}(\mathcal{M},\chi)
$ is the maximum number of $\chi$-separated points in $\mathcal{M}$. Bounds for packing and covering numbers are standard tools in the theory of empirical processes; see Definition 2.2.3 in \cite{vaart:wellner:1996}. In the following, we will focus on a special case, 
where $\mathcal{M}=[0,1]\times [-L,L]$ and for $\xi \in (1/2,1]$ (the same $\xi$ as in Assumption \ref{ass_1}, part ii)
\begin{align}\label{e:dmet}
d_{1/2, \xi}((\lambda,u), (\lambda',u'))=\max(|\lambda-\lambda'|^{1/2}, |u-u'|^\xi).
\end{align}
\medskip

\subsection{Proof of Theorem \ref{thm1}} 

\textbf{The setting}  We prove the most difficult part of the theorem, where the domain is infinite, i.e. $I_1=I_2=\mathbb{R}$. The case of a growing domain follows directly from this proof, by copying all steps except for the tail truncation. The case of a fixed domain follows by similar, simpler arguments and is therefore omitted. We will focus on the case where $d=1$, since the multivariate case follows analogously.
Throughout the proofs, we assume that all random functions $X_{i,M}, X_i$ are already centered. Otherwise, one can replace $X_{i,M}$  by $X_{i,M}-\mathbb{E}X_{i,M}$ and $X_i$ by $X_{i}-\mathbb{E}X_{i}$ and check that after such a replacement Assumptions \ref{ass_1} and \ref{ass_1b} still remain valid.
Now, we consider the process
\begin{align} \label{e:PNred}
    P_{N}(\lambda,u) := \frac{1}{\sqrt{N}} \sum_{i=1}^{\lfloor \lambda N \rfloor} X_{i,M_N}(u), \quad \lambda \in [0,1], \,\, u \in \mathbb{R}.
\end{align}
In the following, let $0<\rho<\varsigma<1$ be numbers that are sufficiently small, where the precise conditions on how small will be determined in the subsequent proofs. Importantly, these numbers do not depend on $N$.\\
\textbf{Discretizations}
Next, we define the discretization of functions in $L^\infty([0,1] \times \mathbb{R}, \mathbb{R})$. Therefore, we introduce an evenly spaced grid $\mathcal{G}_N$ on $[0,1]\times [-N^{\rho}, N^\rho]$ that satisfies the two conditions on mesh width and number of gridpoints
\begin{align} \label{e:rhosig}
\max_{x \in [0,1]\times [-N^{\rho}, N^\rho]} \min_{y \in \mathcal{G}_N} |x-y| \le N^{-\varsigma}, \quad |\mathcal{G}_N| \le C N^{\rho+2\varsigma}.
\end{align}
It is easy to see that the construction of such a grid is indeed possible, using an even grid of $\lceil N^{\varsigma}\rceil$ gridpoints on $[0,1]$ and of  $\lceil 2N^{\rho+\varsigma}\rceil$ gridpoints on $[-N^{\rho},N^{\rho}]$. Now for a function $f \in L^\infty([0,1]\times \mathbb{R}, \mathbb{R})$, we define its discretization $f^{dis}$. Roughly speaking, $f^{dis}$ will be a function defined on rectangles inside  $[0,1]\times[-N^{\rho},N^{\rho}]$, with $f^{dis}(x)=f(x_{\mathcal{G}_N})$ where $x_{\mathcal{G}_N}$ denotes the closest gridpoint to $x$. Outside the compact $[0,1]\times[-N^{\rho},N^{\rho}]$ we set $f^{dis}$ to $0$. Formally, we can define $f^{dis}$ by the two properties:
\begin{itemize}
    \item[Di)] $f^{dis}(\lambda,u)=0$ if $|u|>N^{\rho}$.
    \item[Dii)] If $|u|<N^{\rho}$, $f^{dis}(\lambda,u) := f(\lambda_{\mathcal{G}_N}, u_{\mathcal{G}_N})$, where the pair $(\lambda_{\mathcal{G}_N}, u_{\mathcal{G}_N})$ denotes a closest gridpoint to $(\lambda,u)$ (w.r.t. the maximum distance). If two pairs of gridpoints have equal distance to $(\lambda,u)$, we choose the one with smaller $\lambda_{\mathcal{G}_N}$ and if these values are tied, the one with the smaller value of $u_{\mathcal{G}_N}$.
\end{itemize}
Notice that the function $f^{dis}$ implicitly depends on $N$ via the grid $\mathcal{G}_N$.
We can now define $P_N^{dis}$ as the discretized version of $P_N$ and $W^{dis}$
as the discretized version of the centered Brownian motion $W$
defined in Theorem \ref{thm1}. In order to prove Theorem \ref{thm1}, we want to show that $\pi_\infty(P_N, W_N$) vanishes at polynomial speed in $N$. Therefore, we prove for a sufficiently small number $\tau>0$ the following three statements:
\begin{itemize}
    \item[1)] Step 1:
    \begin{align} \label{e:rate1}
        \pi_\infty(P_N, P_N^{dis})= \mathcal{O}(N^{-\tau}).
    \end{align}
    \item[2)] Step 2:
    \begin{align} \label{e:rate2}
    \pi_\infty(P_N^{dis}, W^{dis})= \mathcal{O}(N^{-\tau}).
    \end{align}
    \item[3)] Step 3:
    \begin{align} \label{e:rate3}
    \pi_\infty(W^{dis}, W)= \mathcal{O}(N^{-\tau}).
    \end{align}
\end{itemize}
Each of these steps is non-trivial. For the first step, 
we need to demonstrate that $P_N$ and $P_N^{dis}$ are uniformly \
close with high probability. This requires finite sample bounds from 
empirical process theory to control the distance of stochastic processes 
on increasing domains (see, e.g.,  \cite{vaart:wellner:1996}).  
The second step is the most challenging one. Here, we exploit that 
the discretized processes live on a finite dimensional subspace of $L^\infty$ 
with slowly growing dimension. We employ normal approximations on vector 
spaces with increasing dimension for dependent random variables. 
For such results, we adapt tools developed in 
\cite{dehling:mikosch:sorensen:2002}. 
We also require some sophisticated bounds for the Prokhorov distance 
between high-dimensional Gaussian random variables, to transition 
from the finite sample variance to the asymptotic variance. In the last step, 
we exploit regularity properties of the Banach-valued Brownian motion, 
e.g. to control its modulus of continuity. To make the proofs of these 
steps easier to read, we have partitioned them into the following Lemmas.

\begin{lemma} \label{lemA1}
    Under the assumptions of Theorem \ref{thm1}, it holds for some $\tau>0$ and some $C'_1, C'_2>0$ only depending on the constants in Assumptions \ref{ass_1} and \ref{ass_1b} that
    \[
    \mathbb{P}(\|P_N^{dis}- P_N\|\ge C'_1 N^{-\tau})\le C'_2 N^{-\tau}.
    \]
\end{lemma}
Notice that according to Theorem 2.13 from \cite{huber:ronchetti:2009}  this result directly implies the approximation rate
in \eqref{e:rate1}. We also point out (for now and for later use of this theorem) that this direction of the Theorem does not require separability of the underlying space.

\begin{proof} Recall that according to the definition of the discretizations,  $P_N^{dis}(\lambda,u)=0$ whenever $|u|>N^\rho$. As a consequence, we can prove the Lemma, by verifying the two following statements.
\begin{itemize}
    \item[] Part i):
    $
    \mathbb{P}(\sup_{\lambda \in [0,1],|u| >N^{\rho}}|  P_N(\lambda,u)|\ge C'_1 N^{-\tau}) \le  (C'_2/2) N^{-\tau}.$
    \item[] Part ii):  $\mathbb{P}(\sup_{\lambda \in [0,1],u \in [-N^{\rho}, N^{\rho}]}|P_N(\lambda,u)-P_N^{dis}(\lambda,u)|\ge C'_1 N^{-\tau})\le  (C'_2/2) N^{-\tau} $.
\end{itemize}
\textbf{Part i):} Using Markov's inequality we obtain
\begin{align*}
& \mathbb{P}\Big(\sup_{\lambda \in [0,1],|u| >N^{\rho}}|  P_N(\lambda,u)|\ge C'_1 N^{-\tau}\Big)
\le  (C'_1)^{-1} N^{\tau} \mathbb{E}\bigg[\sup_{\lambda \in [0,1],|u| >N^{\rho}}|P_N(\lambda, u)| \bigg].
\end{align*}
Next, we employ the definition of the process $P_N$ (in \eqref{e:PNred}) to upper bound the right side by
\begin{align*} (C'_1)^{-1} N^{\tau-1/2}  \sum_{i=1}^{N} \mathbb{E}\big[\sup_{|u| >N^{\rho}}|X_{i,M_N}(u)|\big]
\le  C N^{\tau+1/2} \exp(-N^{\rho \kappa}) \le C N^{-\tau}.
\end{align*}
In the first inequality we have used the exponentially decaying tails of the sparse estimators from Assumption \ref{ass_1b}, part iii).\\
\textbf{Part ii):} By definition of the discretization, we have
\begin{align*}
    &\sup_{\lambda \in [0,1],u \in [-N^{\rho}, N^{\rho}]}|P_N(\lambda,u)-P_N^{dis}(\lambda,u)| \\
    =& \sup_{\lambda \in [0,1],u \in [-N^{\rho}, N^{\rho}]}|P_N(\lambda,u)-P_N(\lambda_{\mathcal{G}_N},u_{\mathcal{G}_N})|,
\end{align*}
where as before $(\lambda_{\mathcal{G}_N},u_{\mathcal{G}_N})$ is a closest gridpoint to $(u, \lambda)$ (for details see the construction in Dii).
In particular, this implies $(\lambda_{\mathcal{G}_N},u_{\mathcal{G}_N})$ is a gridpoint s.t. $|u-u_{\mathcal{G}_N}|, |\lambda-\lambda_{\mathcal{G}_N}|\le N^{-\varsigma}$. We hence get
 \begin{align*}
     & \mathbb{P}\Big(\sup_{\lambda \in [0,1],u \in [-N^{\rho}, N^{\rho}]}|P_N(\lambda,u)-P_N^{dis}(\lambda,u)| \ge C'_1 N^{-\tau}\Big) \\
     \le & \mathbb{P}\bigg(\max_{(\bar \lambda,  \bar u) \in \mathcal{G}_N}\sup_{\substack{\lambda \in [0,1], |u|\le N^{\rho} \\ |(\lambda, u)-(\bar \lambda, \bar u)|\le N^{-\varsigma}}}|P_N(\lambda,u)-P_N(\bar \lambda,  \bar u)| \ge C'_1 N^{-\tau}\bigg).
 \end{align*}
Now, we invoke the definition of the linearized process $P_{N}^{lin}$ (defined in \eqref{e:partialsl}), which can be explicitly written as
\begin{align} \label{e:lin}
P_{N}^{lin}(\lambda,u):= \frac{1}{\sqrt{N}} \sum_{i=1}^{\lfloor \lambda N \rfloor} X_{i,M_N}(u)+\frac{\lambda-(\lfloor N \lambda \rfloor/N ) }{\sqrt{N} } X_{\lceil N \lambda \rceil ,M_N}(u).
\end{align}
We transition to this continuous version of $P_N$, because we want to subsequently use entropy bounds that are only available for continuous processes (notice that $P_N$ is already continuous in $u$). Now, consider a pair of one gridpoint $(\bar u, \bar \lambda)$ and one point $(u, \lambda)$ with $|(\bar u, \bar \lambda)-(u, \lambda)|\le N^{-\varsigma}$. We can find natural numbers $k, \bar k$ such that $\lfloor \lambda N \rfloor =k$ and $\lfloor \bar \lambda N \rfloor =\bar k$. Then, it holds that
\[
P_N(\lambda,u) - P_N(\bar \lambda, \bar u) = P_N^{lin}(k/N, \bar u) -P_N^{lin}(\bar k/N, \bar u)
\]
and by construction that
\[
|k/N-\lambda|\le 1/N, \quad |\bar k/N-\bar \lambda|\le 1/N,
\]
which is $\le N^{-\varsigma}$. Consequently, the following upper bound holds as the supremum that we take in the second line is over a larger set of numbers than in the first line:
\begin{align} \label{e:omeg_N1}
    & \max_{(\bar \lambda,  \bar u) \in \mathcal{G}_N}\sup_{\substack{\lambda \in [0,1], |u|\le N^{\rho} \\ |(\lambda, u)-(\bar \lambda, \bar u)|\le N^{-\varsigma}}}|P_N(\lambda,u)-P_N(\bar \lambda,  \bar u)| \\[1ex]
    \le &  \sup_{\substack{|(\lambda, u)-( \lambda',  u')| \le  2N^{-\varsigma}\\  \lambda, \lambda'\in [0,1] , |u|,|u'| \le N^\rho}}|P_N^{lin}(\lambda,u)-P_N^{lin}( \lambda',   u')|=: \omega_N(P_N^{lin}). \nonumber
\end{align}
Combining the above bound with Chebychev's inequality  yields
\begin{align} \label{e:omeg_N}
& \mathbb{P}\Big(\sup_{\lambda \in [0,1],u \in [-N^{\rho}, N^{\rho}]}|P_N(\lambda,u)-P_N^{dis}(\lambda,u)| >C'_1 N^{-\tau}\Big)\\
\le & \mathbb{P}\Big(\omega_N(P_N^{lin}) >C'_1 N^{-\tau}\Big)  \le C  N^{4\tau}\mathbb{E}|\omega_N(P_N^{lin})|^4.\nonumber
\end{align}
To control the right side of \eqref{e:omeg_N}, we can analyze the modulus of continuity $\mathbb{E}|\omega_N(P_N^{lin})|^4$. Lemma \ref{lem:modcont} implies that there exists a   $\zeta>0$ that is independent of $N$ such that
\begin{align} \label{e:rateequ}
  \{\mathbb{E}|\omega_N(P_N^{lin})|^4\}^{1/4} \le C N^{-\zeta}.
\end{align}
In view of \eqref{e:omeg_N} choosing $\tau<\zeta$ then yields the desired result.

\end{proof}

\begin{lemma} \label{lem:modcont}
    Under the conditions of Lemma \ref{lemA1} and with the definition of $\omega_N(P_N^{lin})$ in \eqref{e:omeg_N1}, the bound \eqref{e:rateequ} holds for a sufficiently small $\zeta>0$.
\end{lemma}

\begin{proof}
For the proof of this Lemma, we apply Theorem 2.2.4 in \cite{vaart:wellner:1996}, which provides a moment bound for the modulus of continuity of a stochastic process. A condition of named result is the proof of continuity of
$
\{\mathbb{E}[|P_{N}^{lin}(\lambda, u)-P_{N}^{lin}(\lambda', u')|^4]\}^{1/4}
$
w.r.t. to the metric $d_{1/2,\xi}$ and tuples $(\lambda, u)$, $(\lambda', u')$. We will now derive a bound for this object.
\\ W.l.o.g. we can assume that $\lambda>\lambda'$. In principle, we now have to consider two different cases, where either $\lfloor \lambda' N \rfloor <\lfloor \lambda N \rfloor $ or $\lfloor \lambda' N \rfloor = \lfloor \lambda N \rfloor $. For the purpose of illustration, we only focus on the first, more difficult case here. Using the representation of the linearly interpolated process $P_{N}^{lin}$ in eq. \eqref{e:lin} and the triangle inequality yields the decomposition
\begin{align*}
     &|P_{N}^{lin}(\lambda, u)-P_{N}^{lin}(\lambda', u')| \le  R_{1,N}+R_{2,N}\\
     R_{1,N}:= &|P_{N}^{lin}(\lambda, u)-P_{N}^{lin}(\lambda', u)|, \quad
       R_{2,N}:=|P_{N}^{lin}(\lambda', u)-P_{N}^{lin}(\lambda', u')|.
\end{align*}
 Using this decomposition and the upper bound $(|a|+|b|)^n \le 2^{n-1}(|a|^{n}+|b|^{n})$, we obtain
\[
\mathbb{E}[|P_{N}^{lin}(\lambda, u)-P_{N}^{lin}(\lambda', u')|^4] \le 8 \Big( \mathbb{E}[R_{1,N}^4]+ \mathbb{E}[R_{2,N}^4]\Big),
\]
and we can upper bound the fourth moments on the right separately. We begin with $R_{2,N}$. Theorem 3 in \cite{yoshihara:1978} (with parameter choices $m=4, \delta=J-4$), which is a moment inequality for sums of strongly mixing random variables, implies that
\begin{align} \label{e:yoshapp}
& \mathbb{E}\Big|\frac{1}{\sqrt{N}} \sum_{i=1}^{\lfloor \lambda' N \rfloor} \frac{X_{i,M_N}(u)-X_{i,M_N}(u')}{|u-u'|^\xi}\\
& \quad+\frac{\lambda'-(\lfloor N \lambda' \rfloor/N ) }{\sqrt{N} } \frac{X_{\lceil N \lambda' \rceil ,M_N}(u)-X_{\lceil N \lambda' \rceil ,M_N}(u')}{|u-u'|^\xi} \Big|^4\le C. \nonumber
\end{align}
The cited theorem has two conditions. First, it requires (in our case) that
\[
\mathbb{E}\Big|\frac{X_{i,M_N}(u)-X_{i,M_N}(u')}{|u-u'|^\xi}\Big|^J\le C,
\]
which holds by Assumption \ref{ass_1}, Condition iii) with $C=C_2$ uniformly in $i$ and $N$. The second
 condition of the cited theorem is $\sum_n  n \alpha(n)^{(J-4)/J}<\infty$ which holds according to Assumption \ref{ass_1b}, part iii). More precisely, with the constant $\nu>0$ from that assumption it follows that
 \begin{align} \label{e:yosh1}
 \sum_n  n \alpha_{X,\delta,M_N}(n)^{(J-4)/J} \le C  \sum_n  n n^{-\nu (J-4)/J} = C  \sum_n   n^{1-\nu (J-4)/J}.
 \end{align}
A small calculation shows that for any $J>4$ (which holds by assumption) we have
\[
\nu \frac{J-4}{J}>\Big(3+\frac{9}{J-4}\Big)\frac{J-4}{J} >2,
\]
and this implies that the right side of \eqref{e:yosh1} is bounded. Notice that we have used here the mixing rate in Assumption \ref{ass_1}, parts i) and ii). We have thus shown \eqref{e:yoshapp}, which in turn implies
\[
\mathbb{E} |R_{2,N}|^4 \le  C |u-u'|^{4 \xi}.
\]
Similar arguments can be used to show that for
$R_{1,N}$ it holds that
\[
\mathbb{E}|R_{1,N}|^4\le C  |\lambda-\lambda'|^{2}.
\]
Together, this demonstrates that
\begin{align} \label{e:QboundP}
& \{\mathbb{E}|P_{N}^{lin}(\lambda, u)-P_{N}^{lin}(\lambda', u')|^4 \}^{1/4} \le C \max(|\lambda-\lambda'|^{1/2}, |u-u'|^\xi)\\
= & C \cdot d_{1/2,\xi}((\lambda,u),(\lambda',u')). \nonumber
\end{align}
Here $d_{1/2,\xi}$ is the metric defined in \eqref{e:dmet}. Notice that by definition of this metric
\begin{align} \label{e:imp}
     &|(u, \lambda)-(u',\lambda')| \le 2N^{-\varsigma}  \\\Rightarrow &   d_{1/2,\xi}((u,\lambda), (u',\lambda'))\le \max(2N^{-\varsigma/2}, 2N^{-\varsigma\xi})
    \le  2N^{-\varsigma/2}. \nonumber
\end{align}
Let us hence define the index set
\begin{align*}
\mathcal{E}_N:= &\{((\lambda, u), (\lambda', u')): \lambda, \lambda' \in [0,1], |u|, |u'|<N^\rho,\\
& d_{1/2,\xi}((u,\lambda), (u',\lambda'))\le 2 N^{-\varsigma/2}\}.
\end{align*}
Recall the definition of the modulus of continuity in \eqref{e:omeg_N1}. By the implication \eqref{e:imp} it follows that
\[
\omega_N(P_N^{lin}) \le \sup_{((\lambda, u),( \lambda',  u')) \in \mathcal{E}_N}|P_N^{lin}(\lambda,u)-P_N^{lin}( \lambda',   u')|.
\]
The inequality holds because we have made the index set over which we take the supremum larger.\\
Next, we note a bound for the packing number of $[0,1]\times [-N^\rho, N^\rho]$ w.r.t. to the metric $d_{1/2,\xi}$ (recall the definition at the beginning of our Supplementary Information, right above eq. \eqref{e:dmet}). According to \cite{vaart:wellner:1996} (below their Definition 2.2.3) it suffices to prove a bound for the covering number, since packing and covering number are proportional to each other.
It is a standard exercise to show that the covering number of  $[0,1]\times [-N^\rho, N^\rho]$ w.r.t. the metric $d_{1/2,\xi}$ and radius $x$ can be bounded by $C x^{2+1/\xi}N^\rho$
and thus we obtain
\[
D_{d_{1/2,\xi}}([0,1]\times [-N^\rho, N^\rho], x) \le C x^{2+1/\xi}N^\rho.
\]
Now, using Theorem 2.2.4 in \cite{vaart:wellner:1996} (in the second inequality) shows that for any $\eta \in (0,1]$
\begin{align*}
& \{\mathbb{E}|\omega_N(P_N^{lin})|^4\}^{1/4} \le \Big\{\mathbb{E}\Big[\sup_{((\lambda, u),( \lambda',  u')) \in \mathcal{E}_N}|P_N^{lin}(\lambda,u)-P_N^{lin}( \lambda',   u')|^4\Big]\Big\}^{1/4}\\
\le & C \int_0^{\eta} \{D_{d_{1/2,\xi}}([0,1]\times [-N^\rho, N^\rho], x)\}^{1/4} dx \\
&+ C N^{-\varsigma/2} D_{d_{1/2,\xi}}([0,1]\times [-N^\rho, N^\rho], \eta)^{1/2}\\
\le &   C \int_0^{\eta}   x^{-(2+1/\xi)/4} N^{\rho/4} dx + C N^{-\varsigma /2}   \eta^{-(2+1/\xi)/2} N^{\rho/2}\\
\le &  C \Big( \eta^{1-(2+1/\xi)/4}N^{\rho/4}+C N^{-\varsigma/2}   \eta^{-2} N^{\rho/2}\Big).
\end{align*}
Now, choose $\eta = N^{-\rho'}$ for some $\rho'>0$. We can choose $\rho'>0$ in such a way that both terms on the right are (polynomially) decaying in $N$. This claim is easy to check, if we recall that $1-(2+1/\xi)/4<0$ by assumption (since $\xi>1/2$) and if we choose
\[
0<\rho << \rho'<<\varsigma<<1.
\]
   Now, by our above arguments, we have for a sufficiently small number, say $\zeta>0$,  that $\{\mathbb{E}|\omega_N(P_N^{lin})|^4\}^{1/4}=\mathcal{O}(N^{-\zeta})$, hence we obtain \eqref{e:rateequ}. This completes the proof. 
\end{proof}

\begin{lemma} \label{lemA2}
    Suppose that the assumptions of Theorem \ref{thm1} are met and that $\rho>0$ is sufficiently small. Then, there exists a small enough $\tau>0$ and some large enough $C>0$, such that the following approximation holds
    \begin{align*}
        \pi_\infty(P_N^{dis}, W_N^{dis}) \le C N^{-\tau}.
    \end{align*}
\end{lemma}

\begin{proof}
    Review the definition of $P_N^{dis}$ as a bounded random function, that takes constant values on a finite number of  $|\mathcal{G}_N|\le C N^{\rho+2\varsigma}$ rectangular tiles for $(\lambda, u) \in [0,1] \times [-N^\rho, N ^\rho]$ and is $0$ for $|u|>N^\rho$ (see Di) and Dii)). By construction, for any fixed $N$, the space of functions
    \[
    Dis(N):= \{f^{dis}: f \in L^\infty([0,1]\times \mathbb{R}, \mathbb{R})\}
    \]
    is a finite dimensional subspace of $L^\infty([0,1]\times \mathbb{R})$ and it is naturally isomorphic to the space $(\mathbb{R}^{|\mathcal{G}_N|}, |\cdot|)$ (where $|\cdot|$ is, as always, the maximum norm) using the map
    \[
    \Vec{{}}:   Dis(N) \ni f \mapsto \Vec{f}:=(f(x))_{x \in \mathcal{G}_N}.
    \]
    For ease of notation, we will subsequently define for a gridpoint $x \in \mathcal{G}_N$ the corresponding dimension $d(x)$ under this isomorphism with
    \[
    f(x) = \Vec{f}[d(x)].
    \]
    By definition of the Prokhorov metric (and because the map $\,\Vec{\cdot}\,\,$ is norm preserving) it holds for two random variables $X,Y$ in $Dis(N)$ that
    \[
    \pi_\infty(X,Y) = \pi_\infty(\Vec{X},\Vec{Y}).
    \]
    Consequently it holds that
    \[
    \pi_\infty(P_N^{dis}, W_N^{dis})= \pi_\infty(\Vec{P}_N^{dis}, \Vec{W}_N^{dis})
    \]
    and we can thus derive the desired bound for the left side by bounding the right side.
    Let us now componentwise for $x = (\lambda,u) \in \mathcal{G}_N$ define the random vector
    \begin{align} \label{e:vecv}
    \vec{v}_i[d(x)] := X_{i,M_N}(u) \cdot \mathbb{I}\{1 \le i \le \lfloor \lambda N \rfloor\}.
    \end{align}
We will use a Gaussian approximation from  Lemma \ref{lemdehl} 
for $\frac{1}{\sqrt{N}}\sum_{i=1}^N \vec{v}_i$ to obtain the desired result. In this Lemma we set $\eta_1=J-3$.
    Let us denote by $\Sigma_N$ the covariance matrix of the sum $\frac{1}{\sqrt{N}}\sum_{i=1}^N \vec{v}_i$. In order to apply the approximation we have to check its two assumptions.
    \begin{itemize}
        \item[1)] There exists a $C>0$ s.t.
        \[
        \max_i\mathbb{E}|\vec{v}_i|^{J} \le \max_{i,M} \mathbb{E}\|X_{i,M}\|^{J}\le C.
        \]
        The last inequality  follows by combining Conditions ii) and iii) of Assumption \ref{ass_1}.
        \item[2)] The vectors $\vec{v}_1,\ldots,\vec{v}_N$ are strongly mixing. Indeed, since they are measurable transforms of $X_{i,M_N}$ (i.e., of $X_{i}$ and $\delta_{M_N}$) they inherit the decay of the mixing coefficients implied by Assumption \ref{ass_1}  Conditions i)  and ii), which means that
                \[
        \alpha_{v,N}(n) \le C n^{-\nu}, \qquad \textnormal{with} \qquad        \nu > 3+\frac{9}{J-3}.
        \]
    \end{itemize}
    Lemma \ref{lemdehl} now implies
    \begin{align} \label{e:Berry1}
        \pi_\infty(\Vec{P}_N^{dis}, \mathcal{N}(0, \Sigma_N))\le C N^{-1/20} |\mathcal{G}_N|^3 \le C N^{-1/20+3(\rho+2\varsigma)},
    \end{align}
    which tends to $0$ at polynomial speed if both $\rho, \varsigma$ are small enough.
    Next, the triangle inequality for the Prokhorov metric implies that
    \begin{align} \label{e:pidecmain}
        &\pi_\infty(P_N^{dis}, W_N^{dis}) =  \pi_\infty(\Vec{P}_N^{dis}, \Vec{W}_N^{dis})  \\
        \le & \pi_\infty(\Vec{P}_N^{dis}, \mathcal{N}(0, \Sigma_N)) + \pi_\infty(\Vec{W}_N^{dis}, \mathcal{N}(0, \Sigma_N)) \nonumber \\
        =& \mathcal{O}(N^{-1/20+3(\rho+2\varsigma)})+\pi_\infty(\Vec{W}_N^{dis}, \mathcal{N}(0, \Sigma_N)). \nonumber
    \end{align}
     In view of \eqref{e:pidecmain} , the remaining step to show the decay of $\pi(P_N^{dis}, W_N^{dis})$ is to prove a polynomial decay rate of $\pi(\Vec{W}_N^{dis}, \mathcal{N}(0, \Sigma_N))$ in $N$.  Therefore, notice that the vector $\Vec{W}_N^{dis}$ is a centered Gaussian and can thus be characterized by a covariance matrix $\tilde \Sigma_N$. We can specify this matrix by considering a pair of gridpoints $x=(\lambda,u), x' =(\lambda',u')$ in $ \mathcal{G}_N$ and noticing that
     \[
     \tilde \Sigma_N[d(x),d(x')] := \min(\lambda, \lambda') \cdot c(u,u'),
     \]
    where $c$ is again the covariance kernel of the Brownian motion $W$ defined in \eqref{e:lr_var}. Now, let us calculate for the fixed gridpoints $x,x'$ the corresponding entry $ \Sigma_N[d(x),d(x')] $.  For this calculation let us again assume that $\lambda'<\lambda$. Recall the constant  $\gamma>0$  from Assumption \ref{ass_1} part iii) and notice that we can assume w.l.o.g. that $\gamma<1$. In this proof, we will assume that $N^{-\gamma/2}<\lambda'$. The case where this is not so is much easier (essentially one Cauchy-Schwarz inequality) and is therefore omitted. We can now decompose the covariance of the sum $\frac{1}{\sqrt{N}}\sum_{i=1}^N \vec{v}_i$ (where $\vec{v}_i$ is defined in \eqref{e:vecv}).
    \begin{align*}
          \Sigma_N[d(x),d(x')] = &\mathbb{E}\Big[\Big(\frac{1}{\sqrt{N}} \sum_{i=1}^{\lfloor \lambda N\rfloor }X_{i,M_N}(u)\Big)\Big(\frac{1}{\sqrt{N}} \sum_{i=1}^{\lfloor \lambda' N\rfloor }X_{i,M_N}(u')\Big)\Big]\\
         = &  R_1 + R_2+R_3, \\
         R_1 := & \frac{1}{N}\sum_{|i-j|=0, i,j \le \lfloor \lambda' N\rfloor} \mathbb{E}[X_{i,M_N}(u)X_{j,M_N}(u')]\\
         &+ \frac{1}{N}\sum_{|i-j|=1, i,j \le \lfloor \lambda 'N\rfloor}  \mathbb{E}[X_{i,M_N}(u)X_{j,M_N}(u')]+\ldots\\
         & + \frac{1}{N}\sum_{|i-j|=\lfloor N^{\gamma/2}\rfloor-1, i,j \le \lfloor \lambda' N\rfloor}  \mathbb{E}[X_{i,M_N}(u)X_{j,M_N}(u')]\\
         R_2 := & \frac{1}{N}\sum_{|i-j|\ge \lfloor N^{\gamma/2}\rfloor, i,j \le \lfloor \lambda' N\rfloor}  \mathbb{E}[X_{i,M_N}(u)X_{j,M_N}(u')]\\
         R_3 := &\mathbb{E}\Big[\Big(\frac{1}{\sqrt{N}} \sum_{i=1}^{\lfloor \lambda' N\rfloor }X_{i,M_N}(u)\Big)\Big(\frac{1}{\sqrt{N}} \sum_{i=\lfloor \lambda' N\rfloor +1}^{\lfloor \lambda N\rfloor }X_{i,M_N}(u')\Big)\Big].
    \end{align*}
    In the following, $R_1$ will be the dominating term, while $R_2,R_3$ are negligible. To analyze $R_1$ notice that
    \begin{align*}
       & \mathbb{E}[X_{i,M_N}(u)X_{j,M_N}(u')]\\
       =& \mathbb{E}[X_i(u) X_{j,M_N}(u')]+\mathbb{E}[(X_{i,M_N}(u)-X_i(u))X_{j,M_N}(u')] \\
        =& \mathbb{E}[X_i(u) X_j(u')] + \mathbb{E}[X_i(u) (X_{j,M_N}(u')-X_j(u'))]\\
        &+\mathbb{E}[(X_{i,M_N}(u)-X_i(u))X_{j,M_N}(u')] \\
        =& \mathbb{E}[X_1(u) X_{|i-j|+1}(u')] + \mathcal{O}( \sup_{i}\{\mathbb{E} \|X_i-X_{i,M_N}\|^2\}^{1/2} )\\
        =& \mathbb{E}[X_1(u) X_{|i-j|+1}(u')]  + \mathcal{O}(N^{-\gamma}),
    \end{align*}
    where we have used the weak stationarity of the latent functions, together with the approximation rate in Assumption \ref{ass_1} part iii). As a consequence, we observe that
    \begin{align*}
        & R_1 =  \frac{\lfloor \lambda' N \rfloor}{N}\mathbb{E}[X_1(u) X_{1}(u')]\\
        &+\frac{2\lfloor \lambda' N \rfloor}{N} \sum_{i=1}^{\lfloor N^{\gamma/2}\rfloor} \Big\{\Big( 1-\frac{i}{N}\Big)\mathbb{E}[X_1(u) X_{i+1}(u')] \Big\}+\mathcal{O}(N^{-\gamma/2}) \\
        = & \lambda' \mathbb{E}[X_1(u) X_{1}(u')]+2\lambda' \sum_{i=1}^{\lfloor N^{\gamma/2}\rfloor} \Big\{\Big( 1-\frac{i}{N}\Big)\mathbb{E}[X_1(u) X_{i+1}(u')] \Big\}\\
        &+\mathcal{O}(N^{-\gamma/2}) + \mathcal{O}(N^{-1+\gamma/2}) \\
        = & \lambda' \mathbb{E}[X_1(u) X_{1}(u')] + 2\lambda' \sum_{i=1}^\infty
        \mathbb{E}[X_1(u) X_{i+1}(u')] \\
        &-2\lambda' \sum_{i=\lfloor N^{\gamma/2}\rfloor+1}^\infty
        \mathbb{E}[X_1(u) X_{i+1}(u')] +\mathcal{O}(N^{-\gamma/2}) + \mathcal{O}(N^{-1+\gamma}).
    \end{align*}
    Now, combining the mixing rate in Assumption \ref{ass_1} part i) with Lemma 3.11 in \cite{dehling:mikosch:sorensen:2002}, we observe that
    \begin{align*}
        \mathbb{E}[X_1(u) X_{n+1}(u')] \le C \{\mathbb{E}\|X_1\|^{J}\}^{2/J} \{\alpha_X(n)\}^{(J-2)/J} \le C n^{-\nu (J-2)/J}.
    \end{align*}
    As by Assumption \ref{ass_1}, part ii) it holds that $\nu (J-2)/J>1$ it follows that
    \[
    \sum_{i=N^{\lfloor \gamma/2\rfloor}+1}^\infty
        \mathbb{E}[X_1(u) X_{i+1}(u')] =\mathcal{O}(N^{(1-\nu (J-2)/J)(\gamma/2)}).
    \]
    Hence, we have shown that for a sufficiently small number, say $\gamma_2>0$, it holds that
    \[
    R_1 = \lambda' c(u,u')+ \mathcal{O}(N^{-\gamma_2})=\min(\lambda, \lambda')c(u,u') + \mathcal{O}(N^{-\gamma_2}),
    \]
    and the approximation rate holds uniformly over $(u, \lambda), (u',  \lambda') \in \mathcal{G}_N$. Using analogous techniques, we can show that
    \[
    R_2 = \mathcal{O}(N^{-\gamma_2}).
    \]
    In order to  show that $R_3$ is also polynomially decreasing in $N$, consider a sequence of natural numbers $(b_N)_{N \in \mathbb{N}}$ with $b_N:=\lfloor N^{1/3}\rfloor $. Then,  it holds that
    \begin{align*}
        & \bigg|\mathbb{E}\bigg[\Big(\frac{1}{\sqrt{N}} \sum_{i=1}^{\lfloor \lambda' N\rfloor }X_{i,M_N}(u)\Big)\Big(\frac{1}{\sqrt{N}} \sum_{i=\lfloor \lambda' N\rfloor +1}^{\lfloor \lambda N\rfloor }X_{i,M_N}(u')\Big)\bigg] \bigg| \le B_{1,N} + B_{2,N}\\
       & B_{1,N}:=  \bigg|\mathbb{E}\bigg[\Big(\frac{1}{\sqrt{N}} \sum_{i=1}^{\lfloor \lambda' N\rfloor }X_{i,M_N}(u)\Big)\Big(\frac{1}{\sqrt{N}} \sum_{i=\lfloor \lambda' N\rfloor +1}^{\lfloor \lambda' N\rfloor +b_N}X_{i,M_N}(u')\Big)\bigg] \bigg|\\
        &B_{2,N}:= \bigg|\mathbb{E}\bigg[\Big(\frac{1}{\sqrt{N}} \sum_{i=1}^{\lfloor \lambda' N\rfloor }X_{i,M_N}(u)\Big)\Big(\frac{1}{\sqrt{N}} \sum_{i=\lfloor \lambda' N\rfloor +b_N +1}^{\lfloor \lambda N\rfloor }X_{i,M_N}(u')\Big)\bigg] \bigg|.
    \end{align*}
    Now, we will prove that both terms $B_{i,N}$ are polynomially decreasing in $N$. For the first one, we employ the Cauchy-Schwarz inequality to obtain
    \begin{align*}
         & B_{1,N} \le \bigg\{\mathbb{E}\bigg[\Big(\frac{1}{\sqrt{N}} \sum_{i=1}^{\lfloor \lambda' N\rfloor }X_{i,M_N}(u)\Big)^2\bigg]\mathbb{E}\bigg[\Big(\frac{1}{\sqrt{N}} \sum_{i=\lfloor \lambda' N\rfloor +1}^{\lfloor \lambda' N\rfloor +b_N}X_{i,M_N}(u')\Big)^2\bigg]\bigg\}^{1/2}\\
         =&\mathcal{O}(1) \times \mathcal{O}(b_N^{1/2}N^{-1/2})=\mathcal{O}(N^{-1/6}).
    \end{align*}
    To upper bound the moments of the sums, we have used once more Theorem 3 in \cite{yoshihara:1978}, where the conditions of named theorem are easily seen to be satisfied in the above setup. To upper bound $B_{2,N}$, notice that  we have according to Lemma 3.11 in \cite{dehling:mikosch:sorensen:2002}
    \begin{align*}
        B_{2,N}\le  &\bigg\{\mathbb{E}\bigg[\Big|\frac{1}{\sqrt{N}} \sum_{i=1}^{\lfloor \lambda' N\rfloor }X_{i,M_N}(u)\Big|^{4}\bigg]\\
        & \,\,\,\times   \mathbb{E}\bigg[\Big|\frac{1}{\sqrt{N}} \sum_{i=\lfloor \lambda' N\rfloor +b_N +1}^{\lfloor \lambda N\rfloor }X_{i,M_N}(u')\Big)\Big|^{4}\bigg]\bigg\}^{1/2} (\alpha_{X,\delta,M_N}(b_N))^{1/2}.
    \end{align*}
    The sums on the right are again bounded according to Theorem 3 in \cite{yoshihara:1978}. The key assumption of the cited Theorem is that
    \[
    \sum_{n=1}^\infty  n^{4/2-1} [\alpha(n)]^{(J-4)/J} \le C_\alpha<\infty.
    \]
    This condition follows by Assumption
    \ref{ass_1b} part iii) and a small calculation. Recall that $\alpha_{X,\delta,M_N}(n) \le C n^{- \nu}$ for
    \[
    \nu>3+\frac{9}{J-4}.
    \]
    To show convergence of the series, it will suffice to show that
    \[
     n^{4/2-1} [\alpha_{X,\delta,M_N}(n)]^{(J-4)/J} \le C n^{1-\big[3+\frac{9}{J-4}\big]\frac{J-4}{J}}<C n^{-(1+\zeta)}
    \]
    for some $\zeta>0$. We focus only on the exponents and observe that
    \begin{align*}
        &1-\Big[3+\frac{9}{J-4}\Big]\frac{J-4}{J} = 1- \frac{3(J-4)+9}{J} = -\frac{2J+9-12}{J}\\
        =& -\frac{2J-3}{J}>-\frac{J+1}{J},
    \end{align*}
where in the last step we have used that $J>4$. Now, we return to upper bounding $B_{2,N}$ and notice that the
     factor $(\alpha_{X,\delta,M_N}(b_N))^{1/2}$ is polynomially decaying in $N$. More precisely
    \[
    (\alpha_{X,\delta,M_N}(b_N))^{1/2} \le C N^{-\nu/6}
    \]
    where we have used the definition of $b_N = \lfloor N^{1/3} \rfloor$ and the bound $\alpha_{X,\delta,M_N}(n) \le Cn^{-\nu}$.
    Potentially shrinking $\gamma_2>0$ we thus obtain $R_3 =  \mathcal{O}(N^{-\gamma_2})$, and by our above arguments
    \[
    \Sigma_N[d(x),d(x')] = R_1 +R_2+R_3 = \tilde \Sigma_N[d(x),d(x')] +  \mathcal{O}(N^{-\gamma_2}).
    \]
    Since the $\mathcal{O}$-term was independent of $x,x'$ we obtain
    \begin{align} \label{e:matappr}
    \max_{x,x'}|\Sigma_N(d(x),d(x'))-\tilde \Sigma_N(d(x),d(x'))| \le C N^{-\gamma_2}.
    \end{align}
    Now, let us consider the Prokhorov distance
    \[
    \pi_\infty(\Vec{W}_N^{dis}, \mathcal{N}(0, \Sigma_N)) = \pi_\infty(\mathcal{N}(0, \tilde \Sigma_N), \mathcal{N}(0, \Sigma_N)).
    \]
    Since the Euclidean norm $|\cdot|_2$ is stronger than the maximum norm, it follows that
    \[
    \pi_\infty(\mathcal{N}(0, \tilde \Sigma_N), \mathcal{N}(0, \Sigma_N)) \le \pi_2(\mathcal{N}(0, \tilde \Sigma_N), \mathcal{N}(0, \Sigma_N)),
    \]
    where as defined before, the right side refers to the Prokhorov metric w.r.t. the probability measures on $(\mathbb{R}^{|\mathcal{G}_N|}, |\cdot|_2)$. It is well known (see, e.g. \cite{gibbs:su:2002}), that the Prokhorov metric can be further upper bounded by the Wasserstein metric
    \[
    \pi_2(\mathcal{N}(0, \tilde \Sigma_N), \mathcal{N}(0, \Sigma_N))^2 \le \mathcal{W}_2(\mathcal{N}(0, \tilde \Sigma_N), \mathcal{N}(0, \Sigma_N)).
    \]
    In turn, the Wasserstein metric between the two Gaussian measures can be expressed as
    \begin{align} \label{e:wasserb}
    \mathcal{W}_2(\mathcal{N}(0, \tilde \Sigma_N), \mathcal{N}(0, \Sigma_N)) =  Tr\Big[ \tilde \Sigma_N+ \Sigma_N-2\big( \tilde \Sigma_N^{1/2}\Sigma_N\tilde \Sigma_N^{1/2}\big)^{1/2}\Big].
    \end{align}
    Here $Tr$ refers to the trace of a matrix and $()^{1/2}$ to the canonical root (for details see \cite{panaretos:zemel:2020}, Section 1.6.3).
    Recalling that the matrices $\tilde \Sigma_N,  \Sigma_N$ act on a space of dimension $|\mathcal{G}_N|\le C N^{\rho+2\varsigma}$ the inequality \eqref{e:matappr}
    implies
    \begin{align}
        \label{e:matbound1}
    \|\tilde \Sigma_N-\Sigma_N\|_{Tr} \le &  C |\mathcal{G}_N|^{2} \max_{x,x'}|\Sigma_N[d(x),d(x')]-\tilde \Sigma_N[d(x),d(x')]|  \\
    \le & C N^{2\rho+4\varsigma-\gamma_2}. \nonumber
    \end{align}
    Here $\|\cdot\|_{Tr}$ denotes the trace norm of a matrix and for subsequent reference we denote by $\|\cdot\|_2$ the Frobenius norm. In particular, we obtain
    \begin{align}\label{e:matbound1b}
    Tr\Big[ \tilde \Sigma_N+ \Sigma_N\Big] = 2 Tr\Big[ \tilde \Sigma_N\Big] + \mathcal{O}(N^{2\rho+4\varsigma-\gamma_2}).
    \end{align}
    For the next step, we use a proximity result for the square roots of matrices, found in Lemma 4.1 of \cite{powers:stormer:1970}. It states for any two covariance matrices $\Sigma_1, \Sigma_2$ that
    \[
    \|\Sigma_1^{1/2}-\Sigma_2^{1/2}\|_2^2 \le \|\Sigma_1-\Sigma_2\|_{Tr}.
    \]
    Since the root of a covariance matrix is again a covariance this also implies
    \[
    \|\Sigma_1^{1/2}\|_2^2 \le \|\Sigma_1\|_{Tr}.
    \]
    We also notice that for square matrices it holds that $\|\Sigma_1\Sigma_2\|_{Tr} \le \|\Sigma_1\|_{Tr} \|\Sigma_2\|_{2}$.
    In the following calculations we also use the fact that for a covariance matrix $\Sigma$ of dimension $q \times q$ the well-known identities hold:
    \begin{align*}
        \|\Sigma\|_2 \le \|\Sigma\|_{Tr} \le \sqrt{q}\|\Sigma\|_{2} \quad \textnormal{and} \quad \|\Sigma\|_{Tr} \le q^2 \max_{1 \le \ell, k \le q} \Sigma[\ell, k].
    \end{align*}
    Using these results implies
\begin{align}\label{e:matbound2}
    &\big\|\big( \tilde \Sigma_N^{1/2}\Sigma_N\tilde \Sigma_N^{1/2}\big)^{1/2}-\tilde \Sigma_N\big\|_{Tr}  \le |\mathcal{G}_N|^{1/2} \big\|\big( \tilde \Sigma_N^{1/2}\Sigma_N\tilde \Sigma_N^{1/2}\big)^{1/2}-\big(\tilde \Sigma_N^2\big)^{1/2}\big\|_{2} \\
    \le & |\mathcal{G}_N|^{1/2} \sqrt{\|\tilde \Sigma_N^{1/2}\Sigma_N\tilde \Sigma_N^{1/2}-\tilde \Sigma_N^2\|_{Tr} } \le  |\mathcal{G}_N|^{1/2}  \sqrt{\|\tilde \Sigma_N^{1/2}\|_2^2 \| \Sigma_N - \tilde\Sigma_N\|_{Tr}} \nonumber \\
    \le &  |\mathcal{G}_N|^{3/2}\sqrt{\|\tilde \Sigma_N\|_{Tr}\max_{x,x'}|\Sigma_N[d(x),d(x')]-\tilde \Sigma_N[d(x),d(x')]|}\nonumber \\
    \le &  |\mathcal{G}_N|^{5/2}\sqrt{\max_{x,x'}|\Sigma_N[d(x),d(x')]-\tilde \Sigma_N[d(x),d(x')]|}
    = \mathcal{O}(N^{5/2(\rho+2\varsigma)-\gamma_2/2}). \nonumber
    \end{align}
    Together, \eqref{e:matbound1b} and \eqref{e:matbound2} imply that
    \[
    \eqref{e:wasserb} \le \mathcal{O}(N^{2\rho+4\varsigma-\gamma_2}) + \mathcal{O}(N^{5/2(\rho+2\varsigma)-\gamma_2/2}).
    \]
    Choosing $\rho, \varsigma$ sufficiently small, the right side tends to $0$ polynomially in $N$.
    Together with the decomposition \eqref{e:pidecmain}, this now yields the desired result for a sufficiently small choice of $\tau>0$. 
\end{proof}

\begin{lemma} \label{lemA3}
     Suppose that the assumptions of Theorem \ref{thm1} are met. Then for some sufficiently small $\tau>0$ and some fixed $C_1', C_2'>0$ (only depending on the constants in Assumptions \ref{ass_1} and \ref{ass_1b})
     \[
     \mathbb{P}(\|W^{dis}- W\|>C_1' N^{-\tau}) <C_2' N^{-\tau}.
     \]
\end{lemma}

\begin{proof}
    We first derive for fixed but arbitrary $\lambda \in [0,1]$ and  $\chi>0$ a bound on the tail probability
   $ \mathbb{P}(\sup_{|u|>N^{\rho}}|W(\lambda,u)|>\chi)$. We can decompose this probability into
   \begin{align*}
   &\qquad \quad\,\mathbb{P}\Big(\sup_{|u|>N^{\rho}}|W(\lambda,u)|>\chi\Big) \le  A_1+A_2\\
   &A_1:= \mathbb{P}\Big(\sup_{u>N^{\rho}}|W(\lambda,u)|>\chi\Big), \quad  A_2:=\mathbb{P}\Big(\sup_{u<-N^{\rho}}|W(\lambda,u)|>\chi\Big).
   \end{align*}
   For the purpose of illustration, we focus on the first term $A_1$. Let us define the intervals
   \[
   I^{(N)}_1:=[\lfloor N^\rho\rfloor, \lfloor N^\rho\rfloor + 1], I^{(N)}_2:=[\lfloor N^\rho\rfloor+1, \lfloor N^\rho\rfloor+2],\ldots
   \]
   Then, for the interval $I_j^{(N)}$ define a total of $j+\lfloor N^\varsigma \rfloor $ gridpoints that are evenly spread across the interval s.t. for any $u \in I_j^{(N)}$ there exists a gridpoint $g_j(u)$ with
   \[
    |u-g_j(u)|\le (j+\lfloor N^\varsigma \rfloor )^{-1}.
   \]
   Now, we observe the following upper bound
   \begin{align}\label{e:splitpro}
       &\mathbb{P}\Big(\sup_{u>N^{\rho}}|W(\lambda,u)|>\chi\Big) \le  \sum_{j \ge 1} \mathbb{P}\Big(\sup_{u \in I_j^{(N)}}|W(\lambda,u)|>\chi\Big) \\
       \le  &\sum_{j \ge 1}  A_{1,1,j} + A_{1,2,j}, \nonumber\\
         A_{1,1,j}:=  &  \mathbb{P}\Big(\sup_{u \in I_j^{(N)}}|W(\lambda,u)-W(\lambda,g_j(u))|>\chi/2\Big), \nonumber\\
         A_{1,2,j}:= & \mathbb{P}\Big(\sup_{u \in I_j^{(N)}}|W(\lambda,g_j(u))|>\chi/2\Big).\nonumber
   \end{align}
   We consider $A_{1,2,j}$  first. Recall that by $c$ (defined in \eqref{e:lr_var}), we denote the covariance kernel of the Brownian motion $W$. A small calculation yields that
   \begin{align*}
       A_{1,2,j}  \le &(j+\lfloor N^\varsigma \rfloor) \max_{u \in I_j^{(N)}} \mathbb{P}(|W(\lambda,g_j(u))|>\chi/2)  \\
       \le &C (j+N^\varsigma) \chi^{-2} \max_{u \in I_j^{(N)}}\mathbb{E}[|W(\lambda,g_j(u))|^2] \\
       \le & C (j+N^\varsigma)\chi^{-2} \max_{u \in I_j^{(N)}}c(g_j(u),g_j(u)).
   \end{align*}
   Now, let us study $c(g_j(u),g_j(u))$. Let $B>0$ be a fixed number to be determined later. Using the definition of $c$, we observe that it can be upper bounded as follows:
   \begin{align*}
      & c(g_j(u),g_j(u)) \le  2 \sum_{n \ge 1} |\mathbb{E}[X_1(g_j(u))X_n(g_j(u))]| \\\le &2 \lfloor (j+N)^B \rfloor  \mathbb{E}[(X_1(g_j(u))^2] + 2 \sum_{n \ge \lfloor (j+N)^B \rfloor} |\mathbb{E}[X_1(g_j(u))X_n(g_j(u))]|.
   \end{align*}
   We have thus split up $c(g_j(u),g_j(u))$ in a (small) number of "early terms" and an infinite sum of "later terms".
   Using first the tail condition in Assumption \ref{ass_1b} part iii) and second  that $u \in I_j^{(N)}$ implies $g_j(u) \ge  j+\lfloor N^{\rho} \rfloor -1$, we obtain as a bound for the early terms
   \begin{align*}
   &2 \lfloor (j+N)^B \rfloor \mathbb{E}(X_1(g_j(u))^2 \\
   \le & C \lfloor (j+N)^B \rfloor \exp(-g_j(u)^{\kappa}) \le  C \lfloor (j+N)^B \rfloor \exp(-|j+\lfloor N^{\rho} \rfloor -1|^{\kappa}).
   \end{align*}
    Next, we upper bound the tail of the later terms, yielding
\begin{align*}
    & \sum_{n \ge \lfloor (j+N)^B \rfloor} |\mathbb{E}[X_1(g_j(u))X_n(g_j(u))]| \\
    \le &   \sum_{n \ge \lfloor (j+N)^B \rfloor}\sup_\ell\{\mathbb{E}\|X_\ell\|^J\}^{2/J} \alpha_X(n)^{(J-2)/J} \\
    \le & \sum_{n \ge \lfloor (j+N)^B \rfloor} C n^{-\nu (J-2)/J} \le C (j+N)^{B(1-\nu (J-2)/J))}.
\end{align*}
 Putting these bounds together shows that
   \begin{align*}
    &\sum_{j \ge 1} A_{1,2,j}
   \le \sum_{j \ge 1} \big\{C \chi^{-2} (j+N^\varsigma)\big\} \\
   &
   \times \Big[  \big\{C  \lfloor (j+N)^B \rfloor \exp(-|j+\lfloor  N^{\rho} \rfloor -1|^{\kappa})\big\}+ \big\{C(j+N)^{B(1-\nu (J-2)/J))}\big\} \Big].
   \end{align*}
   Here, we have used the mixing inequality from Lemma 3.11 in \cite{dehling:mikosch:sorensen:2002} and the mixing assumption from Assumption \ref{ass_1}.
   Choosing $B$ sufficiently large (notice that we can choose it arbitrarily large) we obtain that the right side is of size $\mathcal{O}(\chi^{-2}N^{-1})$. Notice that this rate is not sharp, but it suffices for our subsequent investigations. Now, we have derived a decay rate for the second sum in \eqref{e:splitpro} and can turn to the first one. We obtain by Markov's inequality
   \[
  A_{1,1,j} \le C\chi^{-Q} \mathbb{E}[\sup_{u \in I_j}|W(\lambda,u)-W(\lambda,g_j(u))|^Q].
   \]
    Here,  $Q>0$ is an even natural number to be determined later. To upper bound the right side, we want to use again Theorem 2.2.4 in \cite{vaart:wellner:1996}, a bound on the modulus of continuity for stochastic processes. For this purpose, we notice that for any $u,u'$
   \begin{align*}
   \mathbb{E}|W(\lambda,u)-W(\lambda,u')|^Q \le  C\{\mathbb{E}|W(\lambda,u)-W(\lambda,u')|^2\}^{Q/2}.
   \end{align*}
   Here we have exploited that the marginals of the Brownian motion are normally distributed. The constant $C$ can here be explicitly given as $(Q-1)!!$. Now we can further bound the second moment by
   \[
   \mathbb{E}|W(\lambda,u)-W(\lambda,u')|^2 \le \mathbb{E}|W(1,u)-W(1,u')|^2 = c(u,u)+c(u',u')-2c(u,u').
   \]
    We can now show that, e.g., the function $c(u,u)-c(u,u')$ is Hölder continuous in the second component. Therefore, simply observe that by definition of $c$ we have
    \begin{align*}
        \frac{|c(u,u)-c(u,u')|}{|u-u'|^\xi}  \le & 2 \sum_{i \ge 1} \bigg|\mathbb{E}\bigg[\frac{(X_1(u)-X_1(u'))}{|u-u'|^\xi}X_i(u)\bigg]\bigg| \\
        \le & 2 \sum_{i \ge 1} \mathbb{E}\bigg|\frac{(X_1(u)-X_1(u'))}{|u-u'|^\xi}\bigg|^J\mathbb{E}|X_i(u)|^J\Big\}^{1/J} \alpha_X(i)^{(J-2)/J}.
    \end{align*}
   Using the mixing Conditions i) and ii)  in Assumption \ref{ass_1}, implies that the sum on the right side is uniformly bounded. With our above considerations this yields
   \[
   \{\mathbb{E}|W(\lambda,u)-W(\lambda,u')|^Q\}^{1/Q} \le C |u-u'|^{\xi}
   \]
   for any even $Q$.
   Now, using Theorem 2.2.4 in \cite{vaart:wellner:1996}, we obtain for any constant $\eta>0$
   \begin{align*}
       & \{\mathbb{E}[\sup_{u \in I_j}|W(\lambda,u)-W(\lambda,g_j(u))|^Q]\}^{1/Q}\\
       \le & C \int_0^{\eta} x^{1/(\xi Q)} dx + C (j+N^\varsigma)^{-\xi} \eta^{-2/(\xi Q)} \le C \eta^{1-1/(\xi Q)} + C(j+N^\varsigma)^{-\xi} \eta^{-2/(\xi Q)}.
   \end{align*}
   Now, choose $\eta=(j+N^\varsigma)^{-1}$ and $Q$ so large that $2/(\xi Q)<\xi/2$. Then the right side is of size $C (j+N^\varsigma)^{-\xi/2}$.
   To see this, recall that $\xi\le 1$ and notice that
    \begin{align*}
        & -\xi + 2/(\xi Q)<-\xi + \xi/2 = -\xi/2\qquad \textnormal{and}\\
        & -1+ 1/(\xi Q) <-1+\xi/4 \le -3/4 <-1/2\le -\xi/2.
    \end{align*}
   As a consequence, we obtain for any $j$ that
   \[
   \mathbb{E}[\sup_{u \in I_j}|W(\lambda,u)-W(\lambda,g_j(u))|^Q] \le C (j+N^\varsigma)^{-\xi Q/2}
   \]
   and hence, since (by our above choice of $Q$) $\xi Q/2>2$ we get
   \begin{align*}
      &  \sum_{j \ge 1} A_{1,1,j}
      \le  \sum_{j \ge 1}C \chi^{-Q} (j+N^\varsigma)^{-\xi Q/2} \\
      \le & C\sum_{j \ge 1}\chi^{-Q}(j)^{-(\xi Q/2-1)} N^{-\varsigma} = \mathcal{O}(N^{-\varsigma}\chi^{-Q}).
   \end{align*}
   So, up to this point, we have shown that
   \[
   A_1 \le   \sum_{j \ge 1}  A_{1,1,j} + A_{1,2,j} \le  C \big[ \chi^{-2}N^{-1} + \chi^{-Q} N^{-\varsigma}\big].
   \]
   Now choosing $\chi=N^{-\varsigma'}$ for a $\varsigma'=\varsigma/(2Q)$, the fact that $\varsigma \le 1$ and $Q \ge 2$ even, shows that
   \[
   \mathbb{P}\Big(\sup_{|u|>N^{\rho}}|W(\lambda,u)|>\chi\Big)= \mathcal{O}(N^{-\varsigma/2}).
   \]
   Next, we consider an even grid $\mathcal{G}^{\lambda}_N$ on the interval $[0,1]$ with gridsize $ N^{-\varsigma/4} $. By our above calculations we have
   \[
   \mathbb{P}\Big(\max_{\lambda \in \mathcal{G}^{\lambda}_N}\sup_{|u|>N^{\rho}}|W(\lambda,u)|>\chi\Big)=\mathcal{O}( N^{-\varsigma/4}).
   \]
  Now, we can leverage the fact that Brownian motion as a process indexed in $[0,1]$ on the space $\mathcal{C}_0(\mathbb{R}, \mathbb{R})$ is Hölder continuous for a constant $(0,1/2)$ (just as the real valued Brownian motion). In particular, for a sufficiently small choice of $\varsigma''>0$ we have
  \[
  \mathbb{P}\Big(\sup_{|u|>N^{\rho}}\sup_{|\lambda'-\lambda|\le N^{-\varsigma/4}, \lambda \in \mathcal{G}^{\lambda}_N}|W(\lambda,u)-W(\lambda',u)|>N^{-\varsigma''}\Big)\le C N^{-\varsigma''}.
  \]
  This claim follows directly from the proof of Lemma 3.1 of \cite{acosta:1985} (the named proof actually shows the stronger result of an exponentially decaying probability, which however is not required for our investigation). Putting these observations together shows for a sufficiently small constant $\tau>0$ (depending on $\varsigma, \varsigma''$) that
  \begin{align*}
   & \mathbb{P}\Big(\sup_{|u|>N^{\rho}} \sup_\lambda |W(u,\lambda)-W^{dis}(u,\lambda)|>CN^{-\tau}\Big)\\
   = & \mathbb{P}\Big(\sup_{|u|>N^{\rho}} \sup_\lambda |W(u,\lambda)|>CN^{-\tau}\Big) \le C N^{-\tau}.
  \end{align*}
  Next, we consider the difference $W(u,\lambda)-W^{dis}(u,\lambda)$ for indices with $|u|<N^{\rho}$. Recalling the definition of the discretization function, it follows that
  \[
  |W(u,\lambda)-W^{dis}(u,\lambda)| \le \sup_{|(u,\lambda)-(u',\lambda')|\le N^{-\varsigma}}|W(u,\lambda)-W(u',\lambda')|.
  \]
  We can now investigate the modulus of continuity  on the right side. The proof follows by another application of the bound in Theorem 2.2.4 in \cite{vaart:wellner:1996} and we omit the details. It shows for a sufficiently small $\tau>0$ that
  \[
  \mathbb{P}\Big(\sup_{|(u,\lambda)-(u',\lambda')|\le N^{-\varsigma}}|W(u,\lambda)-W(u',\lambda')|>N^{-\tau}\Big)\le C N^{-\tau},
  \]
  completing the proof. 
\end{proof}

\bigskip \smallskip

\noindent{\sc Proof of Theorem \ref{thm1}.}
Lemmas \ref{lemA1}, \ref{lemA2}, \ref{lemA3} imply respectively the identities \eqref{e:rate1}, \eqref{e:rate2},  \eqref{e:rate3}. We recall therefore -again that proximity of two random variables, say of $P_N^{dis}$ and $P_N$ in probability as shown in Lemma \ref{lemA1}, implies proximity in the Prokhorov metric, as discussed in Theorem 2.13 from \cite{huber:ronchetti:2009}.
Using identities \eqref{e:rate1}, \eqref{e:rate2},  \eqref{e:rate3} together with the triangle inequality for the Prokhorov metric yields the desired result.  \hfill $\square$\\

\section{Additional lemmas }

\subsection{Approximation results} $ $

 Lemma \ref{lemboundPN} is  used in the proof of Theorem \ref{t:wass}. There, we exploit that $\bar \tau = 6\rho$ is small compared to $\tau$. Notice that the size of $\tau$ in the previous proofs does not increase when making $\rho$ smaller. We also notice that in named proof, $q<4$ can be brought arbitrarily close to $4$ if $\rho$ is sufficiently small (in particular, in the case of $I_2=\mathbb{R}$, we could make $\rho$ arbitrarily small).

\begin{lemma} \label{lemboundPN}
    Under the conditions of Theorem \ref{t:wass} it holds for $\bar \tau =10 \rho$ that
    \[
\mathbb{E}\big[\|P_N^{lin}\|^4\big]\le C N^{\bar \tau}, \qquad \mathbb{E}\big[\|W\|^4\big] \le C.
\]
In particular, $\bar \tau$ can be made arbitrarily small, for sufficiently small $\rho$.
\end{lemma}

\begin{proof}
It follows directly from the theorem of \cite{fernique:1975} that $\mathbb{E}\big[\|W\|^4\big] \le C.$ Hence, we focus on $\mathbb{E}[\|P_N^{lin}\|^4]$.
First, we notice that
\begin{align*}
    &\mathbb{E}[\|P_N^{lin}\|^4] =  \mathbb{E}\big[\sup_{u \in \mathbb{R}, \lambda \in [0,1]} |P_N^{lin}(\lambda, u)|^4 \big] \\
    \le & \mathbb{E}\big[\sup_{|u|\le N^\rho, \lambda \in [0,1]} |P_N^{lin}(\lambda, u)|^4 \big] +  \mathbb{E}\big[\sup_{|u|>N^\rho , \lambda \in [0,1]} |P_N^{lin}(\lambda, u)|^4 \big]\\
    = &  \mathbb{E}\big[\sup_{|u|\le N^\rho, \lambda \in [0,1]} |P_N^{lin}(\lambda, u)-P_N^{lin}(0,0)|^4 \big]+\mathbb{E}\big[\sup_{|u|>N^\rho , \lambda \in [0,1]} |P_N^{lin}(\lambda, u)|^4 \big]\\
    =: &R_1+R_2.
\end{align*}
In the second step, we have used that $P_N(0,0)=0$ (empty partial sum).  We can now analyze $R_1$ and $R_2$ separately. Exactly as in beginning of the proof of Lemma \ref{lemA1} (step i), we can establish that $R_2=o(1)$ and we can hence focus on $R_1$. Recall the definition of the metric \eqref{e:dmet}.
In the proof of Lemma \ref{lem:modcont}, we have established the continuity result \eqref{e:QboundP}. Together with Theorem 2.2.4 in \cite{vaart:wellner:1996} it implies that
\begin{align*}
    \{R_1\}^{1/4} \le & C \int_0^{1} \{D_{d_{1/2, \xi}}([0,1]\times [-N^\rho,N^\rho],x)\}^{1/4} dx\\
    & + C N^{2\rho } D_{d_{1/2, \xi}}([0,1]\times [-N^\rho,N^\rho],1)^{1/2}\\
\le &  C \int_0^{1}   x^{-(2+1/\xi)/4} N^{\rho/4} dx + C N^{5\rho/2} \le C N^{ 5\rho/2}.
\end{align*}
Thus, the claim of the Lemma follows.
\end{proof}

\medskip

Lemma \ref{lemdehl} is used for the proof of the high-dimensional Gaussian approximation in Lemma \ref{lemA2}.

\begin{lemma} \label{lemdehl}
    Suppose that $v_1, v_2,\ldots \in \mathbb{R}^q$ is a sequence of random vectors, satisfying for some $\eta_1>0$ the moment condition $\mathbb{E}|v_i|^{3+\eta_1}\le C'_1$. Moreover, suppose the vectors are strongly mixing with coefficients $\alpha_v(n) \le C'_2  n^{-\eta_2}$ where
    $
    \eta_2 >3+9/\eta_1.
    $
    Then, for $N \in \mathbb{N}$ define $\Sigma_N$ as the covariance matrix of the sum $\frac{1}{\sqrt{N}}\sum_{i=1}^N v_i$. It holds that
    \[
    \pi_\infty\bigg( \frac{1}{\sqrt{N}}\sum_{i=1}^N v_i,\mathcal{N}\big(0, \Sigma_N\big)\bigg) \le C
    N^{-1/20}q^{3}\]
    where $C$ only depends on $C'_1, C'_2, \eta_1, \eta_2$.
\end{lemma}

\begin{proof}
    The result is a direct consequence of Theorem 3.27 from \cite{dehling:mikosch:sorensen:2002}. If $\mathbb{E}|v_i|^{3+\eta_1}\le C'_1$ this implies
    \[
    \mathbb{E}|v_i|_2^{3+\eta_1}\le q^{(3+\eta_1)/2}\mathbb{E}|v_i|^{3+\eta_1} \le C'_1 q^{(3+\eta_1)/2}.
    \]
    Now, w.l.o.g. assuming that $C'_1\ge 1$, named theorem implies
    \begin{align}\label{e:dehlest}
     \pi_2\bigg( \frac{1}{\sqrt{N}}\sum_{i=1}^N v_i,\mathcal{N}\big(0, \Sigma_N\big)\bigg) \le C N^{-1/20} q^{3}.
    \end{align}
    Finally, since the Euclidean norm is stronger than the maximum norm, it follows that
    \[
    \pi_\infty\bigg( \frac{1}{\sqrt{N}}\sum_{i=1}^N v_i,\mathcal{N}\big(0, \Sigma_N\big)\bigg) \le \pi_2\bigg( \frac{1}{\sqrt{N}}\sum_{i=1}^N v_i,\mathcal{N}\big(0, \Sigma_N\big)\bigg),
    \]
    yielding the desired result.
\end{proof}

\subsection{Examples of sparse functional data} \label{sec_ex} $ $

We now show  that the estimators listed in Example \ref{ex1}
satisfy Assumption \ref{ass_1}, Condition iii) in an appropriate way.
We denote by $\nabla^{(b)} f$ the $b$th derivative of a function $f$. We impose the following condition for all three examples.

\begin{itemize}
    \item[(S1)] Suppose that for some $\beta>0$ and constants
$C_{\beta,1}< C_{\beta,2}$ it holds that
$C_{\beta,1}N^\beta \le M_N \le C_{\beta,2}N^\beta  $
(polynomial growth condition). Furthermore,
suppose that the kernel function $K$ in
parts i) and iii) is Gaussian.
\end{itemize}
Our aim is to derive regularity properties (Assumption \ref{ass_1}, Condition iii)) for the estimators in Example \ref{ex1} directly from the estimation technique and the regularity of the original functions.\\
\textbf{Part i)} We impose an additional assumption for this part. 
\begin{itemize}
    \item[(S2)] Suppose that for some $\zeta>0$, $\sup_n \mathbb{E}\|\nabla^{(b)} X_n\|^{J+\zeta}<\infty$ for $b=0,1,2$. Furthermore, suppose that the design density $f$ is twice continuously differentiable, with bounded 
derivatives and that the bandwidth $h=h_N$ is chosen such that $h=C M_N^{-b}$ for some $0 < b<1/4$.
\end{itemize}
We will show under (S1) and  (S2) that we can define a version of the sparse estimators $\tilde X_{n,M_N}$ such that for all large enough $N$
\begin{align} \label{e:vers}
\mathbb{P}(\tilde X_{n,M_N} = X_{n,M_N}: 1 \le n \le N)\ge 1-1/N
\end{align}
such that $(\tilde X_{n,M_N})_{n \in \mathbb{N}}$ satisfies Assumption \ref{ass_1}, part iii). Therefore, recall the definition 
\[
X_{n,M_N} := \frac{\frac{1}{M_N}\sum_{m=1}^{M_N} K_{h}(u-u_{n,m}^{(M_N)})y_{n,m}^{(M_N)}}{\frac{1}{M_N}\sum_{m=1}^{M_N} K_{h}(u-u_{n,m}^{(M_N)})}:= \frac{\frac{1}{M_N}\sum_{m=1}^{M_N} K_{h}(u-u_{n,m}^{(M_N)})y_{n,m}^{(M_N)}}{\hat f_{n,N}(u)},
\]
where $\hat f_{n,N}$ is defined in the obvious way.  $\hat f_{n,N}$ approximates the density of the design points $f$ which is positive and hence, on the interval $I_2$ bounded away from $0$ with 
\[
\inf_{v \in I_2}f(v)\ge c_f>0.
\]
We will now define the event 
\[
\mathcal{E}_{n,N}:=\big\{\sup_{v \in I_2}|\hat f_{n,N}(v)-f(v)|\le c_f/2\big\}
\]
and notice that on this event it holds that
\[
\inf_{v \in I_2}\hat f_{n,N}(v)\ge c_f/2.
\]
Using Lemma 2.4 in \cite{schuster:1969} implies that (for sufficiently large $N$)
\[
\mathbb{P}(\mathcal{E}_{n,N}) \le C \exp(-CM_N h^2),
\]
which decays exponentially in $N$ since the product $M_N h^2$ is by assumptions (S1) and (S2) polynomially going to $\infty$. Accordingly, if we define 
\[ 
\tilde X_{n,M_N} = X_{n,M_N} \cdot \mathbb{I}\{\mathcal{E}_{n,N}\},
\]
it follows directly, that \eqref{e:vers} is satisfied, since 
\[
\mathbb{P}\Big(\bigcup_{n=1}^N\{\tilde X_{n,M_N} \neq X_{n,M_N}\}\Big) \le \sum_{n=1}^N \mathbb{P}(\mathcal{E}_{n,N}^c) \le C N\exp(-CM_N h^2) \le 1/N,
\]
where the last inequality holds for large enough $N$.
Now, having constructed the version of the estimators $\tilde X_{n,M_N}$, we can establish Assumption \ref{ass_1}, part iii). Indeed, we will prove the stronger result 
\[
\sup_n\mathbb{E}\|\nabla^{(r)}X_n-\nabla^{(r)}\tilde X_{n,M_N}\|^J=\mathcal{O}(N^{-\gamma}),
\]
for $r=0,1$.
 Both cases work analogously and hence we confine ourselves to the case $r=0$. The tools in the work that our approximations rely on (see \cite{schuster:yakowitz:1979}) are developed for derivatives of any order. First, we notice that by definition of the sparse estimators the following upper bound holds
\begin{align*}
  & \|\tilde X_{n,M_N}\| =  \bigg\| \frac{\frac{1}{M_N}\sum_{m=1}^{M_N} K_h(u-u_{n,m}^{({M_N})})(X_n(u_{n,m}^{({M_N})})+\epsilon_{n,m}^{({M_N})})}{\frac{1}{M_N}\sum_{m=1}^{M_N} K_h(u-u_{n,m}^{({M_N})})} \bigg\| \mathbb{I}\{\mathcal{E}_{n,N}\}\\
   \le & (2/c_f) \Big\{ \frac{1}{M_N}\sum_{m=1}^{M_N} K_h(u-u_{n,m}^{({M_N})})\Big\}\max_{m}|X_n(u_{n,m}^{({M_N})})+\epsilon_{n,m}^{({M_N})}| \mathbb{I}\{\mathcal{E}_{n,N}\}\\
   \le &  (2/c_f) (c_f/2+\|f\|)\{\|X_n\|+ \max_{m} |\epsilon_{n,m}^{({M_N})}|\}.
\end{align*}
As a consequence, we can deduce that
\[
\mathbb{E}\|\tilde X_{n,M_N}\|^{J+\zeta}\le C \mathbb{E}\|X_{n}\|^{J+\zeta}+C \log(N).
\]
Here we have used that $M_N \sim N^\beta$ by assumption and that  $\mathbb{E}\max_{m} |\epsilon_{n,m}^{({M_N})}|^{J+\zeta} \le C \log(N^\beta)$ according to Lemma 2.2.2 in \cite{vaart:wellner:1996}, with $\psi(x)=\exp(x)-1$ (sharper bounds can be obtained but are not necessary in our case).
Now, we consider the approximation
\[
 \mathbb{E}\big[\|\tilde X_{n,M_N}-X_n\|^{J} \big]\le  \mathbb{E}\big[\|\tilde X_{n,M_N}-X_n\|^{J}\mathbb{I}\{\|\tilde X_{n,M_N}-X_n\|>N^{-\eta}\}\big]+ N^{-\eta J}.
\]
Here $\eta>0$ is positive but arbitrary and will be specified below. Evidently, the second term is polynomially decaying in $N$. We hence focus on the first one, that is bounded by
\begin{align*}
& C\big(\mathbb{E}\|\tilde X_{n,M_N}\|^{J+\zeta}+\mathbb{E}\|X_{n}\|^{J+\zeta}\big)^{J/(J+\zeta)} \mathbb{P}\big(\|\tilde X_{n,M_N}-X_n\|>N^{-\eta}\big)^{\zeta/(J+\zeta)}\\
\le &
C \log(N) \mathbb{P}\big(\|\tilde X_{n,M_N}-X_n\|>N^{-\eta}\big)^{\zeta/(J+\zeta)}.
\end{align*}
Here we have used the above fact that $\mathbb{E}\|\tilde X_{n,M_N}\|^{J+\zeta}\le C \log(N)$. If we can show for the probability on the right that it decays polynomially in $N$, then it also follows that $\mathbb{E}\|\tilde X_{n,M_N}-X_n\|^{J} $ is polynomially decaying in $N$. The proof of this result follows by the concentration results derived in \cite{schuster:yakowitz:1979}. More precisely, we define for $X_n$ the function
\[
w_n(v):=X_n(v)f(v), \qquad w_n^{(M)}(v):=\frac{1}{M}\sum_{m=1}^M K_h(v-u_{n,m}^{(M)})y_{n,m}^{(M)}.
\]
We can now for any $M$ observe that
\[
|\mathbb{E}[w_n^{(M)}(v)|X_n] - w_n(v)| = \bigg| \int \big[w_n(v-s)-w_n(v)\big] K_h(s) ds \bigg| \le C h \|\nabla w_n\|.
\]
Notice that by the product rule and the fact that $f$ has a bounded derivative we get that $ \|\nabla w_n\| \le C (\|X_n\|+ \|\nabla X_n\|)$. Accordingly, we obtain
\[
\mathbb{P}(\|\mathbb{E}[w_n^{(M_N)}|X_n] - w_n \|\ge N^{-\eta}/2) \le C h N^{\eta} \mathbb{E}(\|X_n\|+ \|\nabla X_n\|).
\]
By assumption, $h$ is polynomially decaying in $N$ with $h \sim M_N^{-b} \sim N^{-b\beta}$. If $\eta<b\beta$ we obtain polynomial decay in $N$ for
 $h N^{\eta}$
and thus polynomial decay of the above probabilities. Let us now for ease of reference define the event
\[
\mathcal{F}_{M_N}:=\big\{\|\mathbb{E}[w_n^{({M_N})}|X_n] - w_n \|< N^{-\eta} \big\}
\]
Choosing $\eta<b\beta$, by our above arguments, the complement $\mathcal{F}_{M_N}^c$ has polynomially decaying probability in $N$.
Now, we get the decomposition
\begin{align*}
    &\mathbb{P}\big(\|\tilde X_{n,{M_N}}-X_n\|>N^{-\eta}\big) \\
    \le & \mathbb{P}\big(\{\|X_{n,{M_N}}-X_n\|>N^{-\eta}\} \cap \mathcal{F}_{M_N}\big) +\mathbb{P}\big(\mathcal{F}_{M_N}^c\big)+\mathbb{P}\big(\mathcal{E}_{n,N}^c\big) .
\end{align*}
Finally, we analyze the decay of the first probability on the right side, which can be rewritten as
\begin{align} \label{e:schuster}
\mathbb{E}\big[\mathbb{P}\big(\{\|X_{n,{M_N}}-X_n\|>N^{-\eta}\} \cap \mathcal{F}_{M_N}\big|X_n \big)\big].
\end{align}
We apply Theorem 1 in \cite{schuster:yakowitz:1979}, which is a finite sample concentration bound for non-parametric regression. It is necessary to revisit the proof of this Theorem, to make two observations. First, the constant $C$ in the theorem depends only on the regression function through the constant $C_v$ in the proof from Lemma 3. Here in turn we see, that $C_v$ only depends on the regression function by its squared sup-norm (see constant $C_1$ in the proof of named result). Second, the bounds in the cited paper are formulated for "sufficiently large sample size". The condition of sufficiently large sample size also stems from Lemma 3. Here, it is used that the bias of the function $w_n$ (or $\nabla w_n$) is eventually smaller than any fixed $\varepsilon/2$. On the event $\mathcal{F}_{M_N}$ this condition is always satisfied with $\varepsilon=4N^{-\eta}$  and hence the cited bound holds for all sample sizes. It implies that
\[
\mathbb{P}\big(\{\|X_{n,{M_N}}-X_n\|>N^{-\eta}\} \cap \mathcal{F}_{M_N}\big|X_n \big) \le C \|X_n\|^2 M_N^{-\eta'} \le C N^{-\beta\eta'}
\]
for a sufficiently small
\[
\eta' = 1-2b-2\eta.
\]
Since $b<1/4$, we have $\eta'>1/2-2\eta$, which is positive if $\eta$ is small enough. In the case where we do all of these calculations for the first derivative, the condition turns out to be slightly stricter and $\eta'= 1-4b-2\eta$. Since $1-4b$ is positive, we can choose $\eta$ small enough such that  $\eta'= 1-4b-2\eta>0$. Notice that this is where the condition on the growth rate for $h$ stems from. We have thus shown that
 the expectation in \eqref{e:schuster} decays polynomially in $N$, proving the first property in Assumption \ref{ass_1} part iii). We get an analogous result for the first derivative, showing the second property in Assumption \ref{ass_1} part iii) with $\xi=1$. \\
 \textbf{Part ii)} We impose an additional assumption.
 \begin{itemize}
     \item[(S3)] We assume that $\sup_n\mathbb{E}\|X_n\|^J<\infty$ and
    \[
    \sup_n \mathbb{E} \bigg[\Big(\sup_{u \neq v} \frac{|X_n(u)-X_n(v)|}{|u-v|^{\xi}}\Big)^J\bigg] <\infty.
    \]
 \end{itemize}
Under (S1) and (S3), we show that Condition iii) of Assumption \ref{ass_1}, follows with the same $\xi$ as for the original functions and with $C_3=3^JC_2$.\\
We begin our proof with the second claim of Assumption \ref{ass_1}, part iii). Let for the moment $M$ be fixed but arbitrary. The random function $X_n$ is by assumption a.s. Hölder continuous. We fix the outcome and call the Hölder constant $C_X$. Now, consider two points $u,v \in I_2$ and the difference $X_{n,M}(u)-X_{n,M}(v)$. The values that $X_{n,M}$ takes in these points depend on the intervals $[0,1/M], [1/M, 2/M],\ldots$ that these points fall into. Hence, we have to consider different cases - namely, both points fall into the same interval, both points fall into adjacent intervals, both points fall into non-adjacent intervals. All cases are treated similarly and we consider the third case, where $u \in [k/M, (k+1)/M]$ and $v \in [\ell/M, (\ell+1)/M]$ with $k+1<\ell$. We now further decompose the difference into
\begin{align*}
& |X_{n,M}(u)-X_{n,M}(v)| \\
\le &|X_{n,M}(u)-X_{n,M}((k+1)/M)|
+|X_{n,M}((k+1)/M)-X_{n,M}(\ell/M)|\\
& +|X_{n,M}(\ell/M)-X_{n,M}(v)|\\
= & |X_{n,M}(u)-X_{n}((k+1)/M)|
+|X_{n}((k+1)/M)-X_{n}(\ell/M)|\\
& +|X_{n}(\ell/M)-X_{n,M}(v)| \\
\le & |X_{n}(k/M)-X_{n}((k+1)/M)|
+|X_{n}((k+1)/M)-X_{n}(\ell/M)|\\
& +|X_{n}(\ell/M)-X_{n}((\ell+1)/M)|\\
\le & C_X|[1/M]^\xi+[(\ell-k-1)/M]^\xi+[1/M]^\xi|
\le 3 C_X  |u-v|^\xi.
\end{align*}
In the first step we have added and subtracted terms, in the second step, we have used that sparse estimators and latent functions are identical in the gridpoints, in the third one, we have used that the sparse estimators are linear interpolations of the latent function between gridpoints, in the fourth one Hölder continuity of $X_n$ and in the last one that $1/M, (\ell-k-1)/M \le |u-v|$. Hence,
\[
    \sup_{n,M} \mathbb{E} \bigg[\Big(\sup_{u \neq v} \frac{|X_{n,M}(u)-X_{n,M}(v)|}{|u-v|^{\xi}}\Big)^J\bigg] <\infty
    \]
and thus the
 second claim of Assumption \ref{ass_1} part iii) follows. For the first claim of the Assumption, notice that in any $u \in [0,1]$ we have a closest gridpoint $u^*=k^*/M$ such that
\[
|X_n(u)-X_{n,M}(u)| \le |X_n(u) -X_n(u^*)|+ |X_{n,M}(u) -X_{n,M}(u^*)|.
\]
Now, using the Hölder continuity of each function in $J$th moment yields
\begin{align*}
   & \sup_n \mathbb{E}\|X_n-X_{n,M}\|^{J} \\
   \le &\frac{C}{M^{J\xi}}\bigg\{\sup_n \mathbb{E} \bigg[\Big(\sup_{u \neq v} \frac{|X_n(u)-X_n(v)|}{|u-v|^{\xi}}\Big)^J\bigg]\\
   &+\sup_{n,M} \mathbb{E} \bigg[\Big(\sup_{u \neq v} \frac{|X_{n,M}(u)-X_{n,M}(v)|}{|u-v|^{\xi}}\Big)^J\bigg] \bigg\}.
\end{align*}
Here we have used that $|u-u^*|\le 1/M$. The right side converges to $0$ polynomially in $N$ for $M=M_N$ and $M_N \ge C N^\beta $ (polynomial growth condition of (S1)).
Thus, the first claim of Assumption \ref{ass_1} part iii) follows.

\textbf{Part iii)} We now turn to the case of kernel density estimation and impose the following additional condition.
\begin{itemize}
    \item[(S4)] Suppose that for some $\zeta>0$, $\sup_n \mathbb{E}\|\nabla^{(b)} X_n\|^{J+\zeta}<\infty$ for $b=0,1,2$. The bandwidth satisfies $h=D_{n,M_N}^{-b}$ for some $0 < b<1/4$. Moreover, for some  $0<\beta_1<\beta_2<\infty$ and $0<C_1<C_2<\infty$ the sample size satisfies the bounds $C_1N^\beta_1<D_{M_N}<C_2N^\beta_2$.
\end{itemize}
If (S1) and (S4) hold, Assumption \ref{ass_1}, part iii) follows with $\xi=1$.
The proof structure is similar to that in part i), but simpler since it is not necessary to construct a modified version of the estimators. As in part i), we need to show polynomial decay of 
\[
\mathbb{E}\|\nabla^{(r)}X_{n,M_N}-\nabla^{(r)}X_n\|^J
\] 
for  $r=0,1$. Again we focus on the case $r=0$, where
\begin{align*}
&\mathbb{E}\|X_{n,M_N}-X_n\|^J \\
\le & C \mathbb{E}\|X_{n,M_N}-\mathbb{E}[X_{n,M_N}|X_n]\|^J+C \mathbb{E}\|\mathbb{E}[X_{n,M_N}|X_n]-X_n\|^J\\
=: & R_1+R_2.
\end{align*}
Now, we can consider each part separately. For $R_2$, we notice that according to the Corollary before Lemma 2 in \cite{bhattacharya:1967} we have
\[
R_2 \le C h^J \mathbb{E}\|\nabla X_n\|^J \le C h^J.
\]
Notice that we have used here the proof of named result to obtain an explicit constant depending on $\nabla X_n$.
Next, we turn to $R_1$, which we split up into
\begin{align*}
    \{R_1\}^{1/J} \le &R_{1,1}+h^{1/2}\\
    R_{1,1}:= & \{\mathbb{E}\|\mathbb{E}[X_{n,M_N}|X_n]-X_{n,M_N}\|^{J+\zeta}\}^{J/(J+\zeta)} \\
    & \cdot \{\mathbb{P}(\|\mathbb{E}[X_{n,M_N}|X_n]-X_{n,M_N}\|>h^{1/2})\}^{\zeta/(J+\zeta)}=:R_{1,1,1}\cdot R_{1,2,2}.
\end{align*}
It is easy to prove a polynomial bound in $N$ for $R_{1,1,1}$
 and we can further upper bound $R_{1,1,2}$ which works as in part i) 
 of this Lemma. More precisely, we can use Lemma 2.2 in \cite{schuster:1969}, 
 which is an exponential concentration result for the kernel density estimator 
 around its mean. This shows an exponential decay of $R_{1,1,2}$ 
 and hence of $R_{1,1}$ in $N$ concluding the proof.

\section{Further information} \label{sec_details}

\subsection{Review of change point analysis for functional data} $ $

As an application of the weak invariance principle developed
in this paper, we have discussed a change point monitoring scheme for time series of multivariate random functions in Sections \ref{sec_mon}. There exists a large literature on inference for change points in general and for functional data in particular. We can hence only refer to some illustrative works, rather than giving a comprehensive literature review. For a concise overview of the change point literature, we refer the interested reader to \cite{horvath:rice:2014}. In functional data, change point detection has usually been developed for time series in $L^2([0,1])$ and a retrospective setup. In \cite{aue:hormann:horvath:reimherr:2009} and \cite{berkes:gabrys:horvath:kokoszka:2009} tests for mean changes were developed that project the data functions on a small number of principal components. Methods without projections were later developed in \cite{horvath:kokoszka:rice:2014} and versions for weakly dependent data (among others) by \cite{aston:kirch:2012} and \cite{aue:rice:sonmez:2018}. Changes in multivariate time series of random functions have been investigated in \cite{gromenko:kokoszka:reimherr:2017}, who developed test statistics under spatial separability of the time series. More recently, \cite{dette:quanz:2023} proposed tests for relevant changes in multivariate functional time series. Changes in other distributional features than the mean, such as the covariance operator or its eigensystems have been studied in works such as \cite{stoehr:aston:kirch:2021,aue:rice:sonmez:2020,dette:kutta}. Changes in functional time series of continuous functions were first investigated by \cite{dette:kokot:aue} and adapted to the study of covariance kernels in \cite{dette:kokot:2022}.

\subsection{On strong mixing} $ $

To specify the dependence structure of our data, we invoke in this work the notion of $\alpha$-mixing. Let us therefore consider a  time series of random variables $(X_n)_{n \in \mathbb{N}}$, with values in a generic Banach space. We say that the time series $(X_n)_{n \in \mathbb{N}}$ is $\alpha$-mixing, if there exists a sequence of non-negative numbers $\alpha(k) \downarrow 0$ such that
\[
  |\mathbb{P}(A \cap B) - \mathbb{P}(A)\mathbb{P}(B)| \le \alpha(k),
\]
for any two events $A \in \sigma(X_i: 1 \le i \le m)$ and $B \in \sigma(X_i: i \ge m+k)$ and any $m,k \ge 1$. Here, $\sigma$ denotes the $\sigma$-algebra generated by the respective random variables. It is a direct consequence of the definition of $\alpha$-mixing, that it is stable under measurable transforms. More precisely, if $X_i$ takes values in the Banach space $\mathcal{X}$ and $g:\mathcal{X} \to \mathcal{Y}$ is a measurable map into another Banach space,  the time series $(g(X_n))_{n \in \mathbb{N}}$ is also $\alpha$-mixing with coefficients $(\alpha_g(n))_{n \in \mathbb{N}}$ satisfying $\alpha_g(n) \le \alpha(n)$ for all $n \ge 1$.
For details on $\alpha$-mixing, as well as helpful covariance inequalities, 
we refer to the monograph of \cite{dehling:mikosch:sorensen:2002}
 and for a survey on  mixing types to \cite{bradley:1986}, vol. 1. 
 In Section \ref{s:cons}, we also use the stronger mixing requirement of 
 $\beta$-mixing. We can define the $\beta$-mixing coefficient as
\[
   \beta(k):= \sup \frac{1}{2} \sum_{i=1}^{I} \sum_{j=1}^J \big|\mathbb{P}\big(A_i \cap B_j\big)-\mathbb{P}\big(A_i\big) \mathbb{P}\big(B_j\big)\big|,
\]
for any finite $I, J$ and events $A_1,\ldots,A_I \in \sigma(X_i: 1 \le i \le m)$ and $B_1,\ldots,B_J \in \sigma(X_i: i \ge m+k)$ that are mutually disjoint. It is well-known that the relation $2\alpha(k) \le \beta(k)$ holds and hence $\beta$-mixing is stronger than $\alpha$-mixing (see \cite{bradley:1986}). Originally, $\beta$-mixing was introduced in a paper by \cite{volkonskii:1959} and in the literature it is sometimes referred to as "absolute regularity". $\beta$-mixing plays an important role in the construction of coupled processes, as has been recognized, e.g., in the famous contribution of  \cite{berbee:1979}.

\end{document}